\documentclass[11pt,a4paper]{article}
\usepackage[left=2.8cm, right=2.8cm, top=3cm, bottom=3cm]{geometry}
\usepackage{graphicx}
\usepackage[T1]{fontenc}
\usepackage[utf8]{inputenc}
\usepackage[overload]{empheq}
\usepackage{setspace}
\usepackage{amsmath}
\usepackage{amssymb}
\usepackage{amsfonts}
\usepackage{relsize}
\usepackage{bm}
\usepackage[dvipsnames]{xcolor}
\usepackage{enumitem}
\usepackage{url}
\usepackage{pdfpages}
\usepackage{palatino}
\usepackage{wrapfig}

\makeatletter
\def\patchsect#1\let\@svsec\@empty{#1\def\@svsec{\leavevmode\kern1sp\relax}}
\let\old@sect\@sect
\def\@sect{\expandafter\patchsect\old@sect}
\makeatother

\usepackage{amsthm}
\theoremstyle{definition}

\theoremstyle{plain}
\newtheorem{defi}{Definition}[section]
\newtheorem{lemma}[defi]{Lemma}
\newtheorem{theorem}[defi]{Theorem}
\newtheorem{prop}[defi]{Proposition}

\newtheorem{corollary}[defi]{Corollary}
\newtheorem{remark}[defi]{Remark}

\usepackage{titlesec}
\titleformat{\section}
  {\centering\normalfont\large\bfseries}{\thesection.}{1em}{}
\titleformat{\subsection}
  {\normalfont\large\bfseries}{\thesubsection.}{1em}{}
\titleformat{\subsubsection}
  {\normalfont\normalsize\itshape}{\thesubsubsection.}{1em}{}

\usepackage{hyperref}
 \hypersetup{
    colorlinks,%
    citecolor=red,%
    filecolor=black,%
    linkcolor=black,%
    urlcolor=black
}

\usepackage{float}

\usepackage{subcaption}
\usepackage{icomma}
\usepackage{appendix}
\usepackage[backend=bibtex, backref=true, style=numeric,maxbibnames=99, url=false, isbn=false]{biblatex}
\DeclareMathOperator*{\Id}{Id}
\DeclareMathOperator*{\Tr}{Tr}
\newcommand{\jap}[1]{\langle #1 \rangle}
\newcommand{\bd}[1]{\mathbf{#1}}

\title{Propagation of chaos for the Landau equation with very soft and Coulomb potentials}
\date{\today}
\author{Côme Tabary\footnote{DMA, ENS, Université PSL, CNRS, 75005 Paris, France, and Université Paris Cité and Sorbonne Université, CNRS, IMJ-PRG, F-75013 Paris, France}}
\begin{document}

\maketitle

\begin{abstract}
\noindent We consider a drift-diffusion process of $N$ stochastic particles and show that its empirical measure converges, as $N\rightarrow\infty$, to the solution of the Landau equation. We work in the regime of very soft and Coulomb potentials using a tightness/uniqueness method. To claim uniqueness, we need high integrability estimates that we obtain by crucially exploiting the dissipation of the Fisher information at the level of the particle system. To be able to exploit these estimates as $N\rightarrow\infty$, we prove the affinity in infinite dimension of the entropy production and Fisher information dissipation (and other first and second-order versions of the Fisher information through a general theorem), results which were up to now only known for the entropy and the usual Fisher information.
\end{abstract}

\setcounter{tocdepth}{1}
\tableofcontents

\section{Introduction}

\subsection{Background}

The homogeneous Landau equation in three space dimensions writes
\begin{equation}
\label{eq:landau}
    \partial_t f_t(v) = \nabla_v \cdot \int_{\mathbb{R}^3} \alpha(\vert v-w \vert) a (v - w) (\nabla_v - \nabla_w) f_t(v)f_t(w) dw,
\end{equation}
where the unknown $f_t:\mathbb{R}^3\rightarrow\mathbb{R}_+$ is the time-dependent distribution of velocities in a plasma.
The matrix-valued function
$$a(z):=\vert z \vert^2 \Id - z \otimes z$$
is the projection on $z^{\perp}$ up to a $\vert z \vert^2$ factor. The non-negative function $\alpha$ is called the interaction potential, and depends on the nature of the interactions between the plasma particles. In this work we will consider power-law interaction potentials 
$$\alpha(r)=r^\gamma$$
and will focus on so-called \textit{very soft potentials} $\gamma\in[-3,-2]$. The case $\gamma=-3$ corresponds to the most physically relevant Coulomb interactions.

The Landau equation has been widely
used in plasma physics as it is derived from the Boltzmann equation in the limit of \textit{grazing collisions} (that is to say, when very few momentum is exchanged during a single interaction between two particles). Its statistical nature makes it a description of the plasma at a \textit{mesoscopic scale}, bridging the microscopic scale of individual particles and the macroscopic scale of fluid dynamics.
For an in-depth review of these kinetic models, we refer to \cite{LifsicPitaevskij2008} for a physicist's approach and to \cite{Villani2002} for a mathematician's.

The rigorous derivation of time-irreversible kinetic equations from a microscopic reversible Newtonian particle system is a deep question initiated by Boltzmann himself in his 1872 work \cite{Boltzmann2003a}. Showing that, as the number of particles becomes infinite, they can be statistically described by the solution to a kinetic equation is called \textit{propagation of chaos}. This name refers to the fact that particles are necessarily correlated in the finite system because they collide with each other, but in the limit they become independent, or \textit{chaotic}. Indeed, the Boltzmann and Landau equations intrinsically postulate that before any collision, two particles have "forgotten" their past and are uncorrelated: this is the hypothesis of "molecular chaos". A testimony of this is the tensor product $f_t(v)f_t(w)$ in~\eqref{eq:landau}, which is the distribution of velocities of two \textit{independent} particles. For a long time Lanford's result \cite{Lanford1975} was the best but still partial answer to this complex question, at least until the recent breakthrough \cite{DengHaniMa2024}.

In 1956, Kac \cite{Kac1956} proposed the simpler goal of obtaining the space-homogeneous Boltzmann equation from a Markov $N$-particle jump process by performing a mean field limit, rather than starting from Newtonian dynamics. He studied a one-dimensional toy model, for which he proved the first result of propagation of chaos, towards the associated toy equivalent of the Boltzmann equation. These results were later extended to higher dimensions and more physical cases, we refer in particular to \cite{MischlerMouhot2012} and the references therein.

The same goal can be set for the space-homogeneous Landau equation~\eqref{eq:landau}: can the Landau equation be obtained as a mean-field limit from a suitable stochastic particle system? 
The answer has been proven positive for a variety of interaction potentials $\alpha$ (and slightly different particle systems to start with), that we will review in more detail below. To summarize, propagation of chaos was shown for the case of moderately soft to hard potentials $\alpha(r)=r^\gamma, \gamma\in(-2,1)$ \cite{FontbonaGuerinMeleard2007, Carrapatoso2015, FournierHauray2015, NicolasFournierArnaudGuillin2017, CarrilloFengGuo2024} and partial results of convergence towards the BBGKY-Landau hierarchy were proven in the Coulomb case \cite{MiotPulvirentiSaffirio2011, CarrilloGuo2025}. Very soft and Coulomb potentials $\gamma\in[-3,-2]$ were, to our knowledge, never covered until the very recent result of Feng and Wang \cite{FengWang2025}. The present work was pursued in parallel to \cite{FengWang2025}, with a very different approach, differences that we discuss in more detail below.

\subsection{Main results}
\label{ssec:mainresult}

\subsubsection{Propagation of chaos}

The main result of this paper is to show propagation of chaos for the Landau equation with very soft (or Coulomb) potentials $\gamma\in[-3,-2]$. We will start from a regularized version of the particle system considered in \cite{MiotPulvirentiSaffirio2011, Carrapatoso2015} (and also studied in \cite{NicolasFournierArnaudGuillin2017}), to which we refer for details. Consider a $N$-particle jump process where each pair $(i,j)$ of particles with velocities $V^N_i$ and $V^N_j$ overcomes elastic collisions with a rate $\vert V^N_i - V^N_j\vert^\gamma$: this system correponds to Kac's original proposition to approximate the Boltzmann equation. One can pass this system to the grazing collision limit and (formally) obtain a diffusion-drift $N$-particle system approaching the Landau equation.

This resulting system is not clearly well-posed for $\gamma<0$ because of the singularity at the origin of the interaction potential $\alpha$. As done in \cite{FournierHauray2015} for moderately soft potentials, we regularize $\alpha$ for each number $N$ of particles and let the regularization vanish as $N\rightarrow+\infty$. We emphasize that since the regularization can be taken arbitrarily close to the original $\alpha$ and can disappear arbitrarily fast, this does not affect the physical meaning of the system.

We consider a pointwise non-decreasing sequence of positive, smooth and bounded functions $(\alpha^N)_{N\geq 2}$, such that $\alpha^N \nearrow \alpha$, and such that $\alpha=\alpha^N$ outside the ball $B(0,\eta_N)$ for some $\eta_N \xrightarrow[N\rightarrow\infty]{} 0$. We also suppose that for each $N\geq 2$, for any $r>0$,
\begin{equation}
\label{eq:fishercond}
    \frac{r \vert (\alpha^{N})'(r) \vert}{\alpha^{N}(r)}\leq \frac{-\gamma}{0.99}.
\end{equation}
Note that this condition is automatically satisfied for $r> \eta_N$ because the left-hand side equals $-\gamma$ for $\alpha^N(r)=r^\gamma$. We will give an example of such a regularization in Section~\ref{sec:regularized particle system}. (We fixed the value $0.99$ for nicer explicit constants in estimates but it can be lowered to anything strictly greater than $(-\gamma)/\sqrt{22}$, see Remark~\ref{rem:on0.99}.)

Using these potentials we can present the particle system we will start from. We define the drift and the diffusion
\begin{align*}
b^N:=\nabla\cdot(\alpha^Na)=\sum_{k=1}^3\partial_k(\alpha^Na_{kl})=-2\alpha^N v,&&\sigma^N:=\left(\alpha^N a\right)^{1/2}=\sqrt{\alpha^N} \vert v\vert^{-1} a.
\end{align*}We consider a family of independent Brownian motions indexed by the ordered pairs of particles, $(B^{ij})_{1 \leq i<j \leq N}$, and let $B^{ji} = - B^{ij}$. The regularized particle system is the following system of stochastic differential equations for the $\mathbb{R}^3$-valued velocities $(V^N_i)_{1\leq i \leq N} $:
\begin{align}
    \label{eq:particlesystem}
    V^N_i (t) = V^N_{i,0} &+ \frac{2}{N-1}  \sum_{\substack{j=1 \\ j \neq i}}^N \int_0^t b^N(V^N_i(s) - V^N_j(s)) ds\nonumber
    \\&+ \frac{\sqrt{2}}{\sqrt{N-1}} \sum_{\substack{j=1 \\ j \neq i}}^N \int_0^t\sigma^N(V^N_i(s) - V^N_j(s)) dB^{N}_{ij}(s).
\end{align}
Here, the initial data $(V^N_{i,0})_{1\leq i \leq N}$ is a random variable independent of the Brownian motions. We suppose that initial chaos holds, \textit{i.e.} that the law of $(V^N_{i,0})_{1\leq i \leq N}$ is given by the tensor product $g_0^{\otimes N}$ for some initial distribution $g_0$.
\begin{remark}
    The original particle system derived in \cite{MiotPulvirentiSaffirio2011, Carrapatoso2015} is exactly~\eqref{eq:particlesystem} but defined with $\alpha$ instead of $\alpha^N$ in the drift and diffusion. Its coefficients are hence singular at the origin, so its well-posedness is a potentially difficult issue that is not the topic of this work.
\end{remark}

To state our result, we define the random empirical measure $(\mu^N_t)_{t\in[0,T]}\in C([0,T],\mathcal{P}(\mathbb{R}^3))$, the space of probability-valued continuous functions on $[0,T]$, by letting
$$\mu^N_t:=\frac{1}{N}\sum_{i=1}^N \delta_{V^N_i(t)}.$$
The space $\mathcal{P}(\mathbb{R}^3)$ of probabilities over $\mathbb{R}^3$ will always be endowed with the (metrizable) topology of weak convergence (against bounded continuous functions) and $C([0,T],\mathcal{P}(\mathbb{R}^3))$ is endowed with the uniform topology.
Our main result is the following (see below for the classical definition of Fisher information and entropy):
\begin{theorem}
    \label{thm:main}
    Let $ \gamma\in[-3,-2]$. Let $g_0\in L^1(\mathbb{R}^3)$ have unit mass, finite energy $E_0$, finite entropy $H_0$, finite Fisher information $I_0$, and finite moment of order $m$
    \begin{equation*}
    \int_{\mathbb{R}^3} \vert v\vert^m g_0 dv <+\infty
    \end{equation*}
    for some $m> 3(2-\gamma)$ ; and let $(g_t)_{t\geq 0}$ be the classical solution to the Landau equation~\eqref{eq:landau} with initial condition $g_0$. 
    Then, for any $T>0$, the random empirical measure converges in probability in $ C([0,T],\mathcal{P}(\mathbb{R}^3))$ to the deterministic $(g_t)_{t\in[0,T]}$:
    \begin{equation*}
    \label{eq:maingoal}
        (\mu^N_t)_{t\in[0,T]} \xrightarrow[N\rightarrow \infty]{\mathbb{P}}(g_t)_{t\in[0,T]}.
    \end{equation*}
    
\end{theorem}
\begin{remark}
    In view of the topology we endowed $ C([0,T],\mathcal{P}(\mathbb{R}^3))$ with, the convergence in probability in the above theorem means that, for any $\varepsilon>0$,
    $$\mathbb{P}\left( \sup_{t\in[0,T]} d\left(\mu^N_t,g_t\right)>\varepsilon\right)\xrightarrow[N\rightarrow \infty]{}0,$$
    where $d$ is any metric for the weak topology on $\mathcal{P}(\mathbb{R}^3)$ (for instance, a Wasserstein metric associated to a bounded distance on $\mathbb{R}^3$, see \cite[Theorem 6.9, Corollary 6.13]{Villani2009}).
\end{remark}
\begin{remark}
\label{rem:postthm}
    The hypotheses on $g_0$ are stronger than those of \cite[Theorem 1.5]{Ji2024}, that guarantees the existence of a classical solution to the Landau equation with constant energy and non-increasing entropy (and non-increasing Fisher information, but we will not use it directly). It will likewise be shown in Proposition~\ref{prop:wellposednessparticlesystem} that~\eqref{eq:particlesystem} admits a pathwise unique solution defined for all positive times, so every object in the above theorem makes sense.
\end{remark}
This formulation, which is natural in the setting of our proof, might not be the most common way to state propagation of chaos. It is indeed usually expressed in terms of the law of the particle system and its marginals. Let $F^N_t$ be the law of $(V^N_i(t))_{1\leq i \leq N}$, and for any $1\leq j \leq N$, let $F^{N:j}_t$ be its $j$-th marginal, \textit{i.e.}
$$F^{N:j}_t=\int_{\mathbb{R}^{3(N-j)}} F^{N}_t dv_{j+1}...dv_N.$$
It actually does not matter over which variables we integrate, because of the exchangeability of particles. Then Theorem~\ref{thm:main} entails:

\begin{corollary}
\label{cor:main}
    Under the same hypothesis on $g_0$ as in Theorem~\ref{thm:main}, for any $t\geq 0$, and any $j\in\mathbb{N}^*$, the weak convergence 
    $$F^{N:j}_t \xrightharpoonup[N\rightarrow\infty]{} g_t^{\otimes j}$$
    holds in $\mathcal{P}(\mathbb{R}^{3j})$.
\end{corollary}

In this form the term \textit{propagation of chaos} can be better understood: note that our hypothesis on $(V^N_{i,0})_{1\leq i \leq N}$ is exactly $F^{N}_0=g_0^{\otimes N}$, so that we suppose molecular chaos (independence of the particles) at initial time, and we show that this chaos is \textit{propagated} to later times in the large $N$ limit. For more details on the different statements of chaos and quantitative equivalences between them, we refer to \cite{HaurayMischler2012}.

\subsubsection{Generalized Fisher information}

Our proof of propagation of chaos naturally leads us to study variations of the Fisher information and extend its well-known properties to these more general versions. We state these results as a separate side theorem since they are in our opinion of independent interest. It also allows us to invoke them as a black box to simplify the presentation of the main result, since their proof is quite orthogonal to the remainder of this work.

Consider, for some integer $k_0\geq 1$, $b=b(v_1,...,v_{k_0})\in\mathbb{R}^{3k_0}$ a smooth, bounded, divergence-free vector field, and $\delta=b\cdot\nabla_{1...k_0}$ the associated derivation operator acting on the $(v_1,...,v_{k_0})$ variables. We consider the following generalized Fisher information: for any symmetric non-negative function $F$ on $\mathbb{R}^{3N}$ for any $N\geq k_0$, 
$$I^\partial(F):=\int_{\mathbb{R}^{3N}}\frac{\vert \partial F\vert^2}{F} dv_1...dv_N.$$
By symmetry of $F$, the usual Fisher information $I$ on $\mathbb{R}^{3N}$ (normalized by $\frac{1}{N}$) corresponds to $$I(F):=\frac{1}{N}\int_{\mathbb{R}^{3N}}\frac{\vert \nabla_{1...N} \partial F\vert^2}{F} dv_1...dv_N=I^{e_1}(F)+I^{e_2}(F)+I^{e_3}(F)$$ where $e_i$ is the canonical basis of $\mathbb{R}^3$ (hence we can take $k_0=1$ in this case). We can also write the entropy production functional of the Landau equation in this framework if we allow for a singular $b$, see Section~\ref{sec:infinite-dimension limit of the regularity estimates}.

We will also study the generalized second order Fisher information
$$K^\partial(F):=\int_{\mathbb{R}^{3N}} \frac{( \partial^2 F)^2}{F}.$$
This study is motivated by the fact that the dissipation of the Fisher information along the flow of the Landau equation is connected to such a functional.

Then, the following properties hold for $I^\partial$ and $K^\partial$:
\begin{theorem}
\label{thm:side}
Let $b=b(v_1,...,v_{k_0})\in\mathbb{R}^{3k_0}$ for some $k_0\geq 1$ be a smooth, bounded, divergence-free vector field, and $\delta=b\cdot\nabla_{1...k_0}$. Then for all $N\geq k_0$,
\begin{enumerate}
    \item \textbf{Lower semi-continuity and convexity:} The functionals $I^\partial$ and $K^\partial$ can be extended as non-negative, possibly infinite functionals to the space $\mathcal{P}(\mathbb{R}^{3N})$ of probabilites on $\mathbb{R}^{3N}$. They are convex and lower semi-continuous with respect to weak convergence over that space.
    \item \textbf{Super-additivity}: if $F\in\mathcal{P}(\mathbb{R}^{3N})$ and $$F^{:j}=\int_{\mathbb{R}^{3(N-j)}} F dv_{j+1}...dv_N,$$
    is its $j$-th marginal for any $k_0\leq j\leq N$, then 
    \begin{align*}
        I^\partial(F^{:j})\leq  I^\partial(F) && K^\partial(F^{:j})\leq  K^\partial(F).
    \end{align*}
    \item \textbf{Infinite-dimensional affinity}: Let $\nu\in\mathcal{P}(\mathcal{P}(\mathbb{R}^{3}))$ and consider the marginals
    $$\nu^j:=\int_{\mathcal{P}(\mathbb{R}^3)} \rho^{\otimes j} \nu(d\rho).$$
    Suppose $\nu$ has finite mean Fisher information and finite mean second moment, i.e.
    \begin{align*}
        \sup_{j\geq 1} I(\nu^j) < +\infty, &&\int_{\mathbb{R}^3} \vert v \vert^2 \nu^1(dv) < +\infty.
    \end{align*}
    Then it holds that
    \begin{align*}
        \int_{\mathcal{P}(\mathbb{R}^3)} I^\partial(\rho^{\otimes k_0}) \nu(d\rho) &=\lim_{j\geq k_0} I^\partial(\nu^j) \\\lim_{j\geq k_0} K^\partial(\nu^j)\leq \int_{\mathcal{P}(\mathbb{R}^3)} K^\partial(\rho^{\otimes k_0}) \nu(d\rho) &\leq 16 \lim_{j\geq k_0} K^\partial(\nu^j), 
    \end{align*}
    and all the sequences of which we take the limits are increasing in $j$.
\end{enumerate}
\end{theorem}
All these properties are known to hold for the Fisher information and the entropy. Their connection with propagation of chaos is further discussed in the subsection below. The hypotheses of point (3) can certainly be relaxed (especially the moment assumption) but they are quite natural in practice. The proofs of points (1) and (2) are quit short and close to the Fisher information case, they are done as a preliminary in Section~\ref{sec:on the fisher dissipation and aux} along a comparison between alternate definitions of the second order Fisher information that will be useful later. The affinity (3) is much more complex and is proven in Section~\ref{sec:infdimlimofDandK}, adapting the Hilbert framework developed in \cite{Rougerie2020}.
\begin{remark}
\label{rem:lvl3energy}
    We call the moment in point (3) above \textbf{mean second moment} because if $f$ is a $\nu$-distributed random variable, then
    $$\mathbb{E}\left[\int \vert v \vert^2 f(dv)\right]=\int \vert v \vert^2 \mathbb{E}[f](dv)=\int \vert v \vert^2 \int_{\mathcal{P}(\mathbb{R}^3)} \rho \nu(d\rho) dv=\int \vert v \vert^2 \nu^1_t (dv).$$
    More generally, we can refer to integrals against $\nu$ as mean, or statistical, quantities related to $\nu$. They are also sometimes called \textbf{level-3} quantities in the literature.
\end{remark}

\subsection{Literature review and discussion}

Propagation of chaos is an expansive topic, reaching beyond kinetic theory. For a general review, we refer to \cite{ChaintronDiez2022a, ChaintronDiez2022, HaurayMischler2012, Sznitman1991}. Below we will focus on the case of the Landau equation which already enjoys a wide literature.
\\

\textbf{On particle systems.} 
As stated above, the particle system we consider in this work was derived by Miot, Pulvirenti, Saffirio, and Carrapatoso \cite{MiotPulvirentiSaffirio2011, Carrapatoso2015}, as a grazing collision limit of a jump process approximating the Boltzmann equation. (Actually, in both works the grazing collision limit is performed on the Kolmogorov equation for the \textit{law} of the process, and in \cite{Carrapatoso2015} proposes an actual stochastic system such that its law satisfies the limit equation.) An important consequence is that, as does the jump process, the drift diffusion system almost surely conserves momentum and energy: $\sum_i V^N_i$ and $\sum_i (V^N_i)^2$ are constant in time. This means that the process is confined on the corresponding Boltzmann sphere (the sphere of dimension $3(N-1)-1$ formed by the points of $\mathbb{R}^{3N}$ of constant given momentum and energy), so there is no diffusion in the orthogonal directions. These constraints also mean that approximate independence of the particles is harder to achieve. Propagation of chaos for this conservative system in the case of Maxwellian molecules $\gamma=0$ was shown in \cite{Carrapatoso2015} adapting the abstract framework developed in \cite{MischlerMouhot2012}, and Fournier and Guillin covered the case of hard potentials $\gamma\in[0,1)$ in \cite{NicolasFournierArnaudGuillin2017} with a probabilistic coupling method. These methods both provide (polynomial in $N$) rates of convergence.

There exists a slightly easier to study but slightly less natural particle system, that does not conserve the energy. It is obtained by considering~\eqref{eq:particlesystem} with $N^2$ independent Brownian motions rather than setting $B^{ij}=-B^{ji}$. Propagation of chaos towards the Landau equation was shown by Fontbona, Guerin, Méléard \cite{FontbonaGuerinMeleard2007} for Maxwell molecules. This system is also considered by Fournier and Hauray \cite{FournierHauray2015} in the only complete result of propagation of chaos with soft potentials we are aware of: they cover moderately soft potentials $\gamma\in(-2,0)$.
They also regularize the particle system (by convolution) to make it well-posed. We believe the difference in regularization has little influence on the system since both are arbitrarily small (in fact, their regularization could probably fit in our method but some proofs and inequalities would be less straightforward). The case $\gamma\in(-1,0)$ is treated by a fine probabilistic coupling (providing convergence with a polynomial rate in $N$), and the case $\gamma\in(-2,-1)$ is by a tightness-uniqueness argument. We also mention a variation of this non-conservative particle system, obtained with $N$ independent Brownian instead of $N^2$, considered in \cite{CarrilloFengGuo2024, CarrilloGuoJabin2024}.

These non-conservative systems still preserve the expected energy $\mathbb{E}(\sum_i (V^N_i)^2)$ and, at first glance, their difference with~\eqref{eq:particlesystem} seems light. However we point out that their laws obey a different Kolmogorov equation, for which it is not clear that the Fisher information is non-increasing, as it will be the case with the Kolmogorov equation of~\eqref{eq:particlesystem}. We believe this makes the conservative system~\eqref{eq:particlesystem} even more preferable to the non-conservative ones, as the Landau equation is now known to dissipate the Fisher information \cite{GuillenSilvestre2023}.
\\

\textbf{On the Kolmogorov/Master equation.}
Although we did not need it to state Theorem~\ref{thm:main}, the law $F^N_t$ of the random velocities $(V^N_i(t))_{1\leq i\leq N}$ plays a central role in both parts of the tightness/uniqueness method we use. First, showing tightness of the empirical measures $\mu^N$ requires estimates on their laws and hence on $F^N_t$. Second, any regularity of the cluster points of $(\mu^N)_{N\geq 2}$ is obtained by passing to the limit in the regularity of $F^N_t$ (once expressed in adequate quantities that behave well as the number $N$ of particles, and hence the dimension, goes to infinity: this is discussed below).

The law $F^N_t=F^N_t(v_1,...v_N)$ obeys the following Kolmogorov (also known as \textit{Master}) equation (in a weak form) (see \cite{Carrapatoso2015,CarrilloGuo2025}, and Appendix~\ref{sec:derivkol} for a derivation):
\begin{equation}
\label{eq:kolintro}
    \partial_t F^N_t =\frac{1}{(N-1)} \sum_{\substack{i, j = 1 \\ i< j}}^N Q^N_{ij}(F^N_t),
\end{equation}
with 
$$Q^N_{ij}(F^N_t)=(\nabla_i - \nabla_j) \cdot \left[ \alpha^N a(v_i - v_j) (\nabla_i - \nabla_j) F^N_t\right],$$
$\nabla_i$ denoting the nabla with respect to $v_i$. It is possible to forget the particle system itself and only work on its law to show propagation of chaos: it amounts to showing Corollary~\ref{cor:main} directly, without Theorem~\ref{thm:main}. This is the approach of \cite{Carrapatoso2015, CarrilloFengGuo2024, CarrilloGuo2025,  MiotPulvirentiSaffirio2011}.

The operator $Q^N_{ij}$ is the same that appears in the landmark proof of the monotonicity of the Fisher information by Guillen and Silvestre \cite{GuillenSilvestre2023} (up to $\alpha^N$ being replaced by $\alpha$). Therein, the question for the nonlinear Landau equation~\eqref{eq:landau} is reduced to the monotonicity for the "lifted" linear equation $\partial_t F(v_1,v_2)=Q_{12}(F)(v_1,v_2)$ on $\mathbb{R}^6$, which is nothing but~\eqref{eq:kolintro} for $N=2$. From there it is easy to see that the Fisher information does not increase along the flow of the general $N$-particle Kolmogorov equation because it decreases along each $Q^N_{ij}$. This was first observed by Carrillo and Guo in \cite{CarrilloGuo2025} in the Coulomb case $\gamma=-3$, where the well-posedness of~\eqref{eq:kolintro} without the regularization of $\alpha$ is studied. The authors use the (uniform in $N$) control provided by the Fisher information to show tightness of the sequences of $j$-marginals $(F^{N:j}_t)_N$ for any $j$. Suitable subsequences in $N$ converge to some limit $G^j_t$ which satisfies a BBGKY-type infinite hierarchy of equations: $G^j_t$ obeys an equation depending on $G^{j+1}_t$. Because the tensorization $(g_t^{\otimes j})_j$ of any solution of the Landau equation is a solution to the hierarchy, uniqueness of solutions to the hierarchy would imply the propagation of chaos $F_t^{N:j}\rightharpoonup g_t^{\otimes j}$. Unfortunately, it is not reached in \cite{CarrilloGuo2025}. We also mention the earlier work \cite{MiotPulvirentiSaffirio2011}, also in the Coulomb case, where~\eqref{eq:kolintro} is further regularized with a Laplacian (vanishing as $N\rightarrow +\infty$), and a similar convergence to a hierarchy is shown (in a weaker sense).

\begin{remark}
Since the divergence with respect to $v_2$ vanishes in the integral, the Landau equation~\eqref{eq:landau} can be rewritten
$$\partial_t f_t(v_1)=\int_{\mathbb{R}^3}Q_{12}(f(v_1)f(v_2))dv_2,$$
which highlights the formal proximity between the particle system and the Landau equation. If we suppose that $F^N_t$ is tensorized at all times, $F^N_t=f_t^{\otimes N}$, then integrating~\eqref{eq:kolintro} over $(v_2,...,v_N)$ and using symmetry on the remaining terms we see that $f_t$ must solve the Landau equation.
\end{remark}

\begin{remark}
The equation~\eqref{eq:kolintro} is a degenerate parabolic equation. As the particle system is confined to Boltzmann spheres of constant impulsion and energy, the diffusion in~\eqref{eq:kolintro} only occurs in the directions tangent to these spheres. When restrained to the unit Boltzmann sphere $\{(v_1,...v_N)\in\mathbb{R}^{3N} \vert \sum_i v_i = 0, \sum_i \vert v_i \vert^2 =1\}$, it reduces for $N=2$ to the heat equation with the Laplace Beltrami operator on the sphere and a diffusion coefficient proportional to the interaction potential $\alpha^N$. In contrast, in higher dimension it still degenerates at certain points where the velocities are aligned. The structure of the equation~\eqref{eq:kolintro} might not be completely understood.
\end{remark}

\textbf{Hierarchy versus probabilistic point of view.}
We believe uniqueness to the BBGKY-type Landau hierarchy mentioned in the above paragraph, which has yet to be proven, is strictly stronger than the uniqueness part of our tightness/uniqueness method. Working at the level of the Landau hierarchy as in \cite{CarrilloGuo2025, MiotPulvirentiSaffirio2011} amounts to showing tightness of $(F^{N:j}_t)_N$ for each time $t$ and extract a limit $G^j_t \in \mathcal{P}(\mathbb{R}^{3j})$ to build a hierarchy $(G^j_t)_j$, solution to the infinite BBGKY-Landau system. To such a hierarchy can be associated a unique element $G_t\in\mathcal{P}(\mathcal{P}(\mathbb{R}^{3}))$ by the de Finetti-Hewitt-Savage theorem, such that
$G^j_t = \int \rho^{\otimes j} G_t(d\rho)$. The limit object is thus an element $(G_t)_{t\in[0,T]}\in C\left([0,T],\mathcal{P}(\mathcal{P}(\mathbb{R}^{3}))\right)$, and the uniqueness part of the tightness/uniqueness method in the setting of the BBGKY-Landau hierarchy would amount to show that $G_t=\delta_{g_t}$ (where $g_t\in\mathcal{P}(\mathbb{R}^3)$ is the solution to the Landau equation at time $t$).

In the present work, however, we show tightness at the level of the particle system, more precisely at the level of the law of empirical measures $(\mu^N_t)_{t\in[0,T]}$. Our limit object $\pi$ hence lives in the space $\mathcal{P}\left(C([0,T],\mathcal{P}(\mathbb{R}^3))\right)$, which can be continuously mapped to $C\left([0,T],\mathcal{P}(\mathcal{P}(\mathbb{R}^{3}))\right)$ by considering the time marginals $(\pi_t)_{t\in [0,T]}$ ($\pi_t$ is the law of $f_t$ if $f=(f_t)_{t\in [0,T]}$ has law $\pi$). The converse cannot be done, since a probability is in general not characterized by its marginals. This means $\pi$ contains more information than $(G_t)_{t\in[0,T]}$, so the uniqueness of $\pi$ we show in this work, that is to say $\pi=\delta_{(g_t)_t}$, is less general than the uniqueness in the hierarchy approach. In some sense, we trade a stronger tightness against an easier uniqueness. 

It is highly probable that the hierarchy $(\pi^j_t)_j$ solves the same BBGKY-Landau equations satisfied by the $(G^j_t)_j$. If so, uniqueness for the hierarchy would entail $\pi_t=\delta_{g_t}$. In this deterministic case knowledge of the marginals is sufficient to claim $\pi=\delta_{(g_t)_t}$, so uniqueness for the hierarchy would indeed be stronger. We did not investigate this point further.
\\

\textbf{Uniqueness in the Landau equation.} In the tightness/uniqueness method, once tightness is proven, we suppose that (up to a subsequence) $(\mu^N_t)_t \rightharpoonup (f_t)_t$ for some $C([0,T],\mathcal{P}(\mathbb{R}^3))$-valued random variable $f=(f_t)_t$ and want to conclude that $(f_t)_t=(g_t)_t$, the classical solution to the Landau equation. To do so, we first show that $f$ solves the Landau equation in some suitable weak sense, and then claim uniqueness of such weak solutions. Uniqueness does not hold without some regularity on $f$: we rely on Guérin and Fournier's results \cite{GuerinFournier2008, Fournier2010} which requires weak solutions to lie in $L^1([0,T],L^p(\mathbb{R}^3))$ for $p>3/(3+\gamma)$ and $L^1([0,T],L^\infty(\mathbb{R}^3))$ for $\gamma=-3$. We will show that for any value of $\gamma$, $f\in L^1([0,T],L^\infty(\mathbb{R}^3))$ almost surely.

Other results exist in the literature such as the Prodi-Serrin-type criteria \cite{AlonsoBaglandDesvillettes2024, ChernGualdani2022}. A uniqueness property for H-solutions (in the sense of Villani \cite{Villani1998}) without any integrability assumptions was proven by the author in \cite{Tabary2025}. However all these results rely on the \textit{H-theorem}: $f$ is supposed to be a weak solution (or H-solution) with the additional assumption that the formal a priori estimate on the entropy given by the Landau equation~\eqref{eq:landau} holds. We stress that such a H-theorem is not automatic for weak solutions, so checking that $f$ lies in the appropriate functional space would not be enough.
\\

\textbf{Behavior of functionals in the infinite dimension limit.} We detail what we consider both one of the most crucial and most original parts of this work: the extension of typical properties of the entropy and Fisher information to more general and higher-order functionals. As described in the previous paragraph, to claim uniqueness we need to recover some regularity or integrability of the limit $f_t$ of the empirical measures $\mu^N_t$. Because empirical measures enjoy virtually no regularity, we have little hope of getting good estimates directly from this limit. What is common in similar proofs of propagation of chaos \cite{FournierHauray2015,Salem2019,Salem2019a} is to recover estimates on the expectation $\mathbb{E}(\mathcal{F}(f_t))$ for some norm or regularity-quantifying functional $\mathcal{F}$, by using uniform-in-$N$ estimates on $\mathcal{F}(F^N_t)$. Indeed, if we can show the finiteness of $\mathbb{E}(\mathcal{F}(f_t))$, then we know that, almost surely, $\mathcal{F}(f_t)$ is finite. Provided that the functional $\mathcal{F}$ controls enough integrability, we can conclude that the limit solution $f_t$ to the Landau equation is unique.

Because the dimension increases as $N\rightarrow\infty$, not every functional can be easily passed to the limit (think for example of the $L^p$ norm of a tensor product $f^{\otimes N}$, which scales as $\Vert f^{\otimes N} \Vert_{L^p}=\Vert f \Vert_{L^p}^{N}$: we have little hope of recovering much as $N\rightarrow+\infty$). Two better-suited candidates are the entropy $H$ and the Fisher information $I$ of a probability measure $F^N$ on $\mathbb{R}^{3N}$:
\begin{align}
\label{eq:defHandI}
    H(F^N):= \frac{1}{N}\int_{\mathbb{R}^{3N}} F^N \log F^N dV,
    &&I(F^N):= \frac{1}{N}\int_{\mathbb{R}^{3N}} \frac{\vert \nabla F^N \vert^2}{F^N} dV,
\end{align}
taking the value $+\infty$ if $F^N$ has no density (or its density has no gradient in the case of $I$). Both $H$ and $F$ are convex and lower-semicontinuous for the weak convergence of probabilities, which is very convenient in our setting. Their nice behavior comes from the following \textit{super-additivity} property: for any symmetric probability $F^N$ on $\mathbb{R}^{3N}$, and any $1\leq j \leq N$, if $F^{N:j}$ is the $j$-th marginal of $F^N$,
\begin{align}
\label{eq:introsuperadditivity}
    H(F^{N:j})\leq  H(F^N), && I(F^{N:j})\leq  I(F^N)
\end{align}
with equality occurring on tensor products
\begin{align}
    H(f^{\otimes N})=H(f), && I(f^{\otimes N})=I(f).
\end{align}
The above scaling on tensor products suggests that these functionals behave better than Lebesgue or Sobolev norms with respect to the dimension.

As seen before, the natural limit space for probabilities on $\mathbb{R}^{3N}$ as $N\rightarrow\infty$ is the probability space $\mathcal{P}(\mathcal{P}(\mathbb{R}^{3}))$. Given $\nu \in \mathcal{P}(\mathcal{P}(\mathbb{R}^3))$, we can consider the sequence $(\nu^j)_j$ obtained by taking the marginals $\nu^j=\int\rho^{\otimes j} d\nu$, and super-additivity entails that $H(\nu^j)$ and $I(\nu^j)$ are increasing in $j$ (because $\nu^{j:k}=\nu^k$ for any $k\leq j$). One can then consider the limit in $j\rightarrow\infty$ and introduce the so-called \textit{level 3}, or \textit{mean} quantities
\begin{align*}
    \mathcal{H}(\nu):=\sup_j H(\nu^j)= \lim_j H(\nu^j), && \mathcal{I}(\nu)=\sup_j I(\nu^j)= \lim_j I(\nu^j).
\end{align*}
Those are the natural extensions of $H$ and $I$ to the limit space $\mathcal{P}(\mathcal{P}(\mathbb{R}^{3}))$. As a supremum of convex functions, these quantities are convex in $\nu$. The remarkable property is that they are actually \textit{affine} \cite{HaurayMischler2012, RobinsonRuelle1967, Rougerie2020}: they satisfy
\begin{align}
\label{eq:introaffinity}
    \mathcal{H}(\nu)=\int_{\mathcal{P}(\mathbb{R}^3)} H(\rho) \nu(d\rho), && \mathcal{I}(\nu)=\int_{\mathcal{P}(\mathbb{R}^3)} I(\rho) \nu(d\rho).
\end{align}
If $\nu=\pi_t$ is the law of the random variable $f_t$, the right-hand side is $\mathbb{E}(\mathcal{F}(f_t))$ for either $\mathcal{F}=H$ or $\mathcal{F}=I$, precisely the quantity we want to show to be finite. The equality above means that to control this expectation, it suffices to control the sequence $\mathcal{F}(\pi_t^j)$ for all $j$. But $\pi_t^j$ is actually the weak limit of $F_t^{N:j}$, the $j$-marginal of the law of the $N$-particle system. Hence, by lower semi-continuity, we only need to control $\mathcal{F}(F_t^{N:j})$, and by super-additivity, that amounts to a control of $\mathcal{F}(F_t^{N})$. Hence the super-additivity~\eqref{eq:introsuperadditivity} and the affinity in infinite-dimension~\eqref{eq:introaffinity} are two key properties that allow us to pass properties of the particle system to the limit $N\rightarrow +\infty$.

However, for our proof, we need higher integrability than what the Fisher information provides by Sobolev embeddings. In particular, we want to be able to pass to the $N\rightarrow +\infty$ limit in (a part of) the dissipation of the Fisher information, which is close to a second-order generalized Fisher information as in Theorem~\ref{thm:side}. We will also need to pass to the limit in the entropy production, which can be seen as a first-order generalized Fisher information. This is why we were lead to prove Theorem~\ref{thm:side}. The super-additivity of the entropy production and the Fisher dissipation (point (2) in Theorem~\ref{thm:side}) can be shortly proven along the same lines as for the Fisher information. To show infinite-dimensional affinity (point (3)), we expand in Section~\ref{sec:infdimlimofDandK} on the framework developed by Rougerie \cite{Rougerie2020} for the Fisher information and its fractional versions, to provide an abstract theorem (Theorem~\ref{thm:abstractlvl3affinity}) that we then apply to the Fisher information and the entropy production. Note that we are unable to prove the exact affinity of the second order Fisher information but only up to the explicit constant $16$ (which suffices for our needs). The origin of this constant will be apparent in the proof.

We believe the nice behavior in infinite dimension of those different Fisher-type functionals is a novelty of this work and could prove useful beyond the study of the Landau equation.
\\

\textbf{Comparison with Feng and Wang \cite{FengWang2025}.} The day before the submission of the first version of this work on arXiv, X. Feng and Z. Wang submitted the remarkable preprint \cite{FengWang2025}, in which they show propagation of chaos for the Landau equation with Coulomb interactions for the first time in the literature, and their method is also able to treat very soft potentials with little to no modification. The author pledges that he did not know of the work \cite{FengWang2025} prior to its submission, and that the present work was pursued in parallel to theirs.

The method in \cite{FengWang2025} has very little in common with the current paper and provides another point of view on propagation of chaos. The authors of \cite{FengWang2025} work solely at the level of the law of the particle system \eqref{eq:kolintro}, without regularizing the potential. Their starting point is the recent duality method introduced by Bresch, Duerinckx and Jabin in \cite{BreschDuerinckxJabin2025}, which they elegantly tailor to the Landau setting. To handle the singularity of the system, they expand on the method by combining it with key estimates on the Landau equation. This clever use of the duality approach ultimately reduces the problem to the question of uniqueness for a certain dual hierarchy of equations, which seems easier to handle than the Landau-BBGKY hierarchy described earlier. The main similarity with the present paper is that both proofs crucially rely on some second-order version of the Fisher information, indicating that the high regularity it provides is certainly essential to handle very singular potentials.

\subsection{Notations and layout of the proof}

\textbf{Definitions and notations.}

The space of (Borel) probability measures on a complete separable metric space $\mathcal{X}$ is denoted by $\mathcal{P}(\mathcal{X})$. Endowed with the topology of weak convergence against continuous bounded functions, it is itself a complete separable metric space. A probability measure $F^n\in\mathcal{P}(\mathbb{R}^{3n})$ is said to be \textit{symmetric} if, for any permutation $\sigma$ of $\{1,...,n\}$, $$F^n(dv_1,..,dv_n)=F^n(dv_{\sigma(1)},...,dv_{\sigma(n)}).$$
For any such symmetric probability and $1\leq j \leq n$ we can consider the marginal $F^{n:j}\in\mathcal{P}(\mathbb{R}^{3j})$ obtained by
$$F^{n:j}(dv_1,...dv_j)=\int_{\mathbb{R}^{3(n-j)}} F^{n} dv_{j+1}...dv_n,$$
or by integrating over any other set of $n-j$ distinct variables.

By \textit{energy} of a probability $\mu\in\mathcal{P}(\mathbb{R}^{3})$ we mean the second moment $\int_{\mathbb{R}^3} \vert v\vert^2 \mu(dv)$, and by momentum the vector-valued first moment $\int_{\mathbb{R}^3} v \mu(dv)$.

The $\nabla$ operator applied to any function is the gradient in all variables, unless an index is specified, such as $\nabla_i$ w.r.t. $v_i$, $\nabla_{ij}$ w.r.t. $(v_i,v_j)$, and so on.

We also recall here the definitions of some objects that will be properly introduced when they are needed, so that the reader can consult this section directly to refresh their memory:

\noindent \textit{Vector fields and derivation operators.}
For $k=1,2,3$, the $3$-dimensional and $6$-dimensional vector fields $b_k=b_k(v_1,v_2)$ and $\tilde{b}_k$ are defined as:
\begin{align*}
b_k(v_1,v_2):=e_k \times (v_1-v_2)\in \mathbb{R}^3, &&\tilde{b}_k=\tilde{b}_k(v_1,v_2):=(b_k,-b_k) \in \mathbb{R}^6,
\end{align*}
with $e_k$ the canonical basis of $\mathbb{R}^3$ and $\times$ the usual cross product. They are introduced in Section~\ref{sec:entropy_and_fisher_in_kolmogorov}.
For $n\geq 2$, $k=1,2,3$, the derivation operator $\partial^n_{b_k}$ is
$$\partial^n_{b_k}:=\vert v_1-v_2\vert^{-\frac{1}{2}} (\alpha^n(\vert v_1-v_2\vert ))^{\frac{1}{4}}\tilde{b}_k(v_1,v_2)\cdot \nabla_{12},$$
and $\partial_{b_k}$ is defined with $\alpha$ in place of $\alpha^n$.

\noindent \textit{Functionals.} We recall the four important functionals used in this work. For non-negative functions $F^n$ on $\mathbb{R}^{3n}$, the entropy and Fisher information are respectively given by
\begin{align*}
    H(F^n):= \frac{1}{n}\int_{\mathbb{R}^{3n}} F^n \log F^n dV,
    &&I(F^n):= \frac{1}{n}\int_{\mathbb{R}^{3n}} \frac{\vert \nabla F^n \vert^2}{F^n} dV.
\end{align*}
For any $n\geq 2$, the (regularized) entropy production is
\begin{align*}
    D^n(F^n):&= \frac{1}{2}\int_{\mathbb{R}^{3n}} \alpha^n a(v_1-v_2):\left[(\nabla_1-\nabla_2)\log{F^n}\right]^{\otimes 2} F^n dV\\
    &=\sum_{k=1}^3 \frac{1}{2}\int_{\mathbb{R}^{3n}} \left(( \sqrt{\alpha^n} \tilde{b}_k\cdot \nabla_{12})F^n \right)^2 F^n dV,
\end{align*}
and the (regularized and convex version of the) Fisher dissipation is
\begin{align*}
    \mathcal{K}^n(F^n):&=\sum_{k=1}^3 \int_{\mathbb{R}^{3n}} \frac{\alpha^n(v_1-v_2)}{\vert v_1-v_2 \vert^2} \frac{\left( (\tilde{b}_k\cdot \nabla_{12})( \tilde{b}_k\cdot \nabla_{12})F^n \right)^2}{F^n} dV
    =\sum_{k=1}^3 \int_{\mathbb{R}^{3n}} \frac{\left( \partial^n_{b_k} \partial^n_{b_k}F^n \right)^2}{F^n} dV.
\end{align*}
The non-regularized versions $D$ and $\mathcal{K}$ are defined with $\alpha$ instead of $\alpha^n$. These functionals are introduced in Section~\ref{sec:infinite-dimension limit of the regularity estimates}. We call the $\mathcal{K}$ the \textit{convex version} of the Fisher dissipation because it is note exactly the dissipation of the Fisher information along the flow of the Landau or master equation, see Section \ref{sec:entropy_and_fisher_in_kolmogorov}.
\\

\noindent \textbf{Strategy and Layout.}

We begin with Section~\ref{sec:on the fisher dissipation and aux}, that contains preliminary results on different formulations of second-order Fisher information, their equivalence, and gives the proof of Theorem~\ref{thm:side} (1) and (2).

In the very short Section~\ref{sec:regularized particle system}, we show, for any number of particles $N\geq 2$, the well-posedness of the particle system~\eqref{eq:particlesystem} for the velocities $(V^N_1,...,V^N_N)$.

Section~\ref{sec:entropy_and_fisher_in_kolmogorov} focuses on the law $F^N_t$ of $(V^N_i(t))_{1\leq i \leq N}$, obeying the Kolmogorov equation~\eqref{eq:kolintro} associated to the particle system. We derive some uniform-in-$N$ estimates on $F^N_t$ on the entropy production and the dissipation of the Fisher information, in view of passing them to the $N\rightarrow +\infty$ limit. 

In Section~\ref{sec:tightness}, we use the boundedness of the Fisher information of $F^N_t$ to show the tightness of the empirical measure. Up to a subsequence, the laws of the empirical measures converge to a cluster point $\pi$:
$$\mathcal{L}((\mu^N_t)_t) \rightharpoonup \pi \in \mathcal{P}(C([0,T],\mathcal{P}(\mathbb{R}^3))).$$
The goal of the remaining sections is to show that $\pi=\delta_{(g_t)_t}$, so that the whole sequence $(\mu^N_t)_t$ converges to the deterministic $(g_t)_t$.

First, in Section~\ref{sec:infinite-dimension limit of the regularity estimates}, we pass the estimates of Section~\ref{sec:entropy_and_fisher_in_kolmogorov} to the $N\rightarrow +\infty$ limit to obtain the so-called \textit{mean} or \textit{level-3} estimates on $\pi$. It entails that any $\pi$-distributed random variable $f$ enjoys some improved regularity. This strongly relies on Theorem~\ref{thm:side}.

Second, in Section~\ref{sec:almostsureweaksolution}, we show that $f$ is almost surely a weak solution of the Landau equation. This is not enough to claim that $f=g$ because we do not have enough regularity on $f$ to apply known uniqueness results.

In the last Section~\ref{sec:a final estimate}, we derive new estimates showing that the Fisher dissipation of $f$ (that we control thanks to Section~\ref{sec:infinite-dimension limit of the regularity estimates}) bounds its $L^1([0,T],L^\infty(\mathbb{R}^3))$ norm, provided $f$ has enough moments and is not too concentrated. The finiteness of this norm is enough to apply the uniqueness results \cite{GuerinFournier2008, Fournier2010} and hence conclude the whole proof. We first show that the moments of $f$ propagate by using the H-formulation of the Landau equation (from \cite{Villani1998}) and adapting results on moments of H-solutions (from \cite{CarrapatosoDesvillettesHe2015}). To show that $f$ is not too concentrated, we cannot rely on the entropy (as is usually done) because we cannot control $H(f_t)$ uniformly in time. We hence measure non-concentration by a different, more general quantity introduced in \cite{FournierHauray2015}, which has the remarkable advantage of being continuous on $\mathcal{P}(\mathbb{R}^3)$.

In Section~\ref{sec:infdimlimofDandK}, we show the affinity in infinite dimension of the generalized Fisher information, i.e. point (3) in Theorem~\ref{thm:side}.

Finally, in Appendix~\ref{sec:derivkol}, we recall the derivation of the Kolmogorov equation~\eqref{eq:kolintro} from the particle system~\eqref{eq:particlesystem}. 

\begin{center}
\begin{large}
\textbf{Acknowledgements}
\end{large}
\end{center}

The author is thankful to his PhD advisors Cyril Imbert and Clément Mouhot for the many insightful discussions that lead to this work, and for proofreading of the manuscript. The author also wishes to thank Nicolas Fournier for an interesting and instructive exchange on an earlier version of this paper and for suggesting a proof of the tightness of the particle system that covers the Coulomb case.

\section{Preliminaries on Fisher information}
\label{sec:on the fisher dissipation and aux}

At the heart of our proof of propagation of chaos lies the monotonicity of the Fisher information for the Landau equation and the associated particle system \cite{GuillenSilvestre2023,CarrilloGuo2025}. In fact, it will —maybe even more crucially— rely on the dissipation term: the time-derivative of the Fisher information along the flow is a second-order term (the same way the time-derivative of the entropy is akin to a first order Fisher information). That is why we begin this work with preliminary results on second-order Fisher information, seldom studied before in the literature (although some arbitrary order Fisher information functionals played a role in \cite{LionsToscani1995}) Since the proof of the properties below also apply to the first-order case, we will in fact be able to prove points (1) and (2) in Theorem~\ref{thm:main}, for both orders.

We place ourselves in the general case: let $k_0\geq 1$, and let $b=b(v_1,...,v_{k_0})\in\mathbb{R}^{3k_0}$ be a smooth, bounded, divergence-free vector field, and we set $\partial=b\cdot\nabla_{1...k_0}$.

For any symmetric function $F:\mathbb{R}^{3N}\rightarrow \mathbb{R}_+$, $N\geq k_0$, recall that the generalized Fisher information along $\partial$ satisfies
$$I^\partial(F)=\int_{\mathbb{R}^{3N}} \frac{( \partial F)^2}{F} =\int_{\mathbb{R}^{3N}} F( \partial \log F)^2 =4\int_{\mathbb{R}^{3N}} ( \partial \sqrt{F})^2.$$
All these definitions admit straightforward second-order versions, which turn to be all different: let us define
\begin{align*}
    K^\partial_1(F)&:=K^\partial(F)=\int_{\mathbb{R}^{3N}}\! \frac{( \partial^2 F)^2}{F},\\
    K^\partial_{0}(F)&:=\int_{\mathbb{R}^{3N}}\! F( \partial^2 \log F)^2,\\
    K^\partial_{1/2}(F)&:=4\int_{\mathbb{R}^{3N}}\! ( \partial^2 \sqrt{F})^2.
\end{align*}
More generally for $\beta\in(0,1]$ we can define:
\begin{align*}
    K^\partial_{\beta}(F)&:=\beta^{-2}\int_{\mathbb{R}^{3N}} F^{1-2\beta}(\partial^2 F^\beta)^2\\
\end{align*}
One can wonder how these functionals relate to each other. It turns out that they are all comparable to one another (excepted $K^\partial_{1/3}$), and the difference can be expressed in terms of the first order functional
\begin{equation}
\label{eq:defJ}
    J^\partial(F):=\int_{\mathbb{R}^{3N}} F( \partial \log F)^4,
\end{equation}
that will play a crucial work at the end of this work. Moreover, once again at the exception of $K^\partial_{1/3}$, they all control a multiple of $J^\partial$.  Indeed, we have:
\begin{prop}
    \label{prop:func_comparison}
    For any smooth, bounded, divergence-free vector field $b(v_1,...,v_{k_0})\in\mathbb{R}^{3k_0}$, and $\partial=b\cdot\nabla_{1...k_0}$, and for any $\beta\in[0,1]$:
    \begin{align*}
    K^\partial_\beta&= K^\partial_1 +(\beta-1)(\beta+\frac{1}{3}) J^\partial\\
    &=K^\partial_{1/3} +(\beta-\frac{1}{3}) ^2J^\partial,
    \end{align*}
    so that the following equivalence holds:
    $$\frac{9}{4}(\beta-\frac{1}{3})^2K^\partial_1 \leq K_\beta^\partial \leq K_1^\partial.$$
\end{prop}
\begin{proof}
    We omit the superscript $\partial$. For $\beta>0$, we write
    $$K_{\beta}(F)=\int_{\mathbb{R}^{3N}} F\left( \frac{\partial^2 F^\beta}{\beta F^\beta}\right)^2$$
    and use the identity
    $$\frac{\partial^2 F^\beta}{\beta F^\beta}=\frac{\partial^2 F}{F}+(\beta-1)\frac{(\partial F)^2}{F^2}.$$
    For $\beta=0$, it also holds that
    $$\partial^2 \log F=\frac{\partial^2 F}{F}+(\beta-1)\frac{(\partial F)^2}{F^2}.$$
    So, in both cases,
    \begin{align*}
       K_\beta (F)&= \int_{\mathbb{R}^{3N}} F\left( \frac{\partial^2 F}{F}+(\beta-1)\frac{(\partial F)^2}{F^2}\right)^2\\
       &=\int_{\mathbb{R}^{3N}} \frac{(\partial^2 F)^2}{F} + 2(\beta-1)\int_{\mathbb{R}^{3N}} \frac{\partial^2 F(\partial F)^2}{F^2}+(\beta-1)^2\int_{\mathbb{R}^{3N}} \frac{(\partial F)^4}{F^3}.
    \end{align*}
    Notice that the first integral is $K_1(F)$ and the last integral is $J(F)$. Since $b$ is divergence-free, we can integrate by parts with $\partial$. We have, by integrating $\partial^2 F$ by parts, that the middle term writes
    \begin{align*}
    \int_{\mathbb{R}^{3N}} \frac{\partial^2 F(\partial F)^2}{F^2} =-2\int_{\mathbb{R}^{3N}} \frac{\partial^2 F(\partial F)^2}{F^2} +2 \int_{\mathbb{R}^{3N}} \frac{(\partial F)^4}{F^3},
    \end{align*}
    so that it equals $\frac{2}{3}J^N(F)$. Hence
    \begin{align*}
       K_\beta (F)
       &=\int_{\mathbb{R}^{3N}} \frac{(\partial^2 F)^2}{F} + \left(\frac{4}{3}(\beta-1)+(\beta-1)^2\right)J(F)\\
       &=K_1(F) + (\beta-1)(\beta+\frac{1}{3})J(F).
    \end{align*}
    In particular, 
    $K_{1/3}=K_1-\frac{4}{9}J$, so that
    $$K_\beta = K_{1/3} + \left(\frac{4}{9}+ (\beta-1)(\beta+\frac{1}{3})\right)J= K_{1/3} + (\beta-\frac{1}{3})^2J. $$
    Finally,
    \begin{align*}
        \frac{9}{4}(\beta-\frac{1}{3})^2K_1
        &=\frac{9}{4}(\beta-\frac{1}{3})^2 \left(K_\beta+(1-\beta)(\beta+\frac{1}{3})J\right)\\
        &\leq \frac{9}{4}(\beta-\frac{1}{3})^2 K_\beta +\frac{9}{4}(1-\beta)(\beta+\frac{1}{3})K_\beta \\
        &=K_\beta,
    \end{align*}
    because $(\beta-\frac{1}{3})^2 J \leq K_\beta$.
\end{proof}
\begin{remark}
    The key integration by parts above has appeared in many works over the years, the earliest mention we could find of it is in a footnote in \cite{Mckean1963}.
\end{remark}
This result allows us to freely exchange between formulations. It will turn useful because they have different advantages: The functional $K_0^\partial$ is the one actually related to the Fisher information dissipation, but the functional $K^\partial=K_1^\partial$ is more convenient because it is the one that is convex, lower-semi-continuous and super-additive (or at least the only one we were able to prove it is). Finally $K^\partial_{1/2}$ fits nicely in the Hilbert structure in which lies $\sqrt{F}$, which will be exploited to show the infinite-dimension affinity (Point (3) in Theorem~\ref{thm:side}).

We now show some relevant properties of the generalized Fisher information, that is points (1) and (2) in Theorem~\ref{thm:side}.

\begin{proof}[Proof of Theorem~\ref{thm:side} (1) and (2)]
We can prove these points nearly all at once by adapting one of the proofs for the Fisher information, found in \cite[Lemmas 3.5 and 3.7]{HaurayMischler2012}. We begin with the second order $K^\partial$, and the proof will be nearly identical for $I^\partial$. Let $N\geq k_0$ and $F$ be a non-negative function on $\mathbb{R}^{3N}$.
For any smooth function $\psi \in C^2(\mathbb{R}^{3N})$, we have
    $$\left\vert \frac{\partial^2 F}{F} \right\vert^2 \geq \left( \frac{\partial^2 F}{F} \right) \psi -\frac{\vert \psi\vert^2}{4},$$
    so that
    \begin{align*}
       K^\partial(F) &\geq \int_{\mathbb{R}^{3N}}\left(\left( \frac{\partial^2 F}{F} \right) \psi -\frac{\vert \psi\vert^2}{4}\right)F\\
        &=\int_{\mathbb{R}^{3N}}\partial^2 F  \psi -\frac{\vert \psi\vert^2}{4}F\\
        &=\int_{\mathbb{R}^{3N}}   \left(\partial^2\psi -\frac{\vert \psi\vert^2}{4}\right)F\\
    \end{align*}
    the last identity holding by two integration by parts, possible since $b$ is divergence free.
    Using smooth $\psi$ approximating $2(\partial^2 F)/F$, we can saturate this inequality, so that
    
    $$K^\partial(F) = \sup_{\psi \in C^2(\mathbb{R}^{3N})}\int_{\mathbb{R}^{3N}}   \left(\partial^2\psi -\frac{\vert \psi\vert^2}{4}\right)F.$$
    
    This formula makes sense in $[0,+\infty]$ for any probability measure $F\in \mathcal{P}(\mathbb{R}^{3N})$, because the integrals are against continuous bounded functions (recall that $b$ is smooth and bounded). Since $K^\partial$ is a pointwise supremum of linear continuous functionals, it is convex and lower semi-continuous for the weak convergence of probability measures, so point (1) is proven.
    
    Restricting the supremum to functions $\psi$ depending only on $v_1,...,v_j$ for some $k_0\leq j\leq N$, we get
    \begin{align*}
        K^\partial(F) &\geq \sup_{\psi \in C^2(\mathbb{R}^{3j})} \int_{\mathbb{R}^{3N}}   \left(\partial^2\psi -\frac{\vert \psi\vert^2}{4}\right)F^{:j}=K^\partial(F^{:j})
    \end{align*}
    because $(\partial^2\psi -\vert \psi\vert^2/4)$ does not depend on $v_{j+1},...,v_{N}$. This is exactly super-additivity, so point (2) is proven.

    We now turn to $I^\partial$. Similarly, for any smooth functions $\psi \in C^1(\mathbb{R}^{3N})$, we have
    \begin{align*}
       I^\partial(F) &\geq \int_{\mathbb{R}^{3N}}\left(\left( \frac{\partial F}{F} \right) \psi -\frac{\vert \psi\vert^2}{4}\right)F=\int_{\mathbb{R}^{3N}}   \left(\partial\psi -\frac{\vert \psi\vert^2}{4}\right)F\\
    \end{align*}
    so that taking the supremum
    $$I^\partial(F)=\sup_{\psi \in C^1(\mathbb{R}^{3N})}\int_{\mathbb{R}^{3N}}   \left(\partial\psi -\frac{\vert \psi\vert^2}{4}\right)F.$$
    From this formula we proceed exactly as for $K^\partial$ and the proof is concluded. 
\end{proof}
The proof of point (3) in Theorem~\ref{thm:side} requires the introduction of a suited Hilbertian framework, so we relegate it to Section~\ref{sec:infdimlimofDandK}. We can now start the proof of propagation of chaos itself, with the study of the particle system.

\section{The regularized particle system}
\label{sec:regularized particle system}

We fix $\gamma \in (-3,-2]$ and a final time $T>0$ for the remainder of this paper. We consider any sequence of regularized potentials $(\alpha^N)_{N\geq 2}$ as in Section~\ref{ssec:mainresult}. 

\begin{remark}
\label{rem:regularizedpotential}
An example of a sequence $(\alpha^N)_{N\geq 2}$ can be built as follows:
let $\chi:\mathbb{R_+}\rightarrow\mathbb{R^*_+}$ be a smooth increasing function such that $\chi(r)=0.99$ for $r\leq \frac{1}{2}$, $\chi(r)=r$ for $r\geq 1$, $0\leq \chi'(r)\leq 1$, and $\chi(r)\geq r$ for all $r$. For a regularization parameter $\eta>0$, consider $\chi_\eta(r):=\eta\chi(\frac{r}{\eta})$ and $\alpha_\eta(r):=\chi_\eta (r)^\gamma$.

We now fix, for any number of particles $N\geq 2$, a regularization parameter $\eta_N$ which is non-increasing in $N$ and goes to zero at infinity. We can take $\alpha^N=\alpha_{\eta_N}$, and we let the reader check that $\alpha^N \nearrow \alpha$, that $\alpha=\alpha^N$ outside the ball $B(0,\eta_N)$ and that it satisfies the condition~\eqref{eq:fishercond}.
\end{remark}

This short section is dedicated to the well-posedness of the particle system:
\begin{prop}
\label{prop:wellposednessparticlesystem}
    For any $N\geq 2$, the system~\eqref{eq:particlesystem} admits a pathwise unique solution $(V^N_i)_{1\leq i\leq N}$ defined for all positive times $t\geq 0$ and which is exchangeable. Furthermore, it conserves total momentum and energy: a.s., for any $t\geq 0$, $\sum_{i=1}^N V^N_i(t)=\sum_{i=1}^N V^N_{i,0}$ and $\sum_{i=1}^N (V^N_i(t))^2=\sum_{i=1}^N (V^N_{i,0})^2$.
\end{prop}
\begin{proof}
    The Lipschitz behaviour and boundedness of $b$ and $\sigma$ classically imply the existence of a unique strong solution (we refer to \cite{PardouxRascanu2014} for the theory of stochastic differential equations). Summing over $i$ the equation~\eqref{eq:particlesystem} directly yields the conservation of momentum using that $b(-v)=-b(v)$ and $B^N_{ji}=-B^N_{ij}$. The conservation of energy is obtained by applying the Itô formula to compute the evolution of $\sum_{i=1}^N (V^N_i(t))^2$. The details of the computation can be found in \cite[Proof of Proposition 3]{NicolasFournierArnaudGuillin2017} (in the hard potential case, but it is identical).
\end{proof}

\section{Entropy and Fisher information in the Kolmogorov equation}
\label{sec:entropy_and_fisher_in_kolmogorov}

In this section we study the law of the particle system. We gather estimates on the energy, entropy and Fisher information. In particular, the remarkable monotonicity of the Fisher information, which was proven in \cite{GuillenSilvestre2023} for $N=2$ (to obtain monotonicity for the non-linear Landau equation), and extended to general $N$ in \cite{CarrilloGuo2025}, will be enough to show tightness of the particle system in the next section. The other estimates will turn useful later on when studying the cluster points of the particle system.

For any $t\geq 0$, any $N\geq 2$, we let $F^N_t$ be the law of $(V^N_i(t))_{1\leq i\leq N}$, so that $F^N_t$ is a deterministic probability measure on $\mathbb{R}^{3N}$.
Using the Itô Formula to compute the derivative $\frac{d}{dt} \mathbb{E}\left[\varphi(V^N_1(t),...,V^N_N(t))\right]$ for a smooth test function $\varphi$, we obtain that $F^N_t$ is a distribution solution to
\begin{equation}
\label{eq:kol}
    \partial_t F^N_t = \mathcal{Q}^N(F^N_t):= \frac{1}{(N-1)} \sum_{\substack{i, j = 1 \\ i< j}}^N Q^N_{ij}(F^N_t) 
\end{equation}
where the operator $Q^N_{ij}$ acts on functions of $\mathbb{R}^{3N}$ through the $v_i$ and $v_j$ variables only:
\begin{equation}
\label{eq:defQij}
    Q_{ij}^N(F):= (\nabla_i - \nabla_j) \cdot \left[ \alpha^Na(v_i - v_j) (\nabla_i - \nabla_j) F\right].
\end{equation}
Note that the only dependence on $N$ of $Q_{ij}$ is through the regularization of the interaction potential. Two integration by parts show that $Q^N_{ij}$ is formally self-adjoint. The full derivation of~\eqref{eq:kol} is done in Appendix~\ref{sec:derivkol}.

\begin{remark}
    The theory of degenerate parabolic equations ensures uniqueness of weak solutions to~\eqref{eq:kol}. Equation~\eqref{eq:kolgeneral} in Appendix~\ref{sec:derivkol} shows that~\eqref{eq:kol} can be put in a more obvious divergence-form $\partial_tF^N = \frac{1}{2}\nabla\cdot(\Gamma \Gamma^T\nabla F^N)$ for some matrix $\Gamma$, which is exactly the setting of the theory of existence and uniqueness of Le Bris and Lions \cite{LeBrisLions2008,LeBrisLions2019}. One can check that $\Gamma \Gamma^T$ is not uniformly elliptic so the equation is truly degenerate.
\end{remark}

We first study the elementary conservation laws: the Kolmogorov equation preserves mass, momentum and kinetic energy:
\begin{prop}
    \label{prop:kol_conservations}
    The following quantities are constant for all times $t\in[0,T]$:
    \begin{align*}
        \int_{\mathbb{R}^{3N}} F^N_tdV&=1,\\
        \frac{1}{N}\int_{\mathbb{R}^{3N}} \sum_{i=1}^N v_iF^N_tdV&=\int_{\mathbb{R}^{3}} v g_0dv, \\
        \frac{1}{N}\int_{\mathbb{R}^{3N}} \vert V\vert^2 F^N_tdV&=E_0.
    \end{align*}
\end{prop}
\begin{proof}

    Since
    $$\int_{\mathbb{R}^{3N}}\varphi F^N_t dV =\int_{\mathbb{R}^{3N}}\varphi F^N_0 dV+ \int_0^t \int_{\mathbb{R}^{3N}}\mathcal{Q}^N(\varphi)F^N_s dV,$$
    the conservation laws are a consequence of $\mathcal{Q}^N(\varphi)=0$ for $\varphi(V)=1,\frac{1}{N}\sum_i v_i, \vert V\vert^2$, which is easily obtained by a direct computation.
\end{proof}

We will now study in more details the production of entropy and the dissipation of Fisher information. For the moment, we keep explicit expressions for the dissipation terms, making them more easily recognizable to any reader familiar with the Landau equation. We will later rewrite them with vector fields and derivation operators to fit the framework of Theorem~\ref{thm:side}. We begin with the entropy:
\begin{prop}
\label{prop:entropy_dissipation}
For any $N\geq 2$ the entropy of $F^N_t$ is non-increasing, and the following estimate holds for any $t\in[0,T]$:
$$H(F^N_t)+\frac{1}{2}\int_0^t \int_{\mathbb{R}^{3N}}\alpha^N a(v_1-v_2):\left[ (\nabla_1 -\nabla_2) \log F^N_s\right]^{\otimes 2} F^N_s dV ds\leq H_0.$$
\end{prop}
\begin{proof}
    We compute
    \begin{align*}
        \frac{d}{dt}H(F^N_t) &= \frac{1}{N}\int_{\mathbb{R}^{3N}} \mathcal{Q}^N( F^N_t) \log F^N_t dV\\
        &=\frac{1}{N(N-1)}\sum_{i<j}\int_{\mathbb{R}^{3N}} Q^N_{ij}( F^N_t) \log F^N_t dV\\
        &=\frac{1}{2}\int_{\mathbb{R}^{3N}} Q^N_{12}( F^N_t) \log F^N_t dV
    \end{align*}
    by symmetry. Recalling the expression of $Q^N_{12}$ and integrating by parts,
    \begin{align*}
        \frac{d}{dt}H(F^N_t) &=-\frac{1}{2}\int_{\mathbb{R}^{3N}} \left(\alpha^N a(v_1-v_2) (\nabla_1 -\nabla_2) \log F^N_t\right)\cdot  (\nabla_1 -\nabla_2) F^N_tdV\\
        &=-\frac{1}{2}\int_{\mathbb{R}^{3N}} \alpha^N a(v_1-v_2):\left[ (\nabla_1 -\nabla_2) \log F^N_t\right]^{\otimes 2} F^N_tdV.
    \end{align*}
    Integrating from $0$ to $t$ yields the result.
\end{proof}
\begin{remark}
    The dissipation for the Landau equation~\eqref{eq:landau} formally writes:
    $$H(f_t)+\frac{1}{2}\int_0^t \int_{\mathbb{R}^{3N}}\alpha^N a(v_1-v_2):\left[ (\nabla_1 -\nabla_2) \log f_s(v_1)f_s(v_2) \right]^{\otimes 2} f_s(v_1)f_s(v_2) dv_1dv_2 ds = H_0.$$
    It is exactly the entropy production for the Kolmogorov equation assuming $F_t^N=f_t^{\otimes N}$.
\end{remark}

In order to study the Fisher information, we first present some tools introduced by \cite{GuillenSilvestre2023}.
We define three vector fields $b_k=b_k(v_1,v_2)\in\mathbb{R}^3$ for $k=1,2,3$ as:
$$b_k(v_1,v_2):=e_k \times (v_1-v_2).$$
Here, $e_k$ is the canonical basis of $\mathbb{R}^3$.
We also define their $6$-dimensional counterparts
$$\tilde{b}_k=\tilde{b}_k(v_1,v_2):=(b_k,-b_k) \in \mathbb{R}^6.$$
One can check that $b_k$ and $\tilde{b}_k$ are divergence free. Furthermore they are orthogonal to $v_1-v_2$, so that for any function $\beta$ of $\vert v_1-v_2\vert$
\begin{equation}
\label{eq:bkandradial}
b_k\cdot\nabla_1 \beta(\vert v_1-v_2\vert) = \tilde{b}_k\cdot\nabla_{12} \beta(\vert v_1-v_2\vert)=0
\end{equation}
Their link with the Landau equation is the decomposition
$$a(v_1-v_2)=\sum_{k=1}^3 b_k\otimes b_k,$$
which in turn yields, using~\eqref{eq:bkandradial}, the decomposition
\begin{align}
\label{eq:decompQinQk}
Q^N_{12}(F)&=\alpha^N(\vert v_1-v_2\vert) \sum_{k=1}^3 (\tilde{b}_k\cdot \nabla_{12})(\tilde{b}_k\cdot \nabla_{12})F\\
&=\sum_{k=1}^3 \left(\sqrt{\alpha^N}\tilde{b}_k\cdot \nabla_{12}\right)\left(\sqrt{\alpha^N}\tilde{b}_k\cdot \nabla_{12}\right)F.\nonumber
\end{align}
Expressing the dissipation of the Fisher information in the Landau equation is made easier with these vector fields.
As observed in \cite{CarrilloGuo2025}, the Fisher information is non-increasing along the flow of the Kolmogorov equation, and this is exactly the linear counterpart of the monotonicity of the Fisher information for the Landau equation. It will allow us to show tightness of the particle system in the following section. We also isolate one term from the dissipation, which will be useful later on, to obtain regularity of the cluster points of the particle system. Exploiting the dissipation and not only the monotonicity of the Fisher information was first done in \cite{Ji2024} at the level of the non-linear Landau equation.
\begin{prop}
\label{prop:fisher_dissipation}
      For any $N\geq 2$, the Fisher information of $F^N_t$ is non-increasing, and the following estimate holds for any $t\in[0,T]$:
      \begin{equation}
      \label{eq:fisher_dissipation}
          I(F^N_t) +\frac{1}{4}\sum_{k=1}^3 \int_0^t \int_{\mathbb{R}^{3N}} \frac{\alpha^N(v_1-v_2)}{\vert v_1-v_2 \vert^2} \left( (\tilde{b}_k\cdot \nabla_{12})( \tilde{b}_k\cdot \nabla_{12})\log F^N_s \right)^2 F^N_s dV ds\leq I_0.
      \end{equation}
\end{prop}
\begin{proof}
    The monotonicity of the Fisher information for the Kolmogorov equation (with $\alpha$ instead of our regularized $\alpha^N$) was first observed in \cite{CarrilloGuo2025}. We do the proof again to extract (some of) the dissipation term.
    
    \textit{Step 1.} We write, using symmetry:
    \begin{align}
    \label{eq:prooffisher}
        \frac{d}{dt} I(F^N_t) &= \langle I'(F^N_t) , \mathcal{Q}^N (F^N_t)\rangle\nonumber\\
        &=\frac{1}{N-1}\sum_{i<j}\langle I'(F^N_t) , Q^N_{ij} (F^N_t)\rangle\nonumber\\
        &=\frac{N}{2}\langle I'(F^N_t) , Q^N_{12} (F^N_t)\rangle
    \end{align}
    We cut the Fisher information in two
    $$I(F)  = \frac{1}{N}\int_{\mathbb{R}^{3N}} \frac{\vert \nabla_1 F \vert^2+\vert \nabla_2 F \vert^2}{F} dV + \sum_{j=3}^N\int_{\mathbb{R}^{3N}} \frac{\vert \nabla_j F \vert^2}{F} dV  =:I_{12}(F)+I_{other}(F),$$
    and show the monotonicity of each part separately.
    
    \textit{Step 2.} We claim that $\langle I'_{other} (F^N_t) , Q^N_{12} (F^N_t)\rangle \leq 0$. Indeed, because of~\eqref{eq:bkandradial}, the vector fields $\sqrt{\alpha^N}\tilde{b}_k$ are divergence free. 
    Since, for $j\geq 3$, the operators $\nabla_j$ and $\mathcal{L}_k:=\sqrt{\alpha^N}\tilde{b}_k \cdot \nabla_{12}$ commute (because they depend on different variables), a direct computation yields
    \begin{align*}
    \langle I_{other}'(F) , \mathcal{L}_k F \rangle &=\sum_{j\geq 3} \int 2 \frac{\nabla_j F \cdot \nabla_j \mathcal{L}_k F}{F} - \frac{\vert \nabla_j F \vert^2}{F^2}\mathcal{L}_k F dV\\
    &=\sum_{j\geq 3} \int 2 \frac{\nabla_j F \cdot \mathcal{L}_k \nabla_j F} {F} - \frac{\vert \nabla_j F \vert^2}{F^2}\mathcal{L}_k F dV\\
    &=\sum_{j\geq 3} \int \mathcal{L}_k \left[\frac{\vert \nabla_j F \vert^2}{F}\right] dV\\
    &=0,
    \end{align*}
    using that $\sqrt{\alpha^N}\tilde{b}_k$ is divergence free for the last line. Differentiating again in the direction $\mathcal{L}_k$,
    we get
    $$\langle I_{other}''(F) \mathcal{L}_k F, \mathcal{L}_k F \rangle + \langle I_{other}'(F) , \mathcal{L}_k \mathcal{L}_k F \rangle=0 $$
    so that using~\eqref{eq:decompQinQk},
    $$\langle I_{other}'(F^N_t) , Q^N_{12} (F^N_t)\rangle = \sum_{k=1}^3\langle I_{other}'(F) , \mathcal{L}_k \mathcal{L}_k F^N_t \rangle = - \langle I_{other}''(F) \mathcal{L}_k F, \mathcal{L}_k F \rangle \leq 0, $$
    by convexity of $I_{other}$ (which can be seen, for instance, by computing $I_{other}''$ explicitly).
    
    \textit{Step 3.} To deal with $\langle I'_{12}(F^N_t) , Q^N_{12} (F^N_t)\rangle$, we rely on the breakthrough result \cite{GuillenSilvestre2023}, which is essentially the $N=2$ case. It is shown there \cite[Lemma 8.4]{GuillenSilvestre2023} that for any smooth symmetric $F:\mathbb{R}^2 \rightarrow \mathbb{R}_+$,
    $$\langle I'_{12}(F), Q^N_{12}F\rangle \leq -D_{parallel}-D_{radial}-D_{spherical} + \sup_{r>0} \left(\frac{r^2((\alpha^N)'(r))^2}{2\alpha^N(r)^2}\right)R_{spherical}$$
    where $D_{parallel},D_{radial},D_{spherical},R_{spherical}\geq 0$ are integral expressions corresponding to the derivative of different Fisher informations in different directions (there is a $\frac{1}{2}$ difference with \cite[Lemma 8.4]{GuillenSilvestre2023} because of our normalization of the Fisher information).
    The term $D_{spherical}$ can be used to control the error $R_{spherical}$. Indeed, it holds that
    $$D_{spherical}\geq 2\Lambda R_{spherical},$$
    where $\Lambda$ is the largest constant in a Poincaré-type inequality on the sphere (see \cite[Lemma 9.1]{GuillenSilvestre2023}). There, it is shown that $\Lambda\geq \frac{19}{4}$, and we know that $\Lambda\geq \frac{11}{2}$ as improved by Ji in \cite{Ji2024a}.
    By hypothesis~\eqref{eq:fishercond} on the regularized potential:
    $$\sup_{r>0} \left(\frac{r^2((\alpha^N)'(r))^2}{2\alpha^N(r)^2}\right)\leq\frac{\gamma^2}{2(0.99)^2}\leq \frac{11}{2} \leq \Lambda, $$
    so that dropping the other terms
    \begin{align}
    \label{eq:boundwith1/2}
        \langle I_{12}'(F), Q^N_{12}F\rangle \leq -D_{spherical} + \sup_{r>0} \left(\frac{r^2((\alpha^N)'(r))^2}{2\alpha^N(r)^2}\right)\frac{1}{2\Lambda}D_{spherical}\leq -\frac{1}{2}D_{spherical}.
    \end{align}
    We want to apply this result to $F=F^N_t(\cdot,\cdot,v_3,...,v_N)$ and integrate in $\bar{V}=(v_3,...,v_N)$.
    With the $2/N$ factor coming from our normalization of the Fisher information, we have
    $$I_{12}(F^N_t) =\frac{2}{N}\int I_{12}(F^N_t(\cdot,\bar{V})) d\bar{V}.$$
    Hence it holds that, plugging in the expression of $D_{spherical}$ \cite[(8.3)]{GuillenSilvestre2023} integrated over $\bar{V}$,
    \begin{align*}
        \langle I'_{12}(F^N_t) , Q^N_{12} (F^N_t)\rangle&=\frac{2}{N}\int\langle I'_{12}(F^N_t(\cdot,\bar{V}) , Q^N_{12} (F^N_t(\cdot,\bar{V}))\rangle d\bar{V}\\
        &\leq- \frac{1}{N} \sum_{k,l=1}^3 \int_{\mathbb{R}^{3N}} \frac{\alpha^N(v_1-v_2)}{2\vert v_1-v_2 \vert^2} \left( (\tilde{b}_k\cdot \nabla_{12})( \tilde{b_l}\cdot \nabla_{12})\log F^N_t \right)^2 F^N_t dV.
    \end{align*}
    Going back to~\eqref{eq:prooffisher} with the results of Steps 2 and 3,
    \begin{align*}
        \frac{d}{dt} I(F^N_t) &\leq \frac{N}{2}\langle I'_{12}(F^N_t) , Q^N_{12} (F^N_t)\rangle \\
        &\leq -\frac{1}{4}\sum_{k,l=1}^3 \int_{\mathbb{R}^{3N}} \frac{\alpha^N(v_1-v_2)}{\vert v_1-v_2 \vert^2} \left( (\tilde{b}_k\cdot \nabla_{12})( \tilde{b_l}\cdot \nabla_{12})\log F^N_t \right)^2 F^N_t dV.
    \end{align*}
    Keeping only the $k=l$ terms and integrating from $0$ to $t$ yields the result.
    \end{proof}

\begin{remark}
    \label{rem:on0.99}
    The hypothesis~\eqref{eq:fishercond} on the regularized potential allows us to get a clean $\frac{1}{2}$ in the bound~\eqref{eq:boundwith1/2} so we get a $\frac{1}{4}$ in~\eqref{eq:fisher_dissipation}. However, \eqref{eq:fishercond} can be replaced by the more general condition
    \begin{equation*}
    \frac{r \vert (\alpha^{N})'(r) \vert}{\alpha^{N}(r)}\leq \frac{-\gamma}{\theta},
    \end{equation*}
    with $\theta> -\gamma /\sqrt{22}$, which yields~\eqref{eq:fisher_dissipation} with a positive constant depending on $\theta$ rather than
    $\frac{1}{4}$.
\end{remark}

\begin{remark}
    In all rigor, a regularization argument is necessary to perform the computations above on weak solutions. We can work with a particle system with Kolmogorov equation $\partial_t F^{N,\varepsilon}_t = \mathcal{Q}^N( F^{N,\varepsilon}_t)+\varepsilon \Delta F^{N,\varepsilon}_t$, which has a unique \textit{smooth} solution for any $\varepsilon>0$, by classical uniformly parabolic theory. Entropy and Fisher information also decrease along the heat flow, so we can write the same estimates with an additional dissipation term that we can ignore. Using the tightness provided by the Fisher information, we can then let $\varepsilon\rightarrow 0$ and the well-posedness of the particle system ensures that we recover the law $F^N$ in the limit. We can easily pass the estimates on the entropy and the Fisher information to the limit because all the terms are lower semi-continuous for the weak convergence of probability measures (it is well known for $I$ and $H$, and proven in Proposition~\ref{prop:convexlscsuperadditve} for the production/dissipation terms). The lower semi-continuity makes it so we only obtain an inequality rather than an equality in the estimates.
    
    We kept a formal approach to lighten the presentation. For a more detailed study of~\eqref{eq:kol} without any regularization on $\alpha$, see \cite{CarrilloGuo2025}. The same regularization technique was invoked in \cite[Proof of Proposition 5.1]{FournierHaurayMischler2014}.
\end{remark}

\section{Tightness of the particle system}
\label{sec:tightness}
We are now ready to show that the sequence of empirical measures of the particle system is tight.

We begin with a key estimate on the proximity of particles, that will be used to control the singularities of the system. It relies strongly on the control of the Fisher information.
\begin{lemma}
\label{lem:keyestimate}
    For any $N\geq 2$:
    $$\sup_{t\geq 0} \mathbb{E}\left[ \left\vert V^N_1(t) - V^N_2(t)\right\vert^{-2} \right] \leq I_0.$$
\end{lemma}
\begin{proof}
    We write the expectation using the law $F_{t}^{N:2}$ of $(V^N_1(t),V^N_2(t))$:
    \begin{align}
        \mathbb{E}\left[ \left\vert V^N_i(t) - V^N_j(t)\right\vert^{-2} \right] = \int_{\mathbb{R}^6}\left\vert v_1 - v_2\right\vert^{-2}F_{t,2}^N(v_1,v_2)dv_1dv_2
    \end{align}
    where $F_{t}^{N:2}$ is given by the $2$-marginal $$F_{t}^{N:2}(v_1,v_2) = \int_{\mathbb{R}^{3(N-2)}}F^N_t(v_1,v_2,v_3,\ldots,v_N)dv_3 \ldots dv_N.$$
    Then we make use of the unitary change of variables:
    $$\phi: (v_1,v_2) \mapsto \frac{1}{\sqrt{2}}(v_1-v_2,v_1+v_2),$$ which allows us to write
    \begin{align*}
         \int_{\mathbb{R}^6}\left\vert v_1 - v_2\right\vert^{-2}F_{t}^{N:2}(v_1,v_2)dv_1dv_2& = \frac{1}{2}\int_{\mathbb{R}^6}\left\vert w_1\right\vert^{-2}F_{t}^{N:2}\circ \phi^{-1}(w_1,w_2)dw_1dw_2 \\&= \frac{1}{2}\int_{\mathbb{R}^6}\left\vert w_1\right\vert^{-2}f(w_1)dw_1,
    \end{align*}
    by letting $f(w_1) = \int_{\mathbb{R}^3} F_{t}^{N:2}\circ \phi^{-1}(w_1,w_2)dw_2$.
    The Hardy inequality applied to $\sqrt{f}$ writes:
    \begin{align*}
        \frac{1}{2}\int_{\mathbb{R}^6}\left\vert w_1\right\vert^{-2}f(w_1)dw_1 = \frac{1}{2}\left\Vert \frac{\sqrt{f}}{\vert\cdot\vert}\right\Vert_{L^2(\mathbb{R}^3)}^2\leq 2 \left\Vert \nabla \sqrt{f}\right\Vert_{L^2(\mathbb{R}^3)}^2 = \frac{1}{2}I(f).
    \end{align*}
    Using super-additivity of Fisher information, $I(f)$ is less than two times (because of our normalization) the Fisher information of $F_{t}^{N:2}\circ \phi^{-1}$, since it is its $1$-marginal. The Fisher information being invariant under the unitary transform $\phi$, we finally obtain
    $$\frac{1}{2}I(f) \leq I(F_{t}^{N:2}\circ \phi^{-1}) = I(F_{t}^{N:2}) \leq  I(F^N_t),$$
    where super-additivity was used again for the last inequality.
    Using the monotonicity of the Fisher information (Proposition~\ref{prop:fisher_dissipation}) to bound this quantity by $I_0$ concludes.
\end{proof}

This control is enough to show tightness of the sequence of empirical measures. In fact, one can check that the following proof extends to $\gamma>-4$ with little to no adaptation.

Let us first define a suitable distance on $\mathcal{P}(\mathbb{R}^3)$.
We claim that there exists a sequence $(\varphi_n)_{n\geq 1}$ such that for every $n\geq 1$, $\varphi_n \in C^2_0(\mathbb{R}^3)$, with
\begin{equation}
\label{eq:phi_nbounds}
    \Vert \varphi_n \Vert_{L^\infty} + \Vert \nabla \varphi_n \Vert_{L^\infty}+ \Vert \nabla^2 \varphi_n \Vert_{L^\infty}\leq 1,
\end{equation}
and such that the distance
$$d(\mu,\nu):= \sum_{n\geq 1} 2^{-n} \left\vert \int_{\mathbb{R}^3} \varphi_n (d\mu-d\nu) \right\vert$$
metricizes the weak convergence of measures on $\mathcal{P}(\mathbb{R}^3)$. In fact, it suffices to choose $(\varphi_n)_{n\geq 1}$ such that its span is dense in $C_0(\mathbb{R}^3)$ for the uniform topology (see \cite[Chapter 6]{Villani2009}, and \cite{Tardy2023} where this distance is used in a similar context).

Recall that the \textit{random} empirical measure at time $t$ is $\mu^N_t=\frac{1}{N}\sum_{i=1}^N \delta_{V^N_i(t)}$. The main result of this section is the following:
\begin{prop}
\label{prop:tightness2}
Consider the random variable $\mu^N=(\mu^N_t)_{t\in[0,T]} \in C\left([0,T],\mathcal{P}(\mathbb{R}^3)\right)$. The family $\left(\mathcal{L}\left( \mu^N\right) \right)_{N\geq 2 }$ is tight in $\mathcal{P} \left(C\left([0,T],\mathcal{P}(\mathbb{R}^3)\right)\right)$.
\end{prop}

Before proving Proposition~\ref{prop:tightness2}, we show that we can bound a Hölder seminorm of $\mu^N$ in expectation.
\begin{lemma}
    \label{lem:holdercont}
    There exists $C_{T,I_0,\gamma}>0$ depending on $T,I_0$ and $\gamma$ such that 
    \begin{align*}
   \mathbb{E}\left[  \sup_{s\neq t\in[0,T]} \frac{d(\mu^N_s,\mu^N_t)}{\vert s-t \vert^\frac{1}{8}} \right]\leq C_{T,I_0,\gamma}.
\end{align*}
\end{lemma}

\begin{proof}
We need to show the Holder continuity of $\mu^N$ for the distance $d$, which involves testing $\mu^N$ against the $\varphi_n$. Let us first work with a given $\varphi \in C^2_b(\mathbb{R}^3)$. For any $s,t\in[0,T]$, the Itô formula yields
\begin{align}
        &\int_{\mathbb{R}^3} \varphi (d\mu^N_t-d\mu^N_s)\nonumber\\
        &=
       \frac{1}{N}\sum_{i=1}^N\varphi(V^N_i(t))-\frac{1}{N}\sum_{i=1}^N\varphi(V^N_i(s))\nonumber \\
       \label{eq:Itôempiricalmeasures6}
       &= \frac{2}{N(N-1)}\int_s^t \sum_{\substack{i,j=1\\ i\neq j}}^N  b^N\left(V^N_i(u) -V^N_j(u)\right) \cdot \nabla\varphi\left(V^N_i(u)\right) du\ \\
       \label{eq:Itôempiricalmeasures7}
       &+ \frac{1}{N(N-1)} \int_s^t \sum_{\substack{i,j=1\\ i\neq j}}^N  \alpha^Na\left(V^N_i(u) -V^N_j(u)\right):\nabla^2\varphi\left(V^N_i(u)\right)du\\
       \label{eq:Itôempiricalmeasures8}
       &+\sqrt{\frac{2}{N-1}}\frac{1}{N}\sum_{\substack{i,j=1\\ i\neq j}}^N \int_s^t  \nabla \varphi\left(V^N_i(u)\right)\cdot \sigma^N \left(V^N_i(u) - V^N_j(u)\right)dB^N_{ij}(u).
    \end{align}
We bound each term one by one.

\textit{Step 1.} For the first term \eqref{eq:Itôempiricalmeasures6}, we first symmetrize the sum using that $b(-v)=-b(v)$:
\begin{align*}
    &\frac{2}{N(N-1)}\int_s^t \sum_{\substack{i,j=1\\ i\neq j}}^N  b^N\left(V^N_i(u) -V^N_j(u)\right) \cdot \nabla\varphi\left(V^N_i(u)\right) du\\
    &=\frac{2}{N(N-1)}\int_s^t \sum_{\substack{i,j=1\\ i<j}}^N  b^N\left(V^N_i(u) -V^N_j(u)\right) \cdot \left(\nabla\varphi\left(V^N_i(u)\right)-\nabla\varphi\left(V^N_j(u)\right)\right) du.
\end{align*}
Thanks to the Lipschitz regularity of $\nabla \varphi$,
$$\left\vert\nabla\varphi\left(V^N_i(u)\right)-\nabla\varphi\left(V^N_j(u)\right)\right\vert \leq \Vert \nabla^2 \varphi\Vert_{L^\infty}  \vert V^N_i(u)-V^N_j(u) \vert,$$ so using this altogether with $\vert b^N(v)\vert \leq 2 \vert v \vert^{\gamma+2}$, we get
\begin{multline*}
\frac{2}{N(N-1)}\left\vert \int_s^t \sum_{\substack{i,j=1\\ i<j}}^N  b^N\left(V^N_i(u) -V^N_j(u)\right) \cdot \left(\nabla\varphi\left(V^N_i(u)\right)-\nabla\varphi\left(V^N_j(u)\right)\right) du\right\vert\\
\leq \frac{4\Vert \nabla^2 \varphi\Vert_{L^\infty}}{N(N-1)}\sum_{\substack{i,j=1\\ i<j}}^N \int_s^t   \vert V^N_i(u)-V^N_j(u) \vert^{\gamma+2}.
\end{multline*}
By the Hölder inequality,
\begin{multline*}
\frac{4\Vert \nabla^2 \varphi\Vert_{L^\infty}}{N(N-1)}\sum_{\substack{i,j=1\\ i<j}}^N \int_s^t   \vert V^N_i(u)-V^N_j(u) \vert^{\gamma+2}\\
\leq \frac{4\Vert \nabla^2 \varphi\Vert_{L^\infty}}{N(N-1)}\sum_{\substack{i,j=1\\ i<j}}^N \left(\int_s^t   \vert V^N_i(u)-V^N_j(u) \vert^{-2}du\right)^{-\frac{\gamma+2}{2}}\vert t-s\vert^{1+\frac{\gamma+2}{2}},
\end{multline*}
so that dividing by $\vert t-s\vert^{1+\frac{\gamma+2}{2}}$, and taking the supremum in $t$ and $s$, we get
\begin{multline*}
    \sup_{s\neq t\in[0,T]} \vert t-s\vert^{-(1+\frac{\gamma+2}{2})}\frac{2}{N(N-1)}\left\vert \int_s^t \sum_{i \neq j}  b^N\left(V^N_i(u) -V^N_j(u)\right) \cdot \nabla\varphi\left(V^N_i(u)\right) du \right\vert\\
    \leq \frac{4\Vert \nabla^2 \varphi\Vert_{L^\infty}}{N(N-1)}\sum_{\substack{i,j=1\\ i<j}}^N \left(\int_0^T   \vert V^N_i(u)-V^N_j(u) \vert^{-2}du\right)^{-\frac{\gamma+2}{2}}
\end{multline*}
Taking expectations, and using that $0\leq -\frac{\gamma+2}{2} \leq 1$ to ensure that $\mathbb{E}(\vert X\vert^{-\frac{\gamma+2}{2}}) \leq (\mathbb{E}\vert X\vert)^{-\frac{\gamma+2}{2}}$, we get
\begin{align}
    &\mathbb{E}\left[\sup_{s\neq t\in[0,T]} \vert t-s\vert^{-(1+\frac{\gamma+2}{2})}\frac{2}{N(N-1)}\left\vert \int_s^t \sum_{i \neq j}  b^N\left(V^N_i(u) -V^N_j(u)\right) \cdot \nabla\varphi\left(V^N_i(u)\right) du \right\vert\right]\nonumber\\
    &\leq \mathbb{E}\left[\frac{4\Vert \nabla^2 \varphi\Vert_{L^\infty}}{N(N-1)}\sum_{\substack{i,j=1\\ i<j}}^N \left(\int_0^T   \vert V^N_i(u)-V^N_j(u) \vert^{-2}du\right)^{-\frac{\gamma+2}{2}}\right]\nonumber\\
    &\leq \frac{4\Vert \nabla^2 \varphi\Vert_{L^\infty}}{N(N-1)}\sum_{\substack{i,j=1\\ i<j}}^N \left(\int_0^T   \mathbb{E}\left[\vert V^N_i(u)-V^N_j(u) \vert^{-2}\right]du\right)^{-\frac{\gamma+2}{2}}\nonumber\\
    &\leq 2\Vert \nabla^2 \varphi\Vert_{L^\infty}\left(\int_0^T   \mathbb{E}\left[\vert V^N_1(u)-V^N_2(u) \vert^{-2}\right]du\right)^{-\frac{\gamma+2}{2}}\nonumber\\
    \label{eq:tightcond2}
    &\leq 2\Vert \nabla^2 \varphi\Vert_{L^\infty}\left(TI_0\right)^{-\frac{\gamma+2}{2}}
\end{align}
by exchangeability, and using Lemma~\ref{lem:keyestimate} for the last line. Remark that the Hölder exponent $(1+\frac{\gamma+2}{2})$ is larger than the $\frac{1}{8}$ we seek.

\textit{Step 2.} Following the same steps (without symmetrizing at the beginning) and using $\vert \alpha^N a(v)\vert\leq  \vert v \vert^{\gamma+2}$ , we can bound the second line \eqref{eq:Itôempiricalmeasures7} by
\begin{align}
\label{eq:tightcond3}
    \mathbb{E}\Bigg[\sup_{s\neq t\in[0,T]} \frac{\vert t-s\vert^{-(1+\frac{\gamma+2}{2})}}{N(N-1)}\Bigg\vert \int_s^t \sum_{i\neq j}  \alpha^Na\left(V^N_i(u) -V^N_j(u)\right)&:\nabla^2\varphi\left(V^N_i(u)\right)du \Bigg\vert\Bigg]\nonumber\\
    &\leq \Vert \nabla^2 \varphi\Vert_{L^\infty}\left(TI_0\right)^{-\frac{\gamma+2}{2}}.
\end{align}
\textit{Step 3.} Finally, we treat the martingale term \eqref{eq:Itôempiricalmeasures8}, that we rewrite using $B^N_{ji}=-B^N_{ij}$:
    \begin{align*}
        & \sqrt{\frac{2}{N-1}}\frac{1}{N}\sum_{i<j} \int_s^t  \nabla \varphi\left(V^N_i(u)\right)\cdot \sigma^N \left(V^N_i(u) - V^N_j(u)\right)dB^N_{ij}(u)\\
    &= \frac{\sqrt{2}}{N\sqrt{(N-1)}}\sum_{i<j} \int_s^t  \left(\nabla\varphi\left(V^N_i(u)\right)-\nabla\varphi\left(V^N_j(u)\right)\right)\cdot \sigma^N \left(V^N_i(u) - V^N_j(u)\right)dB^N_{ij}(u)
    \end{align*}
    since $\sigma^N(-v)=\sigma^N(v)$.
    We introduce the process
    $$\mathcal{M}^N(t):=\sum_{i<j} \int_0^t  \left(\nabla\varphi\left(V^N_i(u)\right)-\nabla\varphi\left(V^N_j(u)\right)\right)\cdot \sigma^N \left(V^N_i(u) - V^N_j(u)\right)dB^N_{ij}(u).$$
    Using bracket notation $\left[ X\right]_t$ for the quadratic variation at time $t$ of a process $X$, the Burkholder-Davis-Gundy inequality \cite[Corollary 4.2]{RevuzYor1999} writes:
\begin{align*}
    \mathbb{E} \vert \mathcal{M}^N(t)-\mathcal{M}^N(s) \vert^4 \leq C \mathbb{E} \left[ \mathcal{M}^N(\cdot)-\mathcal{M}^N(s)\right]_t^2.
\end{align*}
The independence of the Brownian motions allows us to write the variation of the sum as the sum of the variations:
\begin{align*}
    \big[ \mathcal{M}^N(\cdot)&-\mathcal{M}^N(s)\big]_t\\
    &=\left[ \sum_{i<j}\int_{s}^\cdot \left(\nabla\varphi\left(V^N_i(u)\right)-\nabla\varphi\left(V^N_j(u)\right)\right) \cdot \sigma^N(V^N_1(u) - V^N_j(u)) dB^{N}_{ij}(u)\right]_t\\
    &=\sum_{i<j} \int_{s}^t \left\vert\sigma^N(V^N_i(u) - V^N_j(u)) \left(\nabla\varphi\left(V^N_i(u)\right)-\nabla\varphi\left(V^N_j(u)\right)\right) \right\vert^2  du.
\end{align*}
But the estimate
\begin{align*}\left(\nabla\varphi\left(v\right)-\nabla\varphi\left(w\right)\right) &\leq  \min(2\Vert \nabla \varphi\Vert_{L^\infty},\Vert \nabla^2 \varphi\Vert_{L^\infty}\vert v-w\vert)\\
&\leq C(\Vert \nabla \varphi\Vert_{L^\infty}+\Vert \nabla^2 \varphi\Vert_{L^\infty})\vert v-w\vert^{-\frac{\gamma+2}{2}}
\end{align*}
holds for some absolute constant $C>0$, since $0\leq -\frac{\gamma+2}{2}\leq 1$, and we also have $\Vert \sigma^N(v-w)\Vert \lesssim  \vert v-w\vert^{\frac{\gamma+2}{2}}.$
Hence the integrand above is bounded and
$$\big[ \mathcal{M}^N(\cdot)-\mathcal{M}^N(s)\big]_t \leq C(\Vert \nabla \varphi\Vert_{L^\infty}+\Vert \nabla^2 \varphi\Vert_{L^\infty})^2 \frac{N(N-1)}{2}\vert t-s\vert.$$
We thus reach
\begin{equation}
\label{eq:holdercontexp}
\mathbb{E} \vert \mathcal{M}^N(t)-\mathcal{M}^N(s) \vert^4 \leq C(\Vert \nabla \varphi\Vert_{L^\infty}+\Vert \nabla^2 \varphi\Vert_{L^\infty})^{4} \left(\frac{N(N-1)}{2}\right)^{2}\vert t-s\vert^2.
\end{equation} 
Notice that this also yields a control of $\mathcal{M}^N(t)$ since $\mathcal{M}^N(0)=0$.
We now rely on the following Sobolev embedding \cite[Theorem 8.2]{NezzaPalatucciValdinoci2011}: for any $f$,
\begin{equation*}
    \sup_{s\neq t \in [0,T]} \frac{\vert f(t)-f(s)\vert}{\vert t-s\vert^{\frac{1}{8}}}\leq C_{\mathrm{Sob}}\left(\int_0^T \vert f(t)\vert^4 dt+\iint_{[0,T]^2} \frac{\vert f(t)-f(s)\vert^4}{\vert t-s\vert^{\frac{5}{2}}}ds dt\right)^{\frac{1}{4}}.
\end{equation*}
With $f=\mathcal{M}^N$, taking expectations and using $\mathbb{E}(\vert X\vert^{\frac{1}{4}}) \leq (\mathbb{E}\vert X\vert)^{\frac{1}{4}}$, we get
\begin{align*}
    \mathbb{E}&\left[ \sup_{s\neq t \in [0,T]} \frac{\vert \mathcal{M}^N(t)-\mathcal{M}^N(s)\vert}{\vert t-s\vert^{\frac{1}{8}}}\right]\\&\leq C_{\mathrm{Sob}}\left(\int_0^T \mathbb{E}(\vert \mathcal{M}^N (t)\vert^4)dt+\iint_{[0,T]^2} \frac{\mathbb{E}(\vert \mathcal{M}^N(t)-\mathcal{M}^N (s)\vert^4)}{\vert t-s\vert^{\frac{5}{2}}}ds dt\right)^{\frac{1}{4}}\\
    &\leq C_{\mathrm{Sob}}(\Vert \nabla \varphi\Vert_{L^\infty}+\Vert \nabla^2 \varphi\Vert_{L^\infty}) (N(N-1)/2)^{\frac{1}{2}}\left(\int_0^T t^2dt+\iint_{[0,T]^2} \vert t-s\vert^{-\frac{1}{2}}ds dt\right)^{\frac{1}{4}}\\
    &\leq C_T(\Vert \nabla \varphi\Vert_{L^\infty}+\Vert \nabla^2 \varphi\Vert_{L^\infty}) (N(N-1)/2)^{\frac{1}{2}}
\end{align*}
thanks to \eqref{eq:holdercontexp}. We finally obtain the following control of the third term \eqref{eq:Itôempiricalmeasures8}:
\begin{align}
         \mathbb{E}\Bigg[\sup_{s\neq t \in [0,T]}\frac{\sqrt{2}\vert t-s\vert^{-\frac{1}{8}}}{N\sqrt{N-1}}&\Bigg\vert \sum_{i\neq j} \int_s^t  \nabla \varphi\left(V^N_i(u)\right)\cdot \sigma^N \left(V^N_i(u) - V^N_j(u)\right)dB^N_{ij}(u)\Bigg\vert\Bigg]\nonumber\\
    &= \mathbb{E}\left[\sup_{s\neq t \in [0,T]}\frac{\sqrt{2}}{N\sqrt{N-1}} \frac{\vert \mathcal{M}^N(t)-\mathcal{M}^N(s)\vert}{\vert t-s\vert^{\frac{1}{8}}}\right]\nonumber\\
    \label{eq:tightcond4}
    &\leq  \frac{C_T}{\sqrt{N}}(\Vert \nabla \varphi\Vert_{L^\infty}+\Vert \nabla^2 \varphi\Vert_{L^\infty}).
\end{align}
\textit{Step 4.} Using that $(1+\frac{\gamma+2}{2})\geq \frac{1}{8}$, combining \eqref{eq:tightcond2}, \eqref{eq:tightcond3} and \eqref{eq:tightcond4}, we get a control of the Hölder seminorm of exponent $\frac{1}{8}$:
\begin{align*}
   \mathbb{E}\bigg[ \sup_{s\neq t \in [0,T]}\vert t-s\vert^{-\frac{1}{8}}&\left\vert \int_{\mathbb{R}^3} \varphi (d\mu^N_t-d\mu^N_s)\right\vert \bigg]\\
    &\leq C_{T,I_0,\gamma}(1+N^{-\frac{1}{2}} )(\Vert \nabla \varphi\Vert_{L^\infty}+\Vert \nabla^2 \varphi\Vert_{L^\infty}).
\end{align*}
Taking $\varphi=\varphi_n$ and summing, we get that for all $N\geq 2$,
\begin{align*}
   \mathbb{E}\left[  \sup_{s\neq t\in[0,T]} \frac{d(\mu^N_s,\mu^N_t)}{\vert s-t \vert^\frac{1}{8}} \right]\leq C_{T,I_0,\gamma},
\end{align*}
because $\Vert \nabla \varphi_n\Vert_{L^\infty}+\Vert \nabla^2 \varphi_n\Vert_{L^\infty}\leq 1$ by hypothesis. This is the desired result.
\end{proof}

We can now use this Lemma to show tightness of the empirical measures, that is to say prove Proposition~\ref{prop:tightness2}.
\begin{proof}[Proof of Proposition~\ref{prop:tightness2}]
Let $\varepsilon>0$. We need to find a compact set $\mathcal{S} \subset C\left([0,T],\mathcal{P}(\mathbb{R}^3)\right)$ such that for all $N \geq 2$, $\mathbb{P}(\mu^N \in \mathcal{S})\geq 1-\varepsilon.$

We first find a candidate for $\mathcal{S}$. Recall that for the weak topology, for any $R>0$, the set
$$S_R = \left\{ \nu \in \mathcal{P}(\mathbb{R}^3) \Bigg\vert \int_{\mathbb{R}^3}\vert v\vert^2 \nu(dv) \leq R\right\}$$
is compact in $\mathcal{P}(\mathbb{R}^3)$ as a consequence of Markov's inequality. For $R>0$, we will consider the set
\begin{equation}
\label{eq:defcompactS}
\mathcal{S}=\mathcal{S}_{R}:= \left\{ \nu=(\nu_t)_t \in C\left([0,T],\mathcal{P}(\mathbb{R}^3)\right) \Bigg \vert (\forall t \in [0,T], \nu_t \in S_R)  \text{ and } \sup_{s\neq t\in[0,T]} \frac{d(\nu_s,\nu_t)}{\vert s-t \vert^\frac{1}{8}}\leq R\right\},
\end{equation}
which is compact by the Arzèla-Ascoli theorem for continuous functions with values in a general metric space. We treat the two conditions in the two steps below.

We show that there exists $R>0$ such that $\forall t,\ \mu^N_t \in S_R$, with high probability. Thanks to Proposition~\ref{prop:wellposednessparticlesystem}, almost surely, for any $t\in [0,T]$, $$\int_{\mathbb{R}^3} \vert v\vert^2 \mu^N_t(dv)= \frac{1}{N}\sum_{i=1}^N (V^N_i(t))^2=\frac{1}{N}\sum_{i=1}^N (V^N_{i,0})^2,$$
so, since $\mathcal{L}(V^N_{i,0})=g_0$, taking the supremum over $[0,T]$ and then the expectation, we get 
$$\mathbb{E}\left[\sup_{t\in[0,T]}\int_{\mathbb{R}^3} \vert v\vert^2 \mu^N_t(dv)\right]=\int \vert v \vert^2 g_0 dv=E_0.$$
Hence, by Markov's inequality,
\begin{equation*}
\label{eq:tightnesscond1}
    \mathbb{P}\left( \sup_{t\in[0,T]}\int_{\mathbb{R}^3} \vert v\vert^2 \mu^N_t(dv) > \frac{2E_0}{\varepsilon} \right) \leq \frac{\varepsilon}{2},
\end{equation*}
so that for any $R\geq \varepsilon/(2E_0)$, $\mathbb{P}(\forall t, \mu^N_t \in S_{R})\geq 1-\frac{\varepsilon}{2}.$

We finally deal with the Hölder continuity. By Lemma~\ref{lem:holdercont},
\begin{align*}
   \mathbb{E}\left[  \sup_{s\neq t\in[0,T]} \frac{d(\mu^N_s,\mu^N_t)}{\vert s-t \vert^\frac{1}{8}} \right]\leq C_{T,I_0,\gamma}.
\end{align*}
Hence, by a similar application of Markov's inequality,
\begin{equation*}
\label{eq:tightnesscond1}
    \mathbb{P}\left( \sup_{s\neq t\in[0,T]} \frac{d(\mu^N_s,\mu^N_t)}{\vert s-t \vert^\frac{1}{8}} \leq \frac{2C_{T,I_0,\gamma}}{\varepsilon} \right) \geq 1-\frac{\varepsilon}{2},
\end{equation*}
so that, for $R\geq 2\max{(E_0,C_{T,I_0,\gamma})}/\varepsilon$, it holds that
$$\mathbb{P}\left( \mu^N \in \mathcal{S}_R \right) \geq 1-\varepsilon,$$ and the proof is concluded.
\end{proof}
Knowing the tightness of the random variables $(\mu^N)_{N\geq 2}$, Prokhorov's Theorem implies that this sequence admits convergent-in-law subsequences.
We now want to show in the next sections that any cluster point $\pi$ of $(\mathcal{L}(\mu^N))_{N\geq 2 }$ charges only the classical solution $(g_t)_t$ to the Landau equation with initial data $g_0$, \textit{i.e.} $\pi=\delta_{(g_t)_t}$.

\section{Infinite-dimension limit of the entropy and Fisher information estimates}
\label{sec:infinite-dimension limit of the regularity estimates}
We fix $\pi \in \mathcal{P}\left(C([0,T],\mathcal{P}(\mathbb{R}^3)) \right)$ a cluster point of $(\mathcal{L}(\mu^N))_{N\geq 2 }$ for the rest of the proof. This means that up to a subsequence (that we omit to write) $\mathcal{L}(\mu^N)$ converges weakly to $\pi$. The goal is to show that $\pi=\delta_{(g_t)_t}$, so that the whole sequence converges and propagation of chaos holds. Let $f=(f_t)_t$ be any $\pi$-distributed random variable. In this section we want to recover some properties of $f$ using the estimates on the law $F^N$ of the particle system that we established in Section~\ref{sec:entropy_and_fisher_in_kolmogorov}. As announced in the introduction, to show that $f$ enjoys almost surely some regularity, we will bound $\mathbb{E}(\mathcal{F}(f))$ for some regularity-quantifying functional $\mathcal{F}$ (typically, the Fisher information). To do so, we will rely on the super-additivity and affinity in infinite dimension of the functionals at play: for the entropy and Fisher information, those are known properties, and for the dissipation terms we will make use of Theorem~\ref{thm:side}.

We begin by relating more directly the cluster point $\pi$ and the law $(F^N_t)_t$. We denote by $\pi_t \in \mathcal{P}(\mathcal{P}(\mathbb{R}^3))$ the law of $f_t$, \textit{i.e.} the pushforward of $\pi$ by the evaluation map $f\mapsto f_t$. To any $\nu \in \mathcal{P}(\mathcal{P}(\mathbb{R}^3))$, we recall that we can associate a hierarchy of symmetric marginals $\nu^j \in \mathcal{P}(\mathbb{R}^{3j})$ by letting
\begin{equation}
\label{eq:defhsmarg}
    \nu^j:=\int_{\mathcal{P}(\mathbb{R}^3)} \rho^{\otimes j} \nu(d\rho).
\end{equation}
This hierarchy is compatible, in the sense that the marginals $\nu^{j:k}=\nu^k$ for any $1\leq k \leq j$. (The Hewitt-Savage theorem \cite{HewittSavage1955} asserts that such a compatible hierarchy actually \textit{uniquely characterizes} an element of $\mathcal{P}(\mathcal{P}(\mathbb{R}^3))$, but we will not need such a result.) The limit $\pi_t$ and the $F^N_t$ are linked through their respective marginals:
\begin{prop}
\label{prop:convFNj}
   For any $j\in\mathbb{N}^*$, for any $t\in[0,T]$,
   $$F^{N:j}_t \xrightharpoonup[N\rightarrow \infty]{}\pi_t^j,$$
   where $\pi_t^j$ is given by \eqref{eq:defhsmarg}, and the convergence holds in the sense of weak convergence of probability measures (along the common subsequence in $N$ for which $\mathcal{L}(\mu^N)\rightharpoonup \pi$, that we omit to write).
\end{prop}
\begin{proof}
    Let $\hat{F}^N_t\in\mathcal{P}(\mathcal{P}(\mathbb{R}^3))$ be the law of the random empirical measure $\mu^N_t$, or in other words, the pushforward of $F^N_t$ by the empirical measure map
    $(v_1,...,v_n)\mapsto \frac{1}{N}\sum_i \delta_{v_i}$.
    By \cite[Theorem 5.3, (1)]{HaurayMischler2012}, the convergence claimed in the proposition is equivalent to showing $\hat{F}^N_t \rightharpoonup \pi_t$. But since $\mathcal{L}(\mu^N) \rightharpoonup \pi$ by definition, we obtain the desired convergence by pushing forward $\mathcal{L}(\mu^N)$ and $\pi$ by the evaluation map at $t$ (which is continuous from $C([0,T],\mathcal{P}(\mathbb{R}^3))$ to $\mathcal{P}(\mathbb{R}^3)$).
\end{proof}

The first estimate we pass to the infinite-dimension limit is the energy conservation:
\begin{prop}
\label{prop:lvl3energy}
    For any $t\in[0,T]$, $\pi_t$ has finite mean second moment, bounded by the initial energy:
    $$\int_{\mathbb{R}^3} \vert v \vert^2 \pi^1_t(dv) \leq E_0.$$
\end{prop}
\begin{proof}
    Using Proposition~\ref{prop:kol_conservations} and symmetry, for any $t$
    $$\int_{\mathbb{R}^3} \vert v \vert^2 F^{N:1}_t(dv)=\frac{1}{N}\int_{\mathbb{R}^3} \vert V \vert^2 F^{N}_t(dV)\leq E_0.$$
    The portemanteau theorem applied to the bounded below and lower semi-continuous function $v \mapsto \vert v \vert^2$ yields the result since $F^{N:1}_t \rightharpoonup\pi_t^1$ by Proposition~\ref{prop:convFNj}.
\end{proof}

We now want to pass to the limit in the entropy, the Fisher information and their production/dissipation terms. We begin by defining functionals for the dissipation terms. First, for a symmetric non-negative function $F$ on $\mathbb{R}^{3N}$ ($N\geq 2$), we let the entropy production be:
\begin{equation}
\label{eq:defDn}
    D^N(F):= \frac{1}{2}\int_{\mathbb{R}^{3N}} \alpha^N a(v_1-v_2):\left[(\nabla_1-\nabla_2)\log{F}\right]^{\otimes 2} \ F dV.
\end{equation}
Hence Proposition~\ref{prop:entropy_dissipation} exactly rewrites:
\begin{equation}
\label{eq:entropy_dissipation_rewrite}
    H(F^N_t)+\int_0^t D^N(F^N_s)ds \leq H_0.
\end{equation}
We also define $D$ as $D^N$ but with $\alpha$ instead of $\alpha^N$, that is, without cut-off. Remark that the functional $D^N$ is obviously monotone with respect to $\alpha^N$.

We want to do the same for the dissipation of Fisher information. Since the expression is a little more convoluted, we first introduce the following derivation operator:
\begin{equation}
\label{eq:defpartialNbk}
    \partial^N_{b_k}:=\vert v_1-v_2\vert^{-\frac{1}{2}} (\alpha^N(\vert v_1-v_2\vert ))^{\frac{1}{4}}\tilde{b}_k(v_1,v_2)\cdot \nabla_{12}.
\end{equation}
The vector-field in this derivation is still divergence-free (as a direct consequence of~\eqref{eq:bkandradial}). We also define the operator $\partial_{b_k}$ as $\partial^N_{b_k}$ but with $\alpha$ instead of $\alpha^N$. With this derivation operator we can rewrite the Fisher dissipation estimate~\eqref{eq:fisher_dissipation} in a more concise form as
\begin{equation}
\label{eq:fisher_dissip_rewrite1}
    I(F^N_t)+\frac{1}{4}\sum_{k=1}^3\int_0^t  \int_{\mathbb{R}^{3N}}  \left(\partial^N_{b_k} \partial^N_{b_k}\log F^N_s\right)^2F^N_s dVds \leq I_0,
\end{equation}
(using~\eqref{eq:bkandradial} to freely move $\vert v_1-v_2\vert^{-2} \alpha^N(\vert v_1-v_2\vert )$ in and out of the derivations). We see that the second term is the sum over $k$ of the second order Fisher information $K_0^{\partial^N_{b_k}}$ that we studied in the preliminary Section~\ref{sec:on the fisher dissipation and aux}. We rather want to use $K^{\partial^N_{b_k}}=K_1^{\partial^N_{b_k}}$ because we know it will be l.s.c., convex and super-additive. We hence define
\begin{align}
\label{eq:defK}
    \mathcal{K}^N(F):=\sum_{k=1}^3 K_1^{\partial^N_{b_k}}(F)&=\sum_{k=1}^3 \int_{\mathbb{R}^{3N}} \frac{\left( \partial^N_{b_k} \partial^N_{b_k}F \right)^2}{F} dV\\
    &=\sum_{k=1}^3 \int_{\mathbb{R}^{3n}} \frac{\alpha^n(v_1-v_2)}{\vert v_1-v_2 \vert^2} \frac{\left( (\tilde{b}_k\cdot \nabla_{12})( \tilde{b}_k\cdot \nabla_{12})F^n \right)^2}{F^n} dV.\nonumber
\end{align}
We also define a version without cut-off, $K$, with $\alpha$ instead of $\alpha^N$ (\textit{i.e.} with $\partial_{b_k}$ instead of $\partial^N_{b_k}$).
By applying the equivalence result in Proposition~\ref{prop:func_comparison} with $\beta=0$ and $\partial=\partial^N_{b_k}$, we have
$$\frac{1}{4}\mathcal{K}^N(F) \leq \sum_{k=1}^3 K_0^{\partial^N_{b_k}}(F) = \sum_{k=1}^3 \int_{\mathbb{R}^{3N}}  \left(\partial^N_{b_k} \partial^N_{b_k}\log F\right)^2F dV \leq \mathcal{K}^N(F) . $$
Using the left hand side, the Fisher dissipation estimate~\eqref{eq:fisher_dissip_rewrite1} implies:
\begin{equation}
\label{eq:fisher_dissipation_rewrite}
    I(F^N_t)+\frac{1}{16}\int_0^t \mathcal{K}^N(F^N_s)ds \leq I_0.
\end{equation}

\begin{remark}
    In all rigour, we cannot directly apply Proposition~\ref{prop:func_comparison} because the vector field has the singularity $\vert v_1-v_2\vert^{-\frac{1}{2}}$. This is not a real problem because we can approximate from below the singularity by a smooth function, apply the proposition to obtain the inequality and then let the approximation vanish. We do not detail it because we will do the same in the proof of Proposition~\ref{prop:convexlscsuperadditve} below.
\end{remark}
Now that are estimates are rewritten with the four functionals $H$, $I$, $D^N$ a,d $\mathcal{K}^N$, we want to pass them to the $N\rightarrow \infty$ limit. As explained in introduction, the two key properties to do so are \textit{super-additivity} and \textit{infinite-dimensional affinity}. This is known for the entropy and Fisher information, and using Theorem~\ref{thm:side} we can obtain \textit{super-additivity} of the dissipation terms:
\begin{prop}
\label{prop:convexlscsuperadditve}
    For any $N\geq 2$, the functionals $D^N$ and $\mathcal{K}^N$ can be defined for any symmetric probability measure $F\in\mathcal{P}(\mathbb{R}^{3N})$, and are convex and lower semi-continuous for the weak convergence of probability measures. They are super-additive in the following sense: for any $2\leq j \leq N$, if $F^{:j}$ is the $j$-th marginal of a symmetric $F\in \mathcal{P}(\mathbb{R}^{3N})$, then
    \begin{align*}
        D^{j}(F^{:j}) &\leq D^{N}(F)\\
        \mathcal{K}^{j}(F^{:j}) &\leq \mathcal{K}^{N}(F).
    \end{align*}
\end{prop}
\begin{proof}
    We begin with $D^N$. Define the vector field
    \begin{equation}
    \label{eq:defbbark}
        \bar{b}_k(v_1,v_2)=\sqrt{\alpha^N(\vert v_1-v_2\vert)}\tilde{b}_k(v_1,v_2),
    \end{equation}
    which is smooth, bounded and divergence-free.
    Then
    $$D^N=\frac{1}{2}\sum_{k=1}^3 I^{\bar{b}_k\cdot\nabla_{12}}.$$
    We can hence apply Theorem~\ref{thm:side}. By point (1), we get the claimed lower semi-continuity and convexity. By point (2), we get super-additivity for a fixed $N$:
    $D^N(F)\geq D^N(F^{:j})$. But $ D^N(F^{:j})\geq  D^j(F^{:j})$ since $\alpha^N\geq \alpha^j$ (this is clear from the definition \ref{eq:defDn}).
    
    We now turn to $\mathcal{K}^{N}$. We cannot apply Theorem~\ref{thm:side} directly because the vector field in $\partial^N_{b_k}$ features the singularity $\vert v_1-v_2\vert^{-\frac{1}{2}}$. We instead consider a smooth bounded cut-off function $\chi$ such that $0\leq \chi(r)\leq r^{-2}$ for all $r\geq 0$. We define the smooth, bounded and divergence-free
    $$\partial^{N,\chi}_{b_k}:=\chi(\vert v_1-v_2\vert) (\alpha^N(\vert v_1-v_2\vert ))^{\frac{1}{4}}\tilde{b}_k(v_1,v_2)\cdot \nabla_{12},$$
    and now we can apply Theorem~\ref{thm:side} to each term in
    $$\mathcal{K}^{N,\chi}=\sum_{k=1}^3 K^{\partial^{N,\chi}_{b_k}}.$$
    We remark that
    $$\mathcal{K}^N=\sup_{\chi\leq (\cdot)^{-2}}\mathcal{K}^{N,\chi}$$
    by monotone convergence, so that $\mathcal{K}^N$ is l.s.c and convex because the $\mathcal{K}^{N,\chi}$ are. From point (2) applied to $\mathcal{K}^{N,\chi}$ and using $\alpha^N\geq \alpha^j$ again, we get
    $\mathcal{K}^{N,\chi}(F)\geq \mathcal{K}^{j,\chi}(F^{:j})$ and then take the supremum in $\chi$ to get the claimed super-additivity. .
\end{proof}

\begin{remark}
    Using the same regularization trick, we can also show the property also holds for the versions without cut-off, $\mathcal{K}$ and $D$, by letting $\alpha^N \nearrow\alpha$.
\end{remark}

As discussed in the introduction, the second property that goes hand in hand with super-additivity is \textit{affinity in infinite dimension}. By using point (3) of Theorem~\ref{thm:side}, we can get:

\begin{prop}
    \label{thm:level3affinity}
    Let $\nu\in \mathcal{P}(\mathcal{P}(\mathbb{R}^3))$ have finite mean Fisher information and mean second moment, \textit{i.e.} such that
    \begin{align*}
        \sup_{j\geq 1} I(\nu_j) < +\infty,&&\int_{\mathbb{R}^3} \vert v \vert^2 \nu_1(dv) < +\infty.
    \end{align*}
    (Recall that the $\nu_j$s are defined in \eqref{eq:defhsmarg}.) Then it holds that
    \begin{align*}
        \int_{\mathcal{P}(\mathbb{R}^3)} H(\rho) \nu(d\rho) &=\lim_{j\geq 1} H(\nu^j),\\
        \int_{\mathcal{P}(\mathbb{R}^3)} I(\rho) \nu(d\rho) &=\lim_{j\geq 1} I(\nu^j),\\
        \int_{\mathcal{P}(\mathbb{R}^3)} D(\rho \otimes \rho) \nu(d\rho) &=\lim_{j\geq 2} D^j(\nu^j),\\
        \lim_{j\geq 2} \mathcal{K}^j(\nu^j)\leq \int_{\mathbb{P}(\mathbb{R}^3)} \mathcal{K}(\rho \otimes \rho) \nu(d\rho) &\leq 16 \lim_{j\geq 2} \mathcal{K}^j(\nu^j), 
    \end{align*}
    and all the limits are along increasing-in-$j$ sequences.
\end{prop}
\begin{proof}
    The theorem is already known to hold for the entropy \cite{RobinsonRuelle1967} and the Fisher information \cite{HaurayMischler2012,Rougerie2020}. The proof for $D$ and $\mathcal{K}$ is essentially an application of point (3) in Theorem~\ref{thm:side}. We begin with $D$. Using the vector fields $\bar{b}_k$ defined in \eqref{eq:defbbark}, we can apply it to $I^{\bar{b}_k\cdot\nabla_{12}}$ for $k=1,2,3$ and sum to get, for any $N$:
    $$\int_{\mathcal{P}(\mathbb{R}^3)} D^N(\rho \otimes \rho) \nu(d\rho) =\lim_{j\geq 2} D^N(\nu^j).$$
    We then remark that, because the sequence $D^N(\nu^j)$ is increasing in both $j$ and $N$,
    $$\lim_{j\geq 2} D^j(\nu^j)=\lim_{j\geq 2}\lim_{N\geq 2} D^N(\nu^j),$$
    so taking the limit in $N$ of the result above,
    $$\lim_{j\geq 2} D^j(\nu^j)=\lim_{N\geq 2}\int_{\mathcal{P}(\mathbb{R}^3)} D^N(\rho \otimes \rho) \nu(d\rho)=\int_{\mathcal{P}(\mathbb{R}^3)} D(\rho \otimes \rho) \nu(d\rho)$$
    by monotone convergence under all the integrals, since $\alpha^N \nearrow \alpha$. The proof for $\mathcal{K}$ is the same, dealing with the singularity in the derivation $\partial^N_{b_k}$ as in the previous proposition: first proving the version with a cut-off $\chi(r)$ and then taking the supremum in $\chi$.
\end{proof}
Using Properties~\ref{prop:convexlscsuperadditve} and \ref{thm:level3affinity}, we can pass our estimates on the $N$-particle system to the infinite-dimension limit, to obtain the following \textit{mean} (or \textit{level-3}) estimates:

\begin{prop}
    \label{prop:level3estimates}
     The following two estimates hold for all $t\in[0,T]$:
    \begin{equation}
    \label{eq:entropydissipation_pi}
    \int_{\mathbb{P}(\mathbb{R}^3)} H(\rho) \pi_t(d\rho) + \int_0^t \int_{\mathbb{P}(\mathbb{R}^3)}D(\rho \otimes \rho) \pi_s(d\rho) ds \leq H_0,
    \end{equation}
    and
    \begin{equation}
    \label{eq:fisherdissipation_pi}
    \int_{\mathbb{P}(\mathbb{R}^3)} I(\rho) \pi_t(d\rho) + \frac{1}{256} \int_0^t \int_{\mathbb{P}(\mathbb{R}^3)}\mathcal{K}(\rho \otimes \rho) \pi_s(d\rho) ds \leq I_0.
    \end{equation}
\end{prop}

\begin{proof}
    The proof is identical for both estimates, starting respectively from~\eqref{eq:entropy_dissipation_rewrite} and~\eqref{eq:fisher_dissipation_rewrite}. We prove the second one.
    Let $j\leq N$, by super-additivity (see Proposition~\ref{prop:convexlscsuperadditve}) we have:
    $$I(F^{N:j}_t)+\frac{1}{16}\int_0^t \mathcal{K}^j(F^{N:j}_s)ds \leq I_0$$
    We take the $\liminf$ in $N$ (along the common subsequence for which $F^{N:j}_s\rightharpoonup \pi^j_s$ for all $s$), together with Fatou's lemma and lower semi-continuity of $I$ and $K^j$, to obtain:
    $$I(\pi^{j}_t)+\frac{1}{16}\int_0^t \mathcal{K}^j(\pi^{j}_s)ds \leq I_0.$$
    We take the limit in $j$, by monotone convergence in the integral (since $(\mathcal{K}^j(\pi^{j}_s))_j$ is increasing by super-additivity), we get
    $$\lim_jI(\pi^{j}_t)+\frac{1}{16}\int_0^t \lim_j \mathcal{K}^j(\pi^{j}_s)ds \leq I_0.$$
    Using Proposition~\ref{thm:level3affinity}, we obtain the result.    
\end{proof}

We can formulate these estimates in terms of random variables rather than laws, leading to some almost sure regularity of $\pi$-distributed random variables:
\begin{prop}
    \label{prop:level3estimatesforf}
    Recall that $\pi \in \mathcal{P}\left(C([0,T],\mathcal{P}(\mathbb{R}^3)) \right)$ is a cluster point of $\left(\mathcal{L}\left( \mu^N\right) \right)_{N\geq 2 }$. Let $f$ be a $\pi$-distributed random variable. Then, almost surely, the following quantities are finite:
    \begin{align*}
    \begin{array}{cccc}
        \int_0^T H(f_t)dt, &  \int_0^T I(f_t)dt, & \int_0^T D(f_t\otimes f_t)dt, & \int_0^T \mathcal{K}(f_t\otimes f_t)dt.
    \end{array}
    \end{align*}
\end{prop}
\begin{proof}
    We can rewrite~\eqref{eq:entropydissipation_pi} as
    \begin{align}
    \label{eq:expectofHtheroem}
    \mathbb{E}\left[H(f_t)+\int_0^t D(f_s\otimes f_s)ds\right]\leq H_0.
    \end{align}
    It is classical \cite[Lemma 3.1]{HaurayMischler2012} that the energy $E(f)=\int \vert v \vert^2f(dv)$ bounds the entropy from below: for some $C>0$,
    $$H(f)\geq -C(1+E(f)).$$
    Hence $\mathbb{E}[H(f_T)]\geq -C(1+\mathbb{E}[E(f_T)])\geq -C(1+E_0)$ by Proposition~\ref{prop:lvl3energy} and Remark~\ref{rem:lvl3energy}. Hence $\mathbb{E}\left[\int_0^T D(f_s\otimes f_s)ds\right]$ is finite, meaning that the quantity inside the expectation is almost surely finite.
    
    Similarly
    $\int_0^T \mathcal{K}(f_t\otimes f_t)dt$ is a.s. finite by using~\eqref{eq:fisherdissipation_pi} and that $I(f_T)\geq 0$. For $I$ and $H$, we integrate over $[0,T]$ the estimates~\eqref{eq:entropydissipation_pi} and~\eqref{eq:fisherdissipation_pi}, and drop the nonnegative dissipation terms.
\end{proof}
\begin{remark}
    The equation~\eqref{eq:expectofHtheroem} is exactly the H-theorem in expectancy: $H(f_t)+\int_0^t D(f_s\otimes f_s)ds \leq H_0$ is the formal apriori estimate on the entropy that stems from the Landau equation~\eqref{eq:landau}.
\end{remark}
\begin{remark}
    Even if $\mathbb{E}(H(f_t))$ (as well as $\mathbb{E}(I(f_t))$) are bounded in time, we have little hope to recover $(H(f_t))_t\in L^\infty([0,T])$ because it would require bounding $\mathbb{E}(\sup_{[0,T]} H(f_t))$, with the supremum inside the expectation, a quantity that we cannot easily relate to the particle system. A lot of the difficulties we will encounter when showing that $(f_t)_t$ is the strong solution $(g_t)_t$ stem from this issue.
\end{remark}
In the next section, we will use Proposition~\ref{prop:level3estimatesforf} for the Fisher information to show that $f$ is almost surely a weak solution of the Landau equation. The last section will use the estimates on $D$ and $\mathcal{K}$ to (essentially) show that $f\in L^1([0,T],L^\infty(\mathbb{R}^3))$, which will allow us to claim uniqueness.

\section{The cluster points only charge almost sure weak solutions}
\label{sec:almostsureweaksolution}
Recall that $\pi \in \mathcal{P}\left(C([0,T],\mathcal{P}(\mathbb{R}^3)) \right)$ is a cluster point of $\left(\mathcal{L}\left( \mu^N\right) \right)_{N\geq 2 }$. In this section we show that a $\pi$-distributed random variable is almost surely a weak solution of the Landau equation.

We make our notion of weak solution precise. This is particularly important as we will manipulate another formulation later on.

\begin{defi}
\label{def:weaksolution}
    A \textbf{weak solution} to the Landau equation on $[0,T]$ with initial data $g_0$ is a function $$f=(f_t)_t\in L^1([0,T], L^p(\mathbb{R}^3)) \cap L^\infty([0,T], L^1(\mathbb{R}^3))$$
    for some $p>\frac{3}{\gamma + 5}$ such that, for any $\varphi\in C^2_b(\mathbb{R}^3)$, for any $t\in[0,T]$:
    \begin{align}
    \label{eq:weakform}
     \int_{\mathbb{R}^3} \varphi(v)f_t(v)dv &= \int_{\mathbb{R}^3} \varphi(v)g_0(v)dv\nonumber\\
    &+ \int_0^t \int_{\mathbb{R}^6} b(v-w) \cdot (\nabla\varphi(v)-\nabla\varphi(w))\ f_s(dv)f_s(dw)\ ds\\
    &+ \int_0^t \int_{\mathbb{R}^6} \alpha a(v-w):\nabla^2\varphi(v)\ f_s(dv)f_s(dw)\ ds,\nonumber
    \end{align}
where we recall that $b(z)=\nabla \cdot (\alpha a(z))=-2\alpha(\vert z\vert) z$.
\end{defi}
\begin{remark}
\label{rem:defwsol}
    This definition is mainly taken from Guerin and Fournier \cite[Definition 1.2]{GuerinFournier2008}, because we want in the end to apply their uniqueness result \cite{GuerinFournier2008} for $\gamma\in(-3,-2)$, and \cite{Fournier2010} for $\gamma=-3$. The only difference is that we symmetrized the $b$ term with respect to the test function: we replaced $\nabla\varphi(v)$ by $(\nabla\varphi(v)-\nabla\varphi(w))$. One can check that it allows one to make sense of every term in~\eqref{eq:weakform} under the integrability condition we impose on $f$ (it will be done in the proof of Lemma~\ref{aswsolution-feps f limit} below). Without symmetry, one needs $p>\frac{3}{\gamma + 4}$ to make sense of the $b$ term, which we do not know yet at this point of the proof in the Coulomb case. However, when we will apply the uniqueness result \cite{GuerinFournier2008,Fournier2010} at the end of this work, we will have already shown that $L^1([0,T], L^\infty(\mathbb{R}^3))$, so that the unsymmetrized formulation is also properly defined.
    
    In \cite[Definition 1.2]{GuerinFournier2008}, the authors also require the energy to remain bounded, which is in fact not needed for $\gamma\leq -2$. However, we will show later that the conservation of energy holds anyway (this is done in Proposition~\ref{prop:propofenergy}).
\end{remark}

The goal of this section is to prove the following:
\begin{prop}
\label{prop:aswsolution}
    Let $f=(f_t)_t$ be a $\pi$-distributed random variable. Then, almost surely,
    $$f\in L^1([0,T], L^3(\mathbb{R}^3))$$
    and $f$ is a weak solution of the Landau equation with initial data $g_0$, in the sense of Definition~\ref{def:weaksolution}.
\end{prop}
We fix a $\pi$-distributed random variable $f$ and cut the proof of Proposition~\ref{prop:aswsolution} into several lemmas. We begin by showing the claimed regularity of $f$:
\begin{lemma}
    \label{lem:aswsolution-regularity}
    Almost surely, 
    $$f\in L^1([0,T], L^3(\mathbb{R}^3)) \cap L^\infty([0,T], L^1(\mathbb{R}^3)).$$
\end{lemma}
\begin{proof}
    By Proposition~\ref{prop:level3estimatesforf}, almost surely, the Fisher information of $f$ is integrable in time: $I(f_t)\in L^1([0,T])$. This means that for almost all $t$, $f_t$ has finite Fisher information. It hence admits a density (still denoted $f_t$). Since $f_t$ is a probability measure for all $t$, this implies $f\in L^\infty([0,T], L^1(\mathbb{R}^3))$.

    Since $I(f)=4\Vert \nabla \sqrt{f} \Vert_{L^2(\mathbb{R}^3)}$, by the Sobolev embedding $ H^1(\mathbb{R}^3) \hookrightarrow L^6(\mathbb{R}^3)$ in three dimensions, we have
    $$f\in L^1([0,T], L^3(\mathbb{R}^3)),$$
    as desired.
\end{proof}

We then deal with the initial time.
\begin{lemma}
    \label{lem:aswsolution-initial time}
    Almost surely, $f_0=g_0$.
\end{lemma}
\begin{proof}
    \textit{Step 1.} We show that the random initial empirical measures $\mu^N_0$ converge in law to the deterministic $g_0$.
    Recall that $\mathcal{L}(V^N_{1,0},...,V^N_{N,0})=g_0^{\otimes N}$. Let $\varphi$ be a continuous bounded function on $\mathbb{R}^3$, the random variable
    $$\int_{\mathbb{R}^3} \varphi(v)\mu^N_0(dv)=\frac{1}{N}\sum_{i=0}^N \varphi(V^N_{i,0})$$
    is the sum of $N$ bounded, independent, identically distributed random variables. By the Strong Law of Large numbers, it converges almost surely to
    $$\mathbb{E}(\varphi(V^N_{1,0}))=\int_{\mathbb{R}^3} \varphi(v)g_0(dv).$$
    A priori, the event of probability $1$ on which  $\int_{\mathbb{R}^3} \varphi(v)\mu^N_0(dv) \rightarrow \int_{\mathbb{R}^3} \varphi(v)g_0(dv)$ depends on $\varphi$, but it can be made independent of $\varphi$ by a standard trick that we omit here because it will be showcased in the proof of Proposition~\ref{prop:aswsolution} below.

    \textit{Step 2.} Recall that, up to a subsequence, $(\mu^N_t)_t$ converges in law to $(f_t)_t$. By continuity of the evaluation map $$C([0,T],\mathcal{P}(\mathbb{R}^3))\ni (f_t)_t\mapsto f_0 \in\mathcal{P}(\mathbb{R}^3),$$
    $\mu^N_0$ converges in law to $f_0$ along the same subsequence. So by Step 1, the law $\pi_0$ of $f_0$ is equal to $\delta_{g_0}$.
\end{proof}

We now turn to the weak formulation~\eqref{eq:weakform} itself. Consider $\varphi\in C^2_b(\mathbb{R}^3)$ a test function and $t\in[0,T]$. We formally define a functional $\mathcal{F}_{\varphi,t}$ which tests against $\varphi$: for $h=(h_t)_t\in C([0,T],\mathcal{P}(\mathbb{R}^3))$, under the assumption that the terms below are well-defined:
    \begin{align}
    \label{eq:def_f}
    \mathcal{F}_{\varphi,t}(h) &:=  -\int_{\mathbb{R}^3} \varphi(v)h_t(v)dv + \int_{\mathbb{R}^3} \varphi(v)h_0(v)dv \nonumber\\
    &+ \int_0^t \int_{\mathbb{R}^6} \mathbf{1}_{v\neq w}b(v-w) \cdot (\nabla\varphi(v)-\nabla\varphi(w))\ h_s(dv)h_s(dw) ds\\
    &+ \int_0^t \int_{\mathbb{R}^6} \mathbf{1}_{v\neq w}\alpha a(v-w):\nabla^2\varphi(v)\ h_s(dv)h_s(dw) ds.\nonumber
    \end{align}
Since $\alpha$ and $b$ are unbounded, the hypothesis $h_t \in \mathcal{P}(\mathbb{R}^3))$ alone is not enough to make sense of every term above.  However $\mathcal{F}$ will be well-defined for empirical measures, thanks to the indicator functions.

We use that the particle system approximates the Landau equation to show:
\begin{lemma}
    \label{aswsolution-empirical measures}
    For any $\varphi\in C^2_b(\mathbb{R}^3)$ and any $t\in[0,T]$,
    $$\mathbb{E}\vert \mathcal{F}_{\varphi,t}(\mu^N)\vert\xrightarrow[N\rightarrow+\infty]{}0.$$
\end{lemma}
\begin{proof}
    \textit{Step 1.} Using that
    $$\int \varphi(v)\mu^N_t(dv)=\frac{1}{N}\sum_{i=1}^N\varphi(V^N_i(t))$$
    by definition of empirical measures, we have:
    \begin{align}
    \label{eq:Fagainstempiricalmeasures1}
    \mathcal{F}_{\varphi,t}&(\mu^N) = -\frac{1}{N} \sum_{i=1}^N\varphi(V^N_i(t)) + \frac{1}{N} \sum_{i=1}^N\varphi(V^N_i(0))\\
    \label{eq:Fagainstempiricalmeasures2}
    &+ \frac{1}{N^2}\int_0^t \sum_{\substack{i,j=1\\ i\neq j}}^N \mathbf{1}_{V^N_i(s)\neq V^N_j(s)}b(V^N_i(s)-V^N_j(s)) \cdot (\nabla\varphi(V^N_i(s))-\nabla\varphi(V^N_j(s)))\ ds\\
    \label{eq:Fagainstempiricalmeasures3}
    &+\frac{1}{N^2} \int_0^t \sum_{\substack{i,j=1\\ i\neq j}}^N \mathbf{1}_{V^N_i(s)\neq V^N_j(s)}\alpha a(V^N_i(s)-V^N_j(s)) \cdot \nabla^2\varphi(V^N_i(s))\ ds.
    \end{align}  
    Note that the cases $i=j$ are excluded by the indicator functions. We now make use that the particle system approximates the Landau equation. Indeed, applying the Itô formula to compute the differential of $t \mapsto \frac{1}{N}\sum_{i=1}^N\varphi(V^N_i(t))$ (as we did at the beginning of the proof of Lemma~\ref{lem:holdercont}), we obtain the following identity, which closely resembles $\mathcal{F}(\mu^N)$:
    \begin{align}
    \label{eq:Itôempiricalmeasures1}
       0 &= -\frac{1}{N}\sum_{i=1}^N\varphi(V^N_i(t))+\frac{1}{N}\sum_{i=1}^N\varphi(V^N_i(0)) \\
       \label{eq:Itôempiricalmeasures2}
       &+ \frac{1}{N(N-1)}\int_0^t \sum_{\substack{i,j=1\\ i\neq j}}^N  b^N\left(V^N_i(s) -V^N_j(s)\right) \cdot (\nabla\varphi\left(V^N_i(s)\right)-\nabla\varphi\left(V^N_j(s)\right)) ds\ \\
       \label{eq:Itôempiricalmeasures3}
       &+ \frac{1}{N(N-1)} \int_0^t \sum_{\substack{i,j=1\\ i\neq j}}^N  \alpha^Na\left(V^N_i(s) -V^N_j(s)\right):\nabla^2\varphi\left(V^N_i(s)\right)ds\\
       \label{eq:Itôempiricalmeasures4}
       &+\sqrt{\frac{2}{N-1}}\frac{1}{N}\sum_{\substack{i,j=1\\ i\neq j}}^N \int_0^t  \nabla \varphi\left(V^N_i(s)\right)\cdot \sigma^N \left(V^N_i(s) - V^N_j(s)\right)dB^N_{ij}(s)
    \end{align}
    We compare each line. Remark that the first lines~\eqref{eq:Fagainstempiricalmeasures1} and~\eqref{eq:Itôempiricalmeasures1} are the same.
    
    \textit{Step 2.} We compare the lines~\eqref{eq:Fagainstempiricalmeasures2} and~\eqref{eq:Itôempiricalmeasures2}, that contain the terms in $b$, which are the most singular. With the shorthand notations $\psi^N_{ij}(s):=\nabla\varphi(V^N_i(s))-\nabla\varphi(V^N_j(s))$, and $\mathbf{1}b(v-w):=\mathbf{1}_{v-w}b(v-w)$ we study the difference
    \begin{align*}
        \frac{1}{N(N-1)}&\int_0^t \sum_{\substack{i,j=1\\ i\neq j}}^N  b^N\left(V^N_i(s) -V^N_j(s)\right) \cdot \psi^N_{ij}(s) ds\\
        &-\frac{1}{N^2}\int_0^t \sum_{\substack{i,j=1\\ i\neq j}}^N \mathbf{1}b(V^N_i(s)-V^N_j(s)) \cdot \psi^N_{ij}(s)\ ds\\=
        &\frac{1}{N(N-1)} \int_0^t \sum_{\substack{i,j=1\\ i\neq j}}^N  (b^N-\mathbf{1}b)\left(V^N_i(s) -V^N_j(s)\right) \cdot \psi^N_{ij}(s) ds \ + \mathcal{R}
    \end{align*}
    with the remainder
    $$\mathcal{R}= \frac{1}{N^2(N-1)} \int_0^T \sum_{\substack{i,j=1\\ i\neq j}}^N \mathbf{1}b(V^N_i(t)-V^N_j(t)) \cdot \psi^N_{ij}(t)\ dt, $$
    which we quickly deal with.
    The estimate $\mathbf{1}b(v)\lesssim \vert v \vert^{\gamma+1}$, the boundedness of $\nabla\varphi$, as well as exchangeability allow us to easily control:
    $$\mathbb{E}\vert \mathcal{R} \vert \leq   \frac{2\Vert \nabla\varphi\Vert_{L^\infty}}{N^2(N-1)} \int_0^t \sum_{\substack{i,j=1\\ i\neq j}}^N \mathbb{E}\vert V^N_i(s)-V^N_j(s) \vert^{\gamma+1}\ ds \leq \frac{2\Vert \nabla\varphi\Vert_{L^\infty}}{N} \int_0^t \mathbb{E}\vert V^N_1(s)-V^N_2(s) \vert^{\gamma+1}\ ds$$
    Using $\vert v \vert^{\gamma+1} \leq 1 + \vert v \vert^{-2}$ and Lemma~\ref{lem:keyestimate} to bound the expectation we get
    $$\mathbb{E}\vert \mathcal{R} \vert \leq \frac{C_{\varphi,I_0,t}}{N}.$$
    Now for the main term, we first write a pointwise estimate, using that $b^N(v)$ and $b(v)$ coincide for $\vert v\vert\geq \eta_N$:
    $$\vert (b^N - \mathbf{1}b)(v)\vert \leq 4 \mathbf{1}_{\vert v \vert \leq \eta_N}\vert v \vert^{\gamma + 1}  \leq 4 \vert v \vert^{-3} \eta_N^{\gamma+4}.$$
    Notice that the above bound also holds for $v=0$ because $b^N(0)=0$, since $b^N(v)=-2\alpha^N(\vert v\vert)v$.
    We rely on the $C^2$ behaviour of $\varphi$ to write
    $$\vert \psi^N_{ij}(s)\vert=\vert\nabla\varphi(V^N_i(s))-\nabla\varphi(V^N_j(s))\vert \leq C_\varphi\vert V^N_i(s)-V^N_j(s)\vert.  $$
    With that estimate we cancel one order of the singularity and we have the following bound:
    \begin{align*}
        \frac{2}{N(N-1)} &\mathbb{E}\left\vert \int_0^t \sum_{\substack{i,j=1\\ i\neq j}}^N  (b^N-\mathbf{1}b)\left(V^N_i(s) -V^N_j(s)\right)  \cdot \psi^N_{ij}(s) ds\right\vert \\
        &\leq \frac{\eta_N^{\gamma+4}C_\varphi}{N(N-1)} \sum_{\substack{i,j=1\\ i\neq j}}^N\int_0^t \mathbb{E} \left[\vert V^N_i(s) -V^N_j(s)\vert^{-2}\right]   ds\\
        &\leq \eta_N^{\gamma+4}C_\varphi \int_0^t \mathbb{E} \left[\vert V^N_1(s) -V^N_2(s)\vert^{-2}\right]   ds\\
        &\leq  \eta_N^{\gamma+4} C_{\varphi,I_0,t}
    \end{align*}
    where the third line is obtained by exchangeability and the last line thanks to Lemma~\ref{lem:keyestimate} to bound the expectation.

    \textit{Step 3.}
    We compare~\eqref{eq:Fagainstempiricalmeasures3} and~\eqref{eq:Itôempiricalmeasures3}, exactly as in the previous step.
    This time we use the two bounds
    $$\Vert\mathbf{1}a(v)\Vert \leq C \vert v \vert^{\gamma+2} \leq C(1 + \vert v \vert^{-2})$$ 
    and
    $$\Vert a^N - \mathbf{1}a(v)\Vert \leq C \mathbf{1}_{\vert v \vert \leq \eta_N}\vert v \vert^{\gamma + 2}  \leq C \vert v \vert^{-2} \eta_N^{\gamma+4}$$
    (which still holds for $v=0$ because $\alpha^N a(0)=0$ since $a(0)=0$.)
    They readily give:
    \begin{align*}
    \mathbb{E}\Bigg\vert &\frac{1}{N^2} \int_0^t \sum_{\substack{i,j=1\\ i\neq j}}^N \mathbf{1}_{V^N_i(s)\neq V^N_j(s)}\alpha a(V^N_i(s)-V^N_j(s)) \cdot \nabla^2_v\varphi(V^N_i(s))\ ds\\
    &- \frac{1}{N(N-1)} \int_0^T \sum_{\substack{i,j=1\\ i\neq j}}^N  \alpha^Na\left(V^N_i(s) -V^N_j(t)\right):\nabla^2_v\varphi\left(V^N_i(s)\right)ds \Bigg\vert\leq C_{\varphi,I_0,t}\left(\eta_N^{\gamma+4}+\frac{1}{N}\right) .
    \end{align*}

    \textit{Step 4.} It remains to control the martingale term~\eqref{eq:Itôempiricalmeasures4}, that we rewrite using $B^N_{ji}=-B^N_{ij}$:
    \begin{align*}
        \mathcal{B} &:= \sqrt{\frac{2}{N-1}}\frac{1}{N}\sum_{\substack{i,j=1\\ i\neq j}}^N \int_0^t  \nabla_v \varphi\left(V^N_i(s)\right)\cdot \sigma^N \left(V^N_i(s) - V^N_j(s)\right)dB^N_{ij}(s)\\
    &= \sqrt{\frac{2}{N^2(N-1)}}\sum_{\substack{i,j=1\\ i< j}}^N \int_0^t  \xi^N_{ij}(s)\cdot \sigma^N \left(V^N_i(s) - V^N_j(s)\right)dB^N_{ij}(s)
    \end{align*}
    since $\sigma^N(-v)=\sigma^N(v)$, with $\psi^N_{ij}(s) = \nabla\varphi(V^N_i(s))-\nabla \varphi(V^N_j(s))$ as before, .
    
    We use independence of the Brownian motions $(B^N_{ij})_{1\leq i<j \leq N}$ to cancel cross terms in the square expectation:
    \begin{align*}
        (\mathbb{E}\vert \mathcal{B}\vert)^2 &\leq\mathbb{E} (\mathcal{B}^2)\\
        &= \frac{2}{N^2(N-1)} \sum_{\substack{i,j=1\\ i< j}}^N \mathbb{E}\left[\int_0^t  \psi^N_{ij}(s)\cdot \sigma^N \left(V^N_i(s) - V^N_j(s)\right)dB^N_{ij}(s) \right]^2
    \end{align*}
    By the Iso isometry we can let the square inside the integral, and we obtain:
    \begin{align*}
        \mathbb{E}\left[\int_0^t  \psi^N_{ij}(s)\cdot \sigma^N \left(V^N_i(s) - V^N_j(s)\right)dB^N_{ij}(s) \right]^2 &\leq C_\varphi \int_0^t\mathbb{E}\left[ \Vert\sigma^N \left(V^N_i(s) - V^N_j(s)\right)\Vert^2 \right]ds
    \end{align*}
    We conclude as before, with the bound $\Vert\sigma^N \left(v\right)\Vert^2 \lesssim \vert v \vert^{\gamma+2} \leq 1 + \vert v \vert^{-2}$, exchangeability and Lemma~\ref{lem:keyestimate} yielding
    \begin{align*}
        (\mathbb{E}\vert \mathcal{B}\vert)^2 &\leq \frac{C_{\varphi,I_0,t}}{N}.
    \end{align*}
    Combining all steps, we have obtained
    $$\mathbb{E}\vert \mathcal{F}_{\varphi,t}(\mu^N)\vert \leq C_{\varphi,I_0,t}\left(\eta_N^{\gamma+3} + \eta_N^{\gamma+2} + \frac{1}{N}+\frac{1}{\sqrt{N}}\right)\xrightarrow[N\rightarrow\infty]{} 0,$$
    and the proof is concluded.
\end{proof}

The functional $\mathcal{F}_{\varphi,t}$ does not allow us to pass $\mathcal{F}_{\varphi,t}(\mu^N_t)$ to the limit in $N$ because it is not continuous on $C([0,T],\mathcal{P}(\mathbb{R}^3))$ (it is not even properly defined on this space). We want to replace it by a better-behaved functional $\mathcal{F}^\varepsilon_{\varphi,t}$, which is defined by regularizing $\alpha$ again, this time independently of $N$.
\begin{lemma}
    \label{aswsolution-fepsilon}
    For any $0<\varepsilon<1$, consider the smooth and bounded regularization $\alpha_\varepsilon$ as defined in Remark~\ref{rem:regularizedpotential}. Let $b_\varepsilon=\nabla\cdot(\alpha_\varepsilon a)$, and, for any $\varphi\in C^2_b(\mathbb{R}^3)$ and any $t\in[0,T]$,
    let $\mathcal{F}^\varepsilon_{\varphi,t}$ be defined for $h\in C([0,T],\mathcal{P}(\mathbb{R}^3))$ as
    \begin{align*}
    \mathcal{F}^\varepsilon_{\varphi,t}(h) &:=  -\int_{\mathbb{R}^3} \varphi(v)h_t(v)dv + \int_{\mathbb{R}^3} \varphi(v)h_0(v)dv \nonumber\\
    &+ \int_0^t \int_{\mathbb{R}^6} b_\varepsilon(v-w) \cdot (\nabla\varphi(v)-\nabla\varphi(w))\ h_s(dv)h_s(dw) ds\\
    &+ \int_0^t \int_{\mathbb{R}^6}\alpha_\varepsilon a(v-w):\nabla^2\varphi(v)\ h_s(dv)h_s(dw) ds.\nonumber
    \end{align*}
    Then $\mathcal{F}^\varepsilon_{\varphi,t}$ is continuous and bounded on $C([0,T],\mathcal{P}(\mathbb{R}^3))$.
\end{lemma}
\begin{remark}
    The definition of $\mathcal{F}^\varepsilon_{\varphi,t}$ is exactly the one~\eqref{eq:def_f} of $\mathcal{F}_{\varphi,t}$ but with $\alpha_\varepsilon$ and $b_\varepsilon$, and without the indicator functions.
\end{remark}
\begin{proof}
    Since $\varphi, b_\varepsilon$ and $\alpha_\varepsilon a$ are continuous and bounded, we can integrate them against any probability measure and $\mathcal{F}^\varepsilon_{\varphi,t}(h)$ is always well defined, and bounded by some $C_{\varphi,\varepsilon}$.
    Consider a converging sequence $h^N \rightarrow h$ in $C([0,T],\mathcal{P}(\mathbb{R}^3))$. It implies pointwise weak convergence $h^N_t\rightharpoonup h_t$ in $\mathcal{P}(\mathbb{R}^3)$, which in turn implies $h^N_t(dv)h^N_t(dw)\rightharpoonup h_t(dv) h_t(dw). $
    (To see this, consider the dense sub-algebra of $C_b(\mathbb{R}^6)$ generated by functions with separated variables $\phi(v)\psi(w)$.) Hence, all the integrals over $\mathbb{R}^3$ and $\mathbb{R}^6$ in the definition of $\mathcal{F}^\varepsilon_{\varphi,t}$ converge. A direct dominated convergence in time concludes that $\mathcal{F}^\varepsilon_{\varphi,t}(h^N)\rightarrow\mathcal{F}^\varepsilon_{\varphi,t}(h)$.
\end{proof}

The weak convergence $\mathcal{L}(\mu^N)\rightharpoonup\pi=\mathcal{L}(f)$ in $\mathcal{P}(C([0,T],\mathcal{P}(\mathbb{R}^3)))$ can be tested with the continuous bounded function $\vert \mathcal{F}^\varepsilon_{\varphi,t}\vert$ to get
$$\mathbb{E}\vert \mathcal{F}^\varepsilon_{\varphi,t}(\mu^N)\vert\xrightarrow[N\rightarrow\infty]{}\mathbb{E}\vert \mathcal{F}^\varepsilon_{\varphi,t}(f)\vert.$$
We would like the same to hold for $\mathcal{F}_{\varphi,t}$. We hence want to compare $\mathcal{F}^\varepsilon_{\varphi,t}$ and $\mathcal{F}_{\varphi,t}$ on both the empirical measures and on the limit. This is done in the following two lemmas.

\begin{lemma}
    \label{aswsolution-feps f empirical measures}
    For any $\varphi\in C^2_b(\mathbb{R}^3)$ and any $t\in[0,T]$, there exists a constant $C_{\varphi,t}$ such that for any $0<\varepsilon<1$, for any $N\geq 2$,
    \begin{align*}
        \mathbb{E} \left\vert\mathcal{F}_{\varphi,t}(\mu^N)-\mathcal{F}^\varepsilon_{\varphi,t}(\mu^N) \right\vert \leq C_{\varphi,I_0,t} \varepsilon^{\gamma+4}.
    \end{align*}
    We emphasize that $ C_\varphi$ is independent of $N$. 
\end{lemma}
\begin{proof}
    The proof is very similar to the Steps 2 and 3 of the proof of Lemma~\ref{aswsolution-empirical measures}. We have
    \begin{align*}
        \mathcal{F}_{\varphi,t}(\mu^N)&-\mathcal{F}^\varepsilon_{\varphi,t}(\mu^N)\\
        &=
    \frac{1}{N^2}\int_0^t \sum_{\substack{i,j=1\\ i\neq j}}^N (\mathbf{1}b - b_\varepsilon)(V^N_i(s)-V^N_j(s)) \cdot (\nabla\varphi(V^N_i(s))-\nabla\varphi(V^N_j(s)))\ ds\\
    &+\frac{1}{N^2} \int_0^t \sum_{\substack{i,j=1\\ i\neq j}}^N (\mathbf{1}\alpha a - \alpha_\varepsilon a)(V^N_i(s)-V^N_j(s)) \cdot \nabla^2\varphi(V^N_i(s))\ ds.
    \end{align*}
    Remark that we can still ignore the terms $i=j$ in the sums because both $\alpha_\varepsilon a(0)=0$ and $b_\varepsilon(0)=0$.
    By the two following bounds which we already used
    \begin{align*}
    \vert b_\varepsilon - \mathbf{1}b(v)\vert &\leq C \mathbf{1}_{\vert v \vert \leq \varepsilon}\vert v \vert^{\gamma + 1}  \leq C \vert v \vert^{-3} \varepsilon^{\gamma+4}\\
    \Vert \alpha_\varepsilon a - \mathbf{1}\alpha a(v)\Vert &\leq C \mathbf{1}_{\vert v \vert \leq \varepsilon}\vert v \vert^{\gamma + 2}  \leq C \vert v \vert^{-2}  \varepsilon^{\gamma+4}
    \end{align*}
    and the fact that
    $$\vert \nabla\varphi(V^N_i(s))-\nabla\varphi(V^N_j(s))\vert\leq C_\varphi \vert V^N_i(s)-V^N_j(s)\vert$$
    we have, taking expectations:
    \begin{align*}
        \mathbb{E} \left\vert\mathcal{F}_{\varphi,t}(\mu^N)-\mathcal{F}_{\varphi,t}^\varepsilon(\mu^N) \right\vert &\leq
    \varepsilon^{\gamma+4}\frac{C_\varphi}{N^2}\sum_{\substack{i,j=1\\ i\neq j}}^N\int_0^t  \mathbb{E}\left\vert V^N_i(s)-V^N_j(s) \right\vert^{-2}\ dt\\
    &\leq C_{\varphi,I_0,t} \varepsilon^{\gamma+4}.
    \end{align*}
    The last line was obtained thanks to exchangeability and Lemma~\ref{lem:keyestimate}.
\end{proof}

We finally compare $\mathcal{F}_{\varphi,t}^\varepsilon(f)$ and $\mathcal{F}_{\varphi,t}(f)$ (and check that $\mathcal{F}_{\varphi,t}(f)$ is well-defined).
\begin{lemma}
    \label{aswsolution-feps f limit}
    For any $\varphi\in C^2_b(\mathbb{R}^3)$ and any $t\in[0,T]$, $\mathcal{F}_{\varphi,t}(f)$ is a.s. well-defined ; and there exists a constant $C_{\varphi,I_0,t}$ such that for any $0<\varepsilon<1$,
    \begin{align*}
        \mathbb{E} \left\vert\mathcal{F}_{\varphi,t}(f)-\mathcal{F}^\varepsilon_{\varphi,t}(f) \right\vert \leq C_{\varphi,I_0,t} \varepsilon^{\frac{\gamma+4}{2}}.
    \end{align*} 
\end{lemma}
\begin{proof}
    To see that $\mathcal{F}_{\varphi,t}(f)$ is a.s. well defined, recall that $f\in L^1([0,T],L^3(\mathbb{R}^3))$ a.s by Lemma~\ref{lem:aswsolution-regularity}. Since, once again because $\varphi$ is $C^2$,
    $$\vert b(v-w) \cdot (\nabla\varphi(v)-\nabla\varphi(w))\vert \leq C\Vert \nabla^2\varphi \Vert_{L^\infty}\vert v-w\vert^{\gamma+2},$$
    for any fixed $v$, it lies in $L^{\frac{3}{2}}(\mathbb{R}^3)+L^{\infty}(\mathbb{R}^3)$ as a function of $w$. Using the Hölder inequality,
    $$\int_{\mathbb{R}^6} \vert \mathbf{1} b(v-w) \cdot (\nabla\varphi(v)-\nabla\varphi(w))\vert f_s(dw) \leq C\Vert \nabla^2\varphi \Vert_{L^\infty}(1+ \Vert f_s\Vert_{L^3(\mathbb{R}^3)}).$$
    We hence get that the term in $b$ in~\eqref{eq:def_f} is integrable:
    \begin{align}
    \label{eq:defintegraltermwsol}
        \int_0^t \int_{\mathbb{R}^6} \vert \mathbf{1}b(v-w) \vert &\vert \nabla\varphi(v)-\nabla\varphi(w)\vert f_s(dw)  f_s(dv)\ ds\\
        &\leq C \Vert \nabla^2\varphi \Vert_{L^\infty}\int_0^t \left(1+ \Vert f_s\Vert_{L^3(\mathbb{R}^3)}\right)ds.\nonumber
    \end{align}
    The $a$ term is similarly well-defined, since $\mathbf{1}\alpha a (v-\cdot) \lesssim \vert v-\cdot\vert^{\gamma+2}$.
    
    We now turn to bounding $\mathcal{F}_{\varphi,t}(f)-\mathcal{F}_{\varphi,t}^\varepsilon(f)$:
    \begin{align*}
        \mathcal{F}_{\varphi,t}(f)-\mathcal{F}_{\varphi,t}^\varepsilon(f)&= \int_0^t \int_{\mathbb{R}^6} (\mathbf{1}b - b_\varepsilon)(v-w) \cdot (\nabla\varphi(v)-\nabla\varphi(w) ) f_s(dv)f_s(dw) ds\\
    &+ \int_0^t \int_{\mathbb{R}^6} (\mathbf{1}\alpha a - \alpha_\varepsilon a)(v-w):\nabla^2\varphi(v) f_s(dv)f_s(dw) ds.
    \end{align*}
    We sharpen a little bit the estimates from the last lemma so that the right hand side lies in $L^{\frac{3}{2}}(\mathbb{R}^3)$:
    \begin{align*}
    \vert (b_\varepsilon - \mathbf{1}b)(v-w)\cdot (\nabla\varphi(v)-\nabla\varphi(w))\vert &\leq C_\varphi \mathbf{1}_{\vert v \vert \leq \varepsilon}\vert v-w \vert^{\gamma + 2}  \leq C_\varphi\mathbf{1}_{\vert v-w \vert \leq 1} \vert v-w \vert^{-2+\delta} \varepsilon^{\gamma+4-\delta}\\
    \Vert( \alpha_\varepsilon a - \mathbf{1}\alpha) a(v-w)\Vert &\leq C \mathbf{1}_{\vert v-w \vert \leq \varepsilon}\vert v -w\vert^{\gamma + 2}  \leq C \mathbf{1}_{\vert v-w \vert \leq 1} \vert v -w\vert^{-2+\delta}  \varepsilon^{\gamma+4-\delta}
    \end{align*}
    with $\delta=\frac{\gamma + 4}{2}$. Plugging those in, we get:
    \begin{align*}
        \left\vert \mathcal{F}_{\varphi,t}(\mu)-\mathcal{F}_{\varphi,t}^\varepsilon(\mu) \right\vert&\leq C_\varphi \varepsilon^{\frac{\gamma+4}{2}}\int_0^t \int_{\mathbb{R}^6} \mathbf{1}_{\vert v-w \vert \leq 1} \vert v-w \vert^{-2+\delta}  \ f_s(dv)f_s(dw) ds\\
    &\leq C_\varphi \varepsilon^{\frac{\gamma+4}{2}}\int_0^t \int_{\mathbb{R}^3} \Vert f_s\Vert_{L^3(\mathbb{R}^3)} f_s(dv) ds\\
    &\leq C_\varphi \varepsilon^{\frac{\gamma+4}{2}}\int_0^t \Vert f_s\Vert_{L^3(\mathbb{R}^3)} ds
    \end{align*}
    where the second line is obtained by applying the Hölder inequality in $w$ (since the $L^\frac{3}{2}$ norm of $\mathbf{1}_{\vert v-\cdot \vert \leq 1} \vert v-\cdot \vert^{-2+\delta}$ is a constant).
    Taking expectations and using the Sobolev embedding to control the $L^3$ norm by the Fisher information, one gets
    \begin{align*}
        \mathbb{E}\left\vert \mathcal{F}_{\varphi,t}(\mu)-\mathcal{F}_{\varphi,t}^\varepsilon(\mu) \right\vert&\leq C_\varphi \varepsilon^{\frac{\gamma+4}{2}} \mathbb{E}\left[\int_0^t \Vert f_s\Vert_{L^3(\mathbb{R}^3)}  ds \right]\\
        &\leq C_\varphi \varepsilon^{\frac{\gamma+4}{2}} \int_0^t \mathbb{E}\left[I(f_t)\right]  dt \\
        &\leq C_\varphi \varepsilon^{\frac{\gamma+4}{2}} t I_0,
    \end{align*}
    using that $\mathbb{E}\left[I(f_t)\right]=\int I(\rho)\pi_t(d\rho)\leq I_0$ by Proposition~\ref{prop:level3estimates}.
\end{proof}

We are now all set to show that $f$ satisfies the weak formulation almost surely. Remark that in the following lemma, the almost sure event depends a priori on $\varphi$ and on $t$.
\begin{lemma}
\label{aswsolution:weakformholds}
    Let $\varphi\in C^2_b(\mathbb{R}^3)$ and $t\in[0,T]$.
    Almost surely,~\eqref{eq:weakform} holds.
\end{lemma}
\begin{proof}
    \textit{Step 1.} We first show that $\mathcal{F}_{\varphi,t}(f)=0$ a.s. (because, up to the indicator functions, it is the weak formulation).
    Let $0<\varepsilon<1$,
    \begin{align*}
        \mathbb{E}\vert \mathcal{F}_{\varphi,t}(f) \vert &\leq \mathbb{E}\vert \mathcal{F}_{\varphi,t}(f) -\mathcal{F}_{\varphi,t}^\varepsilon(f) \vert +\mathbb{E}\vert  \mathcal{F}_{\varphi,t}^\varepsilon(f)\vert.
    \end{align*}
    By Lemma~\ref{aswsolution-feps f limit},
    $$\mathbb{E}\vert \mathcal{F}_{\varphi,t}(f) -\mathcal{F}_{\varphi,t}^\varepsilon(f) \vert  \leq C_{\varphi,I_0,t}\varepsilon^{\frac{\gamma+4}{2}}, $$
    and by continuity and boundedness of $\mathcal{F}_{\varphi,t}^\varepsilon$ (Lemma~\ref{aswsolution-fepsilon}),
    $$\mathbb{E}\vert  \mathcal{F}_{\varphi,t}^\varepsilon(f)\vert = \lim_{N\rightarrow \infty}\mathbb{E}\vert  \mathcal{F}_{\varphi,t}^\varepsilon(\mu^N)\vert$$
    (up to a subsequence) by testing the weak convergence $\mathcal{L}(\mu^N)\rightharpoonup \pi=\mathcal{L}(f)$. Hence
    \begin{align}
    \label{eq:aswsolutionproof1}
        \mathbb{E}\vert \mathcal{F}_{\varphi,t}(f) \vert &\leq C_{\varphi,I_0,t}\varepsilon^{\frac{\gamma+4}{2}}+\lim_{N\rightarrow \infty}\mathbb{E}\vert  \mathcal{F}_{\varphi,t}^\varepsilon(\mu^N)\vert.
    \end{align}
    Now, for any $N\geq 2$:
    \begin{align*}
        \mathbb{E}\vert \mathcal{F}_{\varphi,t}^\varepsilon(\mu^N) \vert &\leq \mathbb{E}\vert  \mathcal{F}_{\varphi,t}^\varepsilon(\mu^N)-\mathcal{F}_{\varphi,t}(\mu^N)\vert+\mathbb{E}\vert  \mathcal{F}_{\varphi,t}(\mu^N)\vert \\
        &\leq C_{\varphi,I_0,t}\varepsilon^{\gamma+4}+\mathbb{E}\vert  \mathcal{F}_{\varphi,t}(\mu^N)\vert
    \end{align*}
    by Lemma~\ref{aswsolution-feps f empirical measures}. But $\mathbb{E}\vert  \mathcal{F}_{\varphi,t}(\mu^N)\xrightarrow[N\rightarrow\infty]{} \infty$ by Lemma~\ref{aswsolution-empirical measures}, so that taking the limit
    $$\lim_{N\rightarrow \infty}\mathbb{E}\vert  \mathcal{F}_{\varphi,t}^\varepsilon(\mu^N)\vert\leq C_{\varphi,I_0,t}\varepsilon^{\gamma+4}.$$
    Plugging this back in~\eqref{eq:aswsolutionproof1},
    \begin{align*}
        \mathbb{E}\vert \mathcal{F}_{\varphi,t}(f) \vert &\leq C_{\varphi,I_0,t}\varepsilon^{\frac{\gamma+4}{2}}+C_{\varphi,t}\varepsilon^{\gamma+4},
    \end{align*}
    and letting $\varepsilon\rightarrow 0$ yields $\mathbb{E}\vert \mathcal{F}_{\varphi,t}(f) \vert=0$, so almost surely $\mathcal{F}_{\varphi,t}(f)=0$.

    \textit{Step 2.} Since $f\in L^\infty([0,T],L^1(\mathbb{R}^3))$ by Lemma~\ref{lem:aswsolution-regularity}, it admits a density for almost all $s$, so $f_s(dv)f_s(dw)$ does not weight the set of $0$ Lebesgue measure $\{v=w\}$. We can then remove the indicator functions in the definition~\eqref{eq:def_f} of $\mathcal{F}_{\varphi,t}$ without changing the values of the integral. Moreover, by Lemma~\ref{lem:aswsolution-initial time}, $f_0=g_0$ almost surely so $\mathcal{F}_{\varphi,t}(f)=0$ a.s. entails that the weak formulation~\eqref{eq:weakform} holds.
\end{proof}

We can now end the proof of Proposition~\ref{prop:aswsolution}.
\begin{proof}[Proof of Proposition~\ref{prop:aswsolution}]
The random variable $f$ almost surely satisfies the claimed regularity by Lemma~\ref{lem:aswsolution-regularity}.
By Lemma~\ref{aswsolution:weakformholds}, for any $t\in[0,T]$ and any $\varphi\in C^2_b(\mathbb{R}^3)$,~\eqref{eq:weakform} holds almost surely. We want to make this almost sure event be the same for all $t$ and $\varphi$.

The space of $C^2$ bounded functions $C^2_b(\mathbb{R}^3)$ is not separable, but the space $C^2_0(\mathbb{R}^3)$ of functions that vanish at infinity is, so we can consider a countable dense set $\{\varphi_n\}_n$. We also consider a dense countable subset $\{t_k\}_k$ of $[0,T]$. Then, almost surely,~\eqref{eq:weakform} holds for all $t_k$ and all $\varphi_n$ simultaneously.

Because $f\in C([0,T],\mathcal{P}(\mathbb{R}^3))$,  $t\mapsto\int \varphi(v) f_t(dv)$ is continuous and the other terms in~\eqref{eq:weakform} are also continuous in $t$ (one is constant, the other two are integrals on $[0,t]$ of a function in $L^1([0,T])$, by a reasoning similar to~\eqref{eq:defintegraltermwsol}.) Hence, for any $\varphi_n$, the weak formulation holds for all $t\in[0,T]$ by continuity.

It is then easy to see that a.s., the weak formulation holds for any $\varphi\in C^2_0(\mathbb{R}^3)$ by a density argument: indeed, if $\Vert \varphi - \varphi_n\Vert_{C^2}\leq \epsilon$, the weak formulation with $\varphi$ is smaller than $  C\epsilon \left(1 + \Vert f \Vert_{L^1([0,T],L^3(\mathbb{R}^3))} \right) $,
again by a computation similar to~\eqref{eq:defintegraltermwsol}.

To show the weak formulation for a $\varphi\in C^2_b(\mathbb{R}^3)$, consider a smooth bump function $\chi$ and consider the truncation $\varphi^R:= \varphi \chi(\frac{\cdot}{R})\in C^2_0(\mathbb{R}^3)$. A dominated convergence in the weak formulation for $\varphi^R$ yields the result. Hence, almost surely, the weak formulation holds on $C^2_b(\mathbb{R}^3)$ as desired.
\end{proof}

\begin{remark}
    At this point, if we knew a uniqueness theorem for the Landau equation in the class $L^1([0,T],L^3(\mathbb{R}^3))$, we could conclude that $f$ is a.s. the unique solution, so that $\pi=\delta_{(g_t)_t}$. However, the results from \cite{GuerinFournier2008,Fournier2010} require higher integrability in space, going all the way to $L^\infty$ as $\gamma\rightarrow-3$.
    Even if we managed to show that $f\in L^\infty([0,T],L^3(\mathbb{R}^3))$ a.s. (which would be more natural since the Fisher information of the particle system is $L^\infty$ in time), the known uniqueness theorems \cite{AlonsoBaglandDesvillettes2024,ChernGualdani2022} require some H-theorem, which does not trivially hold for our notion of weak solutions.
\end{remark}
In the next final section, we will use the -yet unexploited- dissipation of the Fisher information to essentially show that $f\in L^1([0,T],L^\infty(\mathbb{R}^3))$ which will allow us to conclude that $f$ is the regular solution $g$.

\section{A final $L^1_t L^\infty_v$ estimate and uniqueness of the cluster point}
\label{sec:a final estimate}
We are still working on $\pi \in \mathcal{P}\left(C([0,T],\mathcal{P}(\mathbb{R}^3)) \right)$ a cluster point of $\left(\mathcal{L}\left( \mu^N\right) \right)_{N\geq 2 }$. We know from the previous section that a $\pi$-distributed random variable $f$ is almost surely a weak solution to the Landau equation, and from Section~\ref{sec:infinite-dimension limit of the regularity estimates} we know that $f$ enjoys a.s. some improved regularity. In particular, recall that by Proposition~\ref{prop:level3estimatesforf}, the Fisher dissipation is finite:
\begin{equation}
\label{eq:Kasfinite}
     \text{a.s.,}\int_0^T \mathcal{K}(f_t\otimes f_t)ds <+\infty.
\end{equation}
The functional $\mathcal{K}$ measures smoothness of $f\otimes f$, but only in the particuliar directions given by the $\tilde{b}^i$s. Using the tensorized structure, we can hope to recover some regularity on $f$ itself in all directions. We want to show in the end that $f\in L^1([0,T],L^\infty(\mathbb{R}^3))$, which is enough to apply Guérin and Fournier's uniqueness results \cite{GuerinFournier2008,Fournier2010} and claim that $f=g$. This entails uniqueness of the cluster point $\pi=\delta_{(g_t)_t}$, so that the whole sequence $\left(\mathcal{L}\left( \mu^N\right) \right)_{N\geq 2 }$ converges and propagation of chaos holds. The goal of this section is to reach the 

\begin{prop}
    \label{prop:asfequalg}
    The random variable $(f_t)_t$ is almost surely equal to the strong solution $(g_t)_t$. In other words, $\pi=\delta_{(g_t)_t}$.
\end{prop}

\begin{remark}
    The strategy we will follow is by now classical in the study of the Landau equation: Desvillettes \cite{Desvillettes2014a} showed that $D(f \otimes f)$, which is a partial Fisher information in the directions $(\tilde{b}_k)_{k=1,2,3}$, controls a full (weighted) Fisher information of $f$. Ji \cite{Ji2024} proved a second order version of this result, showing that a functional that is very similar to $\mathcal{K}(f\otimes f)$ (but with second derivatives in more directions) controls a full (weighted) second order Fisher information $\int f\jap{v}^{\gamma-2}\vert \nabla^2 \log f\vert^2$. We were not able to apply this strategy to $\mathcal{K}$ itself, but there exists a first order term $J\leq \mathcal{K}$ for which we can do so, and it is enough to conclude.
    
    The estimates cited above hold under the assumption that $f$ has finite entropy (and energy), with a constant exponential in $H$. Even if this is to be expected for any reasonable solution of the Landau equation, we do not have a uniform-in-time bound on the entropy of $f$: it is not known to be smooth enough for the H-theorem to hold. Using the strategy introduced in \cite{FournierHauray2015}, we will replace the entropy of $f$ by a weaker quantity measuring that the mass of $f$ is not too concentrated (\textit{i.e.} not distributed along a single line). For this quantity uniform bounds are achievable for short times by continuity arguments.
\end{remark}

\subsection{A conditional $L^1_t L^\infty_v$ bound}

We show that starting from~\eqref{eq:Kasfinite}, we can get a bound on the $L^1_t L^\infty_v$ norm of $(f_t)_t$. However the bound will require two additional quantities to be finite, which we will deal with in the next two sections.
We begin by isolating a first order term from $\mathcal{K}$, on which we will work for the rest of this section:
\begin{lemma}
\label{lem:Jasfinite}
    Let
    \begin{equation}
    \label{eq:Jasfinite}
        \mathcal{J}(f):=\sum_{k=1}^3  \int_{\mathbb{R}^{6}}f(v)f(w)  \frac{\alpha(v-w)}{\vert v - w\vert^2} \left( (\tilde{b}_k(v-w)\cdot \nabla)\log (f(v)f(w)) \right)^4 dvdw.
    \end{equation}
    Almost surely $\int_0^T\mathcal{J}(f_t)dt<+\infty$.
\end{lemma}
\begin{proof}
     We use Proposition~\ref{prop:func_comparison} with $\beta=1$ and $\partial=\partial_{b_k}$ (defined with the unregularized potential $\alpha$, see~\eqref{eq:defpartialNbk}). Since by definition $K_{1/3}^{\partial_{b_k}}\geq 0$, we have by summing Proposition~\ref{prop:func_comparison} for $k=1,2,3$:
     $$\mathcal{K}=\sum_{k=1}^3  K_1^{\partial_{b_k}}\geq \sum_{k=1}^3\frac{4}{9}J^{\partial_{b_k}},$$
     where
     $$J^{\partial_{b_k}}(F)=\int_{\mathbb{R}^{6}}F(\partial_{b_k}\log F)^4.$$
     But $\sum_{k=1}^3J^{\partial_{b_k}}(f\otimes f) = \mathcal{J}(f)$, so this control and~\eqref{eq:Kasfinite} imply
     $\int_0^T \mathcal{J}(f_t)dt<+\infty$ a.s.
\end{proof}

From $\mathcal{J}(f)$ we can isolate a main coercive term with derivatives only on $f(v)$, up to an error of the magnitude of the Fisher information. For this we introduce once again a cut-off potential $\bar{\alpha}=\alpha_\eta$ (with $\eta=1$ to fix ideas) using the notations of Remark~\ref{rem:regularizedpotential}. What matters is that we have $\bar{\alpha}(r)\leq \alpha(r)$, with equality for $r\geq 1$, and $\alpha \leq 0.99^\gamma$.
\begin{lemma}
    \label{lem:Jcoerciveerrorcut}
    The following inequality holds:
    \begin{align}
        \label{eq:Jcoerciveerrorcut}
        \mathcal{J}(f)
    &\geq \mathcal{J}_1(f)
    -0.99^{\gamma}\cdot 432 I(f),
    \end{align}
    with
    $$ \mathcal{J}_1(f):=\sum_{k=1}^3\int_{\mathbb{R}^{6}}f(v)f(w)  \frac{\bar{\alpha}(v - w)}{\vert v - w\vert^2} \left( b_k(v-w)\cdot \nabla_v\log f(v) \right)^4 dvdw$$
\end{lemma}
\begin{proof}
    Recall that $\tilde{b}_k=(b_k,-b_k)$. Replacing $\alpha$ by the smaller $\bar{\alpha}$, we hence have:
    $$\mathcal{J}(f)\geq \sum_{k=1}^3  \int_{\mathbb{R}^{6}}f(v)f(w)  \frac{\bar{\alpha}(v-w)}{\vert v - w\vert^2} \left( (b_k\cdot \nabla_v\log (f(v)) - b_k \cdot \nabla_w f(w)) \right)^4 dvdw,$$
    where we have dropped the $(v-w)$ argument of $b_k$ for concision.
    Expanding and using symmetry between $v$ and $w$, we get:
    \begin{align*}  
    \mathcal{J}(f)&\geq   2\sum_{k=1}^3\int_{\mathbb{R}^{6}}f(v)f(w)  \frac{\bar{\alpha}}{\vert v - w\vert^2} \left( b_k\cdot \nabla_v\log f(v) \right)^4 dvdw\\
    &-8\sum_{k=1}^3\int_{\mathbb{R}^{6}}f(v)f(w)  \frac{\bar{\alpha}}{\vert v - w\vert^2} \left( b_k\cdot \nabla_v\log f(v) \right)^3 \left(b_k\cdot \nabla_w\log f(w)\right) dvdw\\
    &+6\sum_{k=1}^3\int_{\mathbb{R}^{6}}f(v)f(w)  \frac{\bar{\alpha}}{\vert v - w\vert^2} \left( b_k\cdot \nabla_v\log f(v) \right)^2 \left( b_k\cdot \nabla_w\log f(w) \right)^2 dvdw\\
    &=2\mathcal{J}_1(f) - 8J_2 +6J_3
    \end{align*}
    We can drop $J_3$ since it is nonnegative. $\mathcal{J}_1(f)$ is the term we want to keep and $J_2$ the error to control.

    In $J_2$, we combine $f(w)$ and $b^i\cdot \nabla_w\log f(w)$ and then integrate by parts:
    \begin{align*} 
        J_2&=\sum_{k=1}^3\int_{\mathbb{R}^{6}}f(v) \frac{\bar{\alpha}}{\vert v - w\vert^2} \left( b_k\cdot \nabla_v\log f(v) \right)^3 b_k\cdot \nabla_w f(w) dvdw\\
        &=-\sum_{k=1}^3\int_{\mathbb{R}^{6}}f(v)f(w)  b_k\cdot \nabla_w \left[\frac{\bar{\alpha}}{\vert v - w\vert^2} \left( b_k\cdot \nabla_v\log f(v) \right)^3\right] dvdw\\
        &=-\sum_{k=1}^3\int_{\mathbb{R}^{6}}f(v)f(w)  \frac{\bar{\alpha}}{\vert v - w\vert^2} 3\left( b_k\cdot \nabla_v\log f(v) \right)^2(b_k\cdot \nabla_w b_k)\cdot\nabla_w \log f(v) dvdw.
    \end{align*}
    For the last line, we used that the gradient $\nabla_w (\bar{\alpha}\vert v-w\vert^{-2} )$ is parallel to $v-w$ so vanishes against $b_k$.

    Using Young's inequality to say $3ab \leq \frac{a^2}{8} + 18b^2$, 
    \begin{align*}
    J_2
    \leq& \sum_{k=1}^3\int_{\mathbb{R}^{6}}f(v)f(w)  \frac{\bar{\alpha}}{\vert v - w\vert^2} \bigg[\frac{1}{8}\left( b_k\cdot \nabla_v\log f(v) \right)^4\\
    &\; \; \; \; + 18\left((b_k\cdot \nabla_w b_k)\cdot\nabla_v \log f(v)\right)^2\bigg] dvdw\\
    =&\frac{1}{8}\mathcal{J}_1(f) + 18\sum_{k=1}^3\int_{\mathbb{R}^{6}}f(v)f(w)  \frac{\bar{\alpha}}{\vert v - w\vert^2} ((b_k\cdot \nabla_w b_k)\cdot\nabla_v \log f(v))^2 dvdw.
\end{align*}
On the second term we use the bound
$$((b_k\cdot \nabla_w b_k)\cdot\nabla_v \log f(v))^2 \leq \vert v-w\vert^2 \vert \nabla_v \log f(v) \vert^2, $$consequence of a direct computation, and that $\bar{\alpha} \leq 0.99^{-\gamma}$ to get:
\begin{align*}
J_2
&\leq \frac{1}{8}\mathcal{J}_1(f)+ 0.99^{-\gamma}18\sum_{k=1}^3  \int_{\mathbb{R}^{6}}f(v)f(w)  \vert \nabla \log f(v) \vert^2 dvdw \\
&= \frac{1}{8}\mathcal{J}_1(f)+0.99^{-\gamma}54  \int_{\mathbb{R}^{3}}f(v)  \vert \nabla \log f(v) \vert^2 dv,
\end{align*}
the integral being nothing but the Fisher information. Combining, we get:
\begin{align*}
\mathcal{J}(f)&\geq 2\mathcal{J}_1(f)-8J_2\geq \mathcal{J}_1(f)-0.99^{-\gamma}\cdot 432I(f)
\end{align*}
and the proof is over.
\end{proof}
\begin{remark}
    The cut-off potential is only introduced so that it is possible to bound the error term $J_2$. The same is done in \cite{Desvillettes2014a,Ji2024} to prove similar estimates.
\end{remark}

One would like to claim that $\mathcal{J}_1(f)$ controls $\int f\jap{v}^l\vert \nabla_v\log f\vert^4 dv$ for some weight $l$. The classical way to show this would be to use the finite entropy of $f$ to say that $f$ cannot be too concentrated, so that the integral over $w$ makes the vector $b_k(v-w)$ span all directions. However, we cannot show better than $H(f_t)\in L^1([0,T])$ a.s., which is not enough. To circumvent the use of the entropy, we will apply the technique developed in \cite{FournierHauray2015} based on the notion of triplets of $\delta$-non-aligned points:
\begin{defi}\cite[Definition 6.1]{FournierHauray2015}
\label{def:deltanonaligned}
    Let $\delta>0$. We say that three points $v_1,v_2,v_3$ are \textbf{$\delta$-non-aligned} if
    \begin{equation}
        \label{eq:defdeltana1}
        \vert v_2-v_1\vert \geq 6\sqrt{\delta}
    \end{equation}
    and, if $p_{(v_2-v_1)^\perp}$ is the projection on the plane $(v_2-v_1)^\perp$:
    \begin{equation}
        \label{eq:defdeltana2}
        \vert p_{(v_2-v_1)^\perp}( v_3-v_1 )\vert  \geq 24\delta + 2\sqrt{\delta}\vert v_3-v_1\vert.
    \end{equation}
\end{defi}
If a probability $\mu$ has finite entropy, it cannot be too concentrated on a line: it must charge the neighborhoods of some $\delta$-non-aligned points. If it has finite energy, we can also ensure that most of the mass of $\mu$ is in a large ball centered on the origin, so that the three points can be chosen in it. This is expressed in the following lemma, which justifies the notion of \textit{charging non aligned points} as a suitable generalization of the entropy.
We fix a smooth bump function $h$ such that $0\leq h \leq 1$, $h=1$ on $B(0,1)$ and $h=0$ outside $B(0,3/2)$.
\begin{lemma}\cite[Lemma 6.3]{FournierHauray2015}
    \label{lem:entropyimpliesnonaligned}
    Let $H^*>0$ and $E^*>0$. There exists $\delta\in(0,1)$, $R>0$ and $\kappa>0$ (depending only on $H^*$ and $E^*$) such that for any $\mu\in\mathcal{P}(\mathbb{R}^3)$ satisfying $H(\mu)\leq H^*$ and $\int\vert v \vert^2 \mu(dv) \leq E^*$, there exists three points $v_1,v_2,v_3\in B(0,R)$, which are $\delta$-non-aligned, and such that for $k=1,2,3$,
    \begin{equation}
    \label{eq:entropyimpliesnonaligned}
            \int_{\mathbb{R}^3}h\left(\frac{w-v_k}{\delta}\right)\mu(dw) \geq \int_{B(v_k,\delta)}\mu(dw) \geq \kappa.
    \end{equation}
\end{lemma}
\begin{proof}
    See \cite{FournierHauray2015} for a proof that $\int_{B(v_k,\delta)}\mu(dw) \geq \kappa$. The other inequality is simply because $h\left(\frac{\cdot-v_k}{\delta}\right)\geq \mathbf{1}_{B(v_k,\delta)}$.
\end{proof}
Whereas the entropy is not continuous with respect to the weak topology on $\mathcal{P}(\mathbb{R}^3)$, the left hand side of~\eqref{eq:entropyimpliesnonaligned} is, which will allow us to propagate~\eqref{eq:entropyimpliesnonaligned} with $\mu=f_0$ (which, recall, is a.s. equal to the regular $g_0$) over a small time interval $[0,\tau]$. We will not use this lemma immediately, but we still stated it here to showcase that the quantity in the left hand side of~\eqref{eq:entropyimpliesnonaligned}, which will appear in the following lemmas, is really a (more general) replacement of bounds on the entropy of $\mu$.

We hence show how a control like~\eqref{eq:entropyimpliesnonaligned} is sufficient to extract from $\mathcal{J}_1(f)$ the bound we want.
The key lemma is that any narrow enough cone cannot intersect all three balls around a $\delta$-non-aligned triplet:
\begin{lemma}\cite[adapted from Lemma 6.2]{FournierHauray2015}
    \label{lem:conenonaligned}
    Let $v_1,v_2,v_3\in B(0,R)$ be $\delta$-non-aligned.
    Let $v\in\mathbb{R}^3$ and $\xi$ a unit vector and define the (double-sided) cone
    $$C=\left\{ w\in \mathbb{R}^3 \;\bigg\vert\; \frac{\vert p_{\xi^\perp}(v-w)\vert}{\vert v-w\vert} \leq \frac{\delta}{2+R+\vert v\vert} \right\},$$
    centered on $v$ with axis $\xi$. There exists $k\in\{1,2,3\}$ such that $C$ does not intersect $B(v_k,2\delta)$.
\end{lemma}
\begin{proof}
    This lemma is exactly Step 1 in the proof of \cite[ Lemma 6.2]{FournierHauray2015}.
\end{proof}

Using this Lemma, we derive the following estimate. Notice the closeness of the left-hand side with the integral in $w$ in $\mathcal{J}_1(f)$.
\begin{lemma}
    \label{lem:ellipticityestimate}
    Let $v_1,v_2,v_3\in B(0,R)$ be $\delta$-non-aligned.
    For any $v\in\mathbb{R}^3$, any unit vector $\xi$, any $\mu\in\mathcal{P}(\mathbb{R}^3)$,
    $$\sum_{k=1}^3\int_{\mathbb{R}^{3}}\frac{\bar{\alpha}(v-w)}{\vert v - w\vert^2} \left( b_k(v-w)\cdot \xi \right)^4 \mu(dw) \geq c_{\delta,R}\jap{v}^{\gamma-2}\min_{k\in\{1,2,3\} }  \int_{\mathbb{R}^3}h\left(\frac{w-v_k}{\delta}\right)\mu(dw),$$
    where $c_{\delta,R}$ depends on $\delta$ and $R$ only.
\end{lemma}
\begin{proof}
    By the Cauchy-Schwarz inequality,
    $$\sum_{k=1}^3\left( b_k(v-w)\cdot \xi \right)^4 \geq \frac{1}{3}\left(\sum_{k=1}^3\left( b_k(v-w)\cdot \xi \right)^2\right)^2.$$
    But since $a(v-w)=\sum_{k=1}^3 b_k\otimes b_k$,
    \begin{align*}
    \sum_{k=1}^3\left( b_k(v-w)\cdot \xi \right)^2&=a(v-w)\xi \cdot \xi\\
    &=\vert v-w\vert^2-((v-w)\cdot\xi)^2\\
    &=\vert p_{\xi^\perp}(v-w)\vert ^2.
    \end{align*}
    Hence
    \begin{align*}
        \sum_{k=1}^3\int_{\mathbb{R}^{3}}\frac{\bar{\alpha}(v-w)}{\vert v - w\vert^2} \left( b_k(v-w)\cdot \xi \right)^4 \mu(dw) \geq &\frac{1}{3}\int_{\mathbb{R}^{3}}\frac{\bar{\alpha}(v-w)}{\vert v - w\vert^2} \vert p_{\xi^\perp}(v-w)\vert^4 \mu(dw).
    \end{align*}
    By Lemma~\ref{lem:conenonaligned}, there exists $k\in\{1,2,3\}$ such that the cone $C$ does not intersect $B(v_k,2\delta)$. For any $w\in B(v_k,2\delta)$, by definition of $C$,
    $$\vert p_{\xi^\perp}(v-w)\vert > \frac{\delta \vert v-w\vert}{2+R+\vert v\vert},$$
    so that
    \begin{align*}
        \sum_{k=1}^3\int_{\mathbb{R}^{3}}\frac{\bar{\alpha}}{\vert v - w\vert^2} \left( b_k(v-w)\cdot \xi \right)^4 \mu(dw) \geq &\frac{\delta^4}{3(2+R+\vert v\vert)^4} \int_{B(v_k,2\delta)}\bar{\alpha}\vert v - w\vert^2\mu(dw)\\
        \geq&\frac{\delta^4}{3(2+R+\vert v\vert)^4} \int_{\mathbb{R}^{3}}\bar{\alpha}\vert v - w\vert^2h\left(\frac{w-v_k}{\delta}\right)\mu(dw),
    \end{align*}
    because $h(\cdot-v_k/\delta) \leq \mathbf{1}_{B(v_k,2\delta)}$
    We bound $\bar{\alpha}\vert v - w\vert^2$ from below. If $\vert v-w\vert \geq 1$, $\bar{\alpha}(\vert v-w\vert)=\alpha(\vert v-w\vert)$ so $\bar{\alpha}\vert v - w\vert^2=\vert v - w\vert^{2+\gamma}$. Since $2+\gamma\leq 0$,
    $$\bar{\alpha}\vert v - w\vert^2\geq ( \vert v\vert + \vert  w\vert )^{2+\gamma}\geq ( \vert v\vert + R + 2 )^{2+\gamma}$$
    because $\vert w\vert \leq R+2\delta \leq R+2$ for $w\in B(v_k,2\delta)$. Otherwise, if $\vert v-w\vert < 1$, $\bar{\alpha}\geq 1$ and since $h$ has support in $B(0,3/2)$, we can assume $w\in B(v_k,3\delta/2)$, and hence $\vert v-w\vert \geq \delta/2$ because $v\in C$. Then
    $$\bar{\alpha}\vert v - w\vert^2 \geq \frac{\delta^2}{4}. $$
    Furthermore, $\vert v-w\vert < 1$ and $\vert w \vert \leq R+2$ imply that $\vert v\vert\leq R+3$, so the behaviour for large $v$ is given by the other inequality.
    Combining, we obtain that for any $w$ in the support of $h\left(\frac{\cdot-v_k}{\delta}\right)$,
    $$\frac{\delta^4\bar{\alpha}\vert v - w\vert^2} {3(2+R+\vert v\vert)^4}\geq c_{\delta,R} \jap{v}^{\gamma-2},$$
    for some $c_{\delta,R}>0$. Plugging this in we obtain the result.
\end{proof}
Everything is ready to extract from $\mathcal{J}_1(f)$ a clean control on $\vert \nabla_v \log f\vert^4$ provided $f$ charges 3 non-aligned balls.
\begin{lemma}
    \label{lem:J_1controlsFisher4}
    Recall that $$\mathcal{J}_1(f):=\sum_{k=1}^3\int_{\mathbb{R}^{6}}f(v)f(w)  \frac{\bar{\alpha}}{\vert v - w\vert^2} \left( b_k(v-w)\cdot \nabla_v\log f(v) \right)^4 dvdw.$$
    For any triplet of points $v_1,v_2,v_3 \in B(0,R)$ that are $\delta$-non-aligned,
    $$\mathcal{J}_1(f) \geq c_{\delta,R}\left(\min_{k\in\{1,2,3\} }  \int_{\mathbb{R}^3}h\left(\frac{w-v_k}{\delta}\right)f(w)dw\right)\int_{\mathbb{R}^{3}} f(v) \jap{v}^{\gamma-2}\vert \nabla_v\log f(v)\vert^4,$$
    with $c_{\delta,R}$ from Lemma~\ref{lem:ellipticityestimate}, depending only on $\delta$ and $R$.
\end{lemma}
\begin{proof}
    We write $$\mathcal{J}_1(f):=\sum_{k=1}^3\int_{\mathbb{R}^{3}}\left(\int_{\mathbb{R}^{3}}  \frac{\bar{\alpha}}{\vert v - w\vert^2} \left( b_k(v-w)\cdot \xi(v) \right)^4 f(w)dw \right)f(v)\vert \nabla_v\log f(v) \vert^4 dv,$$
    with $$\xi(v)=\frac{\nabla_v\log f(v) }{\vert \nabla_v\log f(v)\vert},$$ and apply Lemma~\ref{lem:ellipticityestimate} with $\mu=f(dv)$.
\end{proof}
We conclude this section by showing that the finiteness of $\int f\jap{v}^{\gamma-2}\vert \nabla \log f\vert^4$ implies that $f$ is bounded (and continuous), by Sobolev embeddings. This requires $f$ to have finite moments of higher order than the energy.
\begin{lemma}
    \label{lem:fisher4controlsLinf}
    Let $p\in (3,4)$. There exists $c_p>0$ such that, for any $f\in\mathcal{P}(\mathbb{R}^3)$, 
    $$\Vert f\Vert_{L^\infty(\mathbb{R}^3)} \leq c_p\left[1+\left( \int_{\mathbb{R}^{3}}\jap{v}^{\gamma-2}f(v) \left\vert\nabla\log f(v) \right\vert^{4} dv \right)^{\frac{p}{4}}\left(\int_{\mathbb{R}^{3}}\jap{v}^{\frac{p(2-\gamma)}{4-p}}f(v)  dv\right)^{1-\frac{p}{4}}\right],$$
    and $f$ is in fact continuous on $\mathbb{R}^3$ if the right-hand side is finite.
\end{lemma}
\begin{proof}
We assume the right-hand side finite, which implies $f$ has a density.
We use Hölder inequality to get:
\begin{align*}
\int_{\mathbb{R}^{3}}f(v) &\left\vert\nabla\log f(v) \right\vert^{p} dv\\\leq
&\left( \int_{\mathbb{R}^{3}}\jap{v}^{\gamma-2}f(v) \left\vert\nabla\log f(v) \right\vert^{4} dv \right)^{\frac{p}{4}}\left(\int_{\mathbb{R}^{3}}\jap{v}^{\frac{p(2-\gamma)}{4-p}}f(v)  dv\right)^{\frac{4-p}{4}}.
\end{align*}
We rewrite the left-hand-side as a proper Sobolev seminorm:
$$\int_{\mathbb{R}^{3}}f(v) \left\vert\nabla\log f(v) \right\vert^{p} dv=p^p \int_{\mathbb{R}^{3}} \left\vert\nabla f^\frac{1}{p} \right\vert^{p} dv=p^p\Vert \nabla f^\frac{1}{p}\Vert_{L^p}^p.$$
Since $p>3$, the Sobolev space $W^{1,p}(\mathbb{R}^3)$ (of $L^p$ functions having their gradient in $L^p$) embeds into continuous bounded functions and
$$\Vert f^\frac{1}{p}\Vert_{L^\infty}^p\leq c_p \Vert f^\frac{1}{p}\Vert_{W^{1,p}}^p=c_p\left(\Vert f^\frac{1}{p}\Vert^p_{L^p}+\Vert \nabla f^\frac{1}{p}\Vert_{L^p}^p\right).$$
But $\Vert f^\frac{1}{p}\Vert^p_{L^p}=1$ since $f$ is a probability, so 
$$\Vert f\Vert_{L^\infty}=\Vert f^\frac{1}{p}\Vert_{L^\infty}^p \leq c_p\left(1+\Vert \nabla f^\frac{1}{p}\Vert_{L^p}^p\right)$$
and chaining the inequalities above we get the result.
\end{proof}
We nearly have the $L^1_t L^\infty_v$
bound we sought. 
Indeed, we have proven that $\Vert f\Vert_{L^\infty}$ is controlled by $\mathcal{J}_1(f_t)$, which is in turned controlled by $\mathcal{J}(f_t)$ and the Fisher information $I(f_t)$, both of which lie in $L^1([0,T])$ a.s. This holds provided $f_t$ satisfies two conditions: it \textit{charges three non-aligned balls} and \textit{has enough finite moments}. We show that $f_t$ has finite moments in the next section and deal with the bounds related to non-aligned points in the following one.

\subsection{Propagating moments}

The propagation of moments for the Landau equation was shown in \cite{Desvillettes2014a, CarrapatosoDesvillettesHe2015}. The proof in \cite{CarrapatosoDesvillettesHe2015} relies on both the "usual" weak formulation of the Landau equation (the one we used up to now) and the H-formulation introduced by Villani \cite{Villani1998} for H-solutions, and we have yet to show that it holds. For a test function $\varphi\in C^2_b(\mathbb{R}^3)$, the H-formulation is
\begin{align}
    \label{eq:Hform}
     \int_{\mathbb{R}^3} \varphi(v) f_t(v)dv-\int_{\mathbb{R}^3} &\varphi(v) f_0(v) dv\nonumber\\
     = -\frac{1}{2}
     \int_0^t \int_{\mathbb{R}^6}  (\nabla_v&\varphi(v) - \nabla_w\varphi(w))\\
     &\cdot \alpha a(v-w) (\nabla_v - \nabla_w)\log(f_s(v)f_s(w)) f_s(dv) f_s(dw)ds.\nonumber
\end{align}
The right-hand side of this formula only makes sense using the entropy production \cite{Villani1998}: 
by the Cauchy-Schwarz inequality,
\begin{align*}
    \bigg\vert \int_0^t \int_{\mathbb{R}^6} &(\nabla_v\varphi(v) - \nabla_w\varphi(w))\cdot \alpha a(v-w) (\nabla_v - \nabla_w)\log(f_s(v)f_s(w))f_s(v) f_s(w)\ dvdwds
    \bigg\vert \\\leq& \left(\int_0^t \int_{\mathbb{R}^6} \vert \nabla_v\varphi(v) - \nabla_w\varphi(w)\vert^2 \vert v-w\vert^{\gamma+2} f_s(v) f_s(w)\ dvdwds \right)^\frac{1}{2}\\
    &\cdot\left(\int_0^t \int_{\mathbb{R}^6}\alpha a(v-w): \left((\nabla_v - \nabla_w)\log(f_s(v) f_s(w))\right)^{\otimes 2} f_s(v) f_s(w)\ dvdwds \right)^\frac{1}{2},
\end{align*}
but the very last term is exactly (two times) the entropy production $\int_0^t D(f_s \otimes f_s)ds$. The other term can be shown to be finite using the $C^2$ regularity of $\varphi$, see \cite[Equation (42)]{Villani1998} for details. The equivalence between the H-formulation~\eqref{eq:Hform} and the usual weak formulation is simply an integration by parts (which require some care for rough functions, but that can be made rigorous using finiteness of the entropy dissipation). It leads to:
\begin{lemma}
    \label{lem:ashsol}
    Almost surely, the H-formulation~\eqref{eq:Hform} for the random variable $(f_t)_t$ holds for any test function $\varphi\in C^2_0(\mathbb{R}^3)$.
\end{lemma}
\begin{proof}
    Almost surely, the weak formulation holds for $(f_t)_t$.
    Also, from Proposition~\ref{prop:level3estimatesforf}, the entropy production is a.s. in $L^1([0,T])$, that is to say
    $\int_0^T D(f_t \otimes f_t)dt <+\infty.$
    Hence we can make sense of the right-hand side in~\eqref{eq:Hform}, which we call below the (H-formulation with $\alpha$).

    We cut $\alpha=\alpha_1^\epsilon+\alpha_2^\epsilon$ into a small singular part $\alpha_2^\epsilon$ with support in $B(0,\epsilon)$ for some $\epsilon>0$, and a smooth $\alpha_1^ \epsilon$. Since the formulations are linear with respect to $\alpha$, we can write
    \begin{align*}
     \int_{\mathbb{R}^3} &\varphi(v) f_t(v)dv-\int_{\mathbb{R}^3} \varphi(v) f_0(v) dv\\
     &= (\text{weak formulation with }\alpha_1^\varepsilon)+(\text{weak formulation with }\alpha_2^\varepsilon)
    \end{align*}
    and
    \begin{align*}
     (\text{H-formulation with }\alpha)= (\text{H-formulation with }\alpha_1^\varepsilon)+(\text{H-formulation with }\alpha_2^\varepsilon).
    \end{align*}
    By the arguments of \cite[Section 6]{Villani1998} (more precisely, the paragraph with equation (52) therein), the two formulations are equivalent far from the singularity (the idea is to approximate $f(v)f(w)$ by smooth functions, for which the formal integration by parts is rigorous, and we can then pass to the limit \textit{because} the singularity is avoided, see \cite{Villani1998} for details). In other terms, for any $\epsilon>0$, the parts of each formulation with $\alpha_1^\epsilon$ coincide. But the parts with $\alpha_2^\epsilon$ vanish as $\epsilon\rightarrow 0$, by dominated convergence (dominating $\alpha_2^\epsilon$ by $\alpha$), because the full formulations are known to be finite. We conclude that 
    \begin{align*}
     \int_{\mathbb{R}^3} \varphi(v) f_t(v)dv-\int_{\mathbb{R}^3} &\varphi(v) f_0(v) dv\nonumber =(\text{H-formulation with }\alpha)
    \end{align*}
    as wanted.
\end{proof}
\begin{remark}
\label{rem:hsol}
    We emphasize that we have \textit{not} proven that $(f_t)_t$ is a H-solution in the sense of \cite{Villani1998} because it does not a priori satisfy the H-theorem
    $$H(f_t)+\int_0^t D(f_s\otimes f_s)ds \leq H(f_0),$$
    which is part of the definition of H-solutions. The definition also includes that the energy is conserved, which we do not know yet but will be able to show it below. If we knew that $(f_t)_t$ was a H-solution, we could apply the uniqueness result \cite{Tabary2025} and directly conclude that $(f_t)_t=(g_t)_t$. We only know the finiteness of $\int_0^t D(f_s\otimes f_s)ds$ (but with no quantitative bound), which is just enough for making sense of the formulation.

    Also, without finiteness of $\int_0^t D(f_s\otimes f_s)ds$, we cannot a priori use the H-formulation, and the dominated convergence in the proof above does not hold. In other terms, we did \textit{not} show that any weak solution satisfies the H-formulation.
\end{remark}

Knowing that we can rely on both formulations, we can prove the propagation of moments. We begin with the conservation of energy:
\begin{lemma}
    \label{prop:propofenergy}
    Almost surely, for all times $t\in[0,T]$,
    $$\int_{\mathbb{R}^3} \vert v \vert^2f_t(v)dv =E_0.$$
\end{lemma}
\begin{proof}
    \textit{Step 1.} We begin by showing that the energy remains bounded by a Gronwall argument. Let 
    $$E(t)=\int_{\mathbb{R}^3} \vert v \vert^2f_t(v)dv.$$
    By Lemma~\ref{lem:ashsol}, almost surely, $f$ satisfies the H-formulation.
    Consider a smooth radially decreasing $\chi$, such that $0 \leq \chi \leq 1$, and which equals $1$ on $B(0,1)$ and vanishes outside $B(0,2)$. Fix a radius $R>0$ and consider the test function $$\varphi_R(v):= \chi\left(\frac{v}{R}\right)\vert v\vert^2.$$
    We can plug it in the H-formulation.
     We compute
    $$\nabla_v\varphi_R(v) = 2v \chi\left(\frac{v}{R}\right) + \frac{1}{R}\nabla\chi\left(\frac{v}{R}\right) \vert v\vert^2.$$
    For $(v,w) \in B(0,R)^2$, 
    $$\nabla_v\varphi_R(v) - \nabla_w\varphi_R(w)= 2(v-w),$$
    so since $a(v-w)(v-w)=0$, this disappears in the H-formulation~\eqref{eq:Hform} and we are only left with the integral over the complement $(B(0,R)^2)^c$:
    \begin{align}
    \label{eq:proofpropenergy1}
     \int_{\mathbb{R}^3} \varphi_R(v) f_t(v)dv-\int_{\mathbb{R}^3} &\varphi_R(v) f_0(v) dv\nonumber\\
     = -
     \frac{1}{2}\int_0^t &\int_{\mathbb{R}^6}  \mathbf{1}_{(B(0,R)^2)^c} \left( \nabla_v\varphi_R(v) - \nabla_w\varphi_R(w)\right)\\
     &\cdot \alpha a(v-w) (\nabla_v - \nabla_w)\log(f_s(v)f_s(w)) f_s(dv) f_s(dw)ds.\nonumber
    \end{align}
    We have for any $t\in[0,T]$
    $$\int_{\mathbb{R}^3} \varphi_R(v) f_t(v)dv \xrightarrow[R\rightarrow\infty]{} E(t)$$
    by monotone convergence, since $\chi\left(\frac{v}{R}\right)$ is increasing with $R$.

    We now want to control the right-hand side of~\eqref{eq:proofpropenergy1}. By Cauchy-Schwarz and using $ab\leq \frac{1}{2}(a^2+b^2)$ we can bound it by
    \begin{align}
    \label{eq:proofpropenergy2}
     \frac{1}{4}
        \int_0^t \int_{\mathbb{R}^6}\mathbf{1}_{(B(0,R)^2)^c} \vert \nabla_v\varphi_R(v) - \nabla_w\varphi_R(w)\vert^2 \vert v-w\vert^{\gamma+2} &f_s(v) f_s(w)\ dvdwds \nonumber\\
        &+\frac{1}{2}\int_0^T D(f_s\otimes f_s)ds .
    \end{align}
    As we already used in the previous lemma, the second integral is a.s. a finite constant (Proposition~\ref{prop:level3estimatesforf}).
    We differentiate again
    $$\nabla^2_v\varphi_R(v) = 2 \chi\left(\frac{v}{R}\right)\Id +\frac{2}{R}v \otimes \nabla\chi\left(\frac{v}{R}\right) +\frac{2}{R}\nabla\chi\left(\frac{v}{R}\right)\otimes v + \frac{1}{R^2}\nabla^2\chi\left(\frac{v}{R}\right) \vert v\vert^2,$$
    which is bounded uniformly in $v$ by some $C_\chi$ \textit{independent} of $R$, since $\vert v \vert \leq 2R$ on the support of $\chi$. Hence
    $$\left\vert \nabla_v\varphi_R(v) - \nabla_w\varphi_R(w)\right\vert \leq C_\chi \vert v-w\vert.$$
    So
    \begin{align}
    \label{eq:proofpropenergy3}
        \int_0^t \int_{\mathbb{R}^6} \mathbf{1}_{(B(0,R)^2)^c}&\vert \nabla_v\varphi_R(v) - \nabla_w\varphi_R(w)\vert^2 \vert v-w\vert^{\gamma+2} f_s(v) f_s(w)\ dvdwds\nonumber\\
        & \leq C_\chi \int_0^t \int_{\mathbb{R}^6} \vert v-w\vert^{\gamma+4} f_s(v) f_s(w) dvdwds\nonumber\\
        &\leq C_{\chi,\gamma}\int_0^t \int_{\mathbb{R}^6} (1+\vert v \vert^2+\vert w \vert^2) f_s(v) f_s(w)\nonumber\\
        &\leq C_{\chi,\gamma}\int_0^t (1+2E(s))ds.
    \end{align}
    since $0\leq \gamma+4 \leq 1$.
    Hence remembering that~\eqref{eq:proofpropenergy2} bounds the right-hand side of~\eqref{eq:proofpropenergy1} and using the above,
    \begin{align*}
     \int_{\mathbb{R}^3} \varphi_R(v) f_t(v)dv&\leq \int_{\mathbb{R}^3} \varphi_R(v) f_0(v) dv\\
     &+C_{\chi,\gamma} \left[\int_0^t (1+2E(s))ds +  \int_0^T D(f_s\otimes f_s)ds\right]
    \end{align*}
    which yields as $R\rightarrow +\infty$
    \begin{align*}
     E(t)\leq E(0)+C_{\chi,\gamma} \left[2\int_0^t E(s)ds + T + \int_0^T D(f_s\otimes f_s)ds\right].
    \end{align*}
    By Gronwall's lemma $E(t)$ remains bounded on $[0,T]$ by some constant depending on the (random but a.s. finite) quantity $\int_0^T D(f_s\otimes f_s)ds$.

    \textit{Step 2.} Now that we know that $E(t)$ is bounded on $[0,T]$,~\eqref{eq:proofpropenergy3} show that the functions
    $$\Phi_R(s,v,w)=\mathbf{1}_{(B(0,R)^2)^c}( \nabla_v\varphi_R(v) - \nabla_w\varphi_R(w)) \vert v-w\vert^{\frac{\gamma+2}{2}} \sqrt{f_s(v) f_s(w)}$$
    are bounded in $L^2([0,T]\times\mathbb{R}^6)$, uniformly in $R$. But because of the indicator function, $\Phi_R \rightharpoonup 0$ in the sense of distributions, so $\Phi_R\rightharpoonup 0$ weakly in $L^2$.
    As the right hand side of~\eqref{eq:proofpropenergy1} rewrites as
    $$-\frac{1}{2}\int_0^t \int_{\mathbb{R}^6} \Phi_R(s,v,w)
     \cdot \vert v-w\vert^{\frac{\gamma+2}{2}-2} a(v-w) (\nabla_v - \nabla_w)\log(f_s(v)f_s(w)) \sqrt{f_s(v) f_s(w)}dvdwds,$$
    which is the $L^2$ scalar product of $\Phi_R$ against another $L^2$ function (taking the square of this other function yields exactly the integrand of $D(f_t\otimes f_t)$). Hence it vanishes as $R\rightarrow+\infty$, so that $E(t)=E(0)=E_0$.
\end{proof}
\begin{remark}
    To apply Gronwall's lemma, we should in all rigor check that $\int_0^t E(s)ds$ is finite. This is the case almost surely because by Proposition~\ref{prop:lvl3energy} $$\mathbb{E}\left[\int_0^T E(t)dt\right]=\int_0^T\int_{\mathbb{R}^3} \vert v \vert^2 \pi^1_t(dv)\leq T E_0.$$
\end{remark}
\begin{remark}
    A similar proof shows that the H-formulation forces momentum conservation as well.
\end{remark}
We then bound moments of any order, essentially using the result of \cite{CarrapatosoDesvillettesHe2015}:
\begin{prop}\cite[adapted from Lemma 8]{CarrapatosoDesvillettesHe2015}
    \label{prop:propofmoments}
    Let $l\geq 2$ such that $g_0$ has finite moment, \textit{i.e.}
    $$M_l(g_0):=\int_{\mathbb{R}^3} \jap{v}^{l}g_0(v)dv<+\infty$$
    There exists a \textit{random} constant $c_{mom}>0$, a.s. finite, depending on $l,\gamma,T$, the energy $E_0$, and the random entropy production $\int_0^T D(f_t\otimes f_t)dt$, such that
    $$\sup_{t\in[0,T]}\int_{\mathbb{R}^3} \jap{v}^{l}f_t(v)dv < c_{mom}$$
\end{prop}
\begin{proof}
    We apply the proof in \cite[Lemma 8]{CarrapatosoDesvillettesHe2015}, which is for H-solutions. The random $(f_t)_t$ is not a H-solution because we do not have the H-theorem (see Remark~\ref{rem:hsol} above). However, in the proof, the H-theorem is used only to bound $\int_0^T D(f_t\otimes f_t)dt$ by the initial entropy $H_0$. We simply keep a dependence in $\int_0^T D(f_t\otimes f_t)dt$ instead of $H_0$ and obtain the claimed result.
\end{proof}

We have obtained boundedness of moments, which is one of the two conditions necessary to apply our $L^1_t L^\infty_v$ estimate. We deal with the other condition, related to non-aligned points, in the following penultimate section.

\subsection{Propagating that $f$ charges non-aligned points}

To obtain that $(f_t)_t\in L^1_t L^\infty_v$ almost surely, it only remains to show that it charges three non-aligned balls. The strategy is the following: we know that $f_0=g_0$ has finite entropy so by Lemma~\ref{lem:entropyimpliesnonaligned} we can find three suitable $\delta$-non-aligned balls. We can show that for small $t$, $f_t$ must still charge those three balls, by using the H-formulation of the Landau equation, which provides Hölder continuity in time of the integral of $f_t$ against test functions. We can thus obtain the $ L^1_t L^\infty_v$ estimate on a short time $[0,\tau]$. It provides uniqueness of the solution over $[0,\tau]$, so that $f$ is regular over this short time. Its entropy can hence only decrease and we can repeat the argument starting from $f_\tau$ instead of $f_0$. This will lead to uniqueness over $[\tau,2\tau]$, and so on.

We begin this program by proving the Hölder continuity claimed above. Recall that $h$ (defined just before Lemma~\ref{lem:entropyimpliesnonaligned}) is a smooth bump function with support on $B(0,3/2)$ that equals $1$ on $B(0,1)$, and that $\delta$-non-aligned points were introduced in Definition~\ref{def:deltanonaligned}.

\begin{lemma}
    \label{lem:Höldercontof3aligned}
    Let $\bar{v}=(v_1,v_2,v_3)$ be three $\delta$-non-aligned points. The function $\iota_{\bar{v}}$ defined by
    \begin{equation}
    \label{eq:defiota}
        \iota_{\bar{v}}(t):=\min_{k=1,2,3} \int_{\mathbb{R}^3}h\left(\frac{w-v_k}{\delta}\right)f_t(w)dw
    \end{equation}
    is almost surely $\frac{1}{2}$-Hölder continuous,
    with, for any $t,s>0$:
    $$\vert \iota_{\bar{v}}(t) - \iota_{\bar{v}}(s)\vert \leq c_\delta \vert t-s\vert^\frac{1}{2} \left( \int_0^T D(f_u\otimes f_u)du\right)^\frac{1}{2}.$$
    where $c_\delta>0$ depends only on $\delta$ (and $h$) but not on the three points.
\end{lemma}
\begin{proof}
    By Proposition~\ref{prop:level3estimatesforf}, $\int_0^T D(f_u\otimes f_u)du$ is a.s. finite, and by Lemma~\ref{lem:ashsol}, $(f_t)_t$ satisfies the H-formulation. We plug $h(\frac{\cdot-v_k}{\delta})$ in the H-formulation~\eqref{eq:Hform} to get the evolution from $s$ to $t$:
    \begin{align*}
     \int_{\mathbb{R}^3} h\left(\frac{v-v_k}{\delta}\right) f_t(v)dv-\int_{\mathbb{R}^3} &h\left(\frac{v-v_k}{\delta}\right) f_s(v) dv\\
     = -\frac{1}{2}
     \int_s^t \int_{\mathbb{R}^6}  \frac{1}{\delta}&\left(\nabla h\left(\frac{v-v_k}{\delta}\right) - \nabla h\left(\frac{w-v_k}{\delta}\right)\right)\\
     &\cdot \alpha a(v-w) (\nabla_v - \nabla_w)\log(f_u(v)f_u(w)) f_u(dv) f_u(dw)du
    \end{align*}
    Using Cauchy-Schwarz the same way we did before, we get
    \begin{align*}
     \bigg\vert\int_{\mathbb{R}^3} h&\left(\frac{v-v_k}{\delta}\right) f_t(v)dv-\int_{\mathbb{R}^3} h\left(\frac{v-v_k}{\delta}\right) f_s(v) dv \bigg\vert\\
     \leq \frac{1}{2}&\bigg(
     \int_s^t \int_{\mathbb{R}^6}  \frac{1}{\delta}\left\vert\nabla h\left(\frac{v-v_k}{\delta}\right) - \nabla h\left(\frac{w-v_k}{\delta}\right)\right\vert^2\vert v-w\vert ^{\gamma+2} f_u(v) f_u(w)dvdwdu \bigg)^\frac{1}{2}\\
     &\times\left(2\int_s^t D(f_u \otimes f_u)du\right)^\frac{1}{2}.
    \end{align*}
    Using the Hessian of $h$, it is easy to show that
    $$\left\vert\nabla h\left(\frac{v-v_k}{\delta}\right)- \nabla h\left(\frac{w-v_k}{\delta}\right)\right\vert \leq \frac{C_h}{\delta}\vert v-w\vert,$$ so combined with a uniform bound on the gradient
    $$\left\vert\nabla h\left(\frac{v-v_k}{\delta}\right)- \nabla h\left(\frac{w-v_k}{\delta}\right)\right\vert \leq C_h,$$
    we get that
    $$\left\vert\nabla h\left(\frac{v-v_k}{\delta}\right) - \nabla h\left(\frac{w-v_k}{\delta}\right)\right\vert^2\vert v-w\vert ^{\gamma+2} \leq c_\delta$$
    since $-1\leq \gamma+2 \leq0$. Using that $\int_{\mathbb{R}^6} f_u(dv)f_u(dw)=1$, we get
    \begin{align*}
     \bigg\vert\int_{\mathbb{R}^3} h\left(\frac{v-v_k}{\delta}\right) f_t(v)dv&-\int_{\mathbb{R}^3} h\left(\frac{v-v_k}{\delta}\right) f_s(v) dv \bigg\vert
     \leq c_\delta \vert t-s\vert^\frac{1}{2}\left(\int_s^t D(f_u \otimes f_u)du\right)^\frac{1}{2}.
    \end{align*}
    It is easy to see that the minimum of three Hölder-continuous functions with the same Hölder constant is Hölder continuous with the same constant.
\end{proof}

We are now ready to give the proof of Proposition~\ref{prop:asfequalg}, \textit{i.e.} that a.s. $(f_t)_t=(g_t)_t$:
\begin{proof}[Proof of Proposition~\ref{prop:asfequalg}]
Let $\delta,R,\kappa$ be given by Lemma~\ref{lem:entropyimpliesnonaligned} with $H^*=H_0$ and $E^*=E_0$. These three constants will remain fixed for the remainder of the proof.

 Recall that $g_0$ has finite moment of order $m> 3(2-\gamma)$. Let $p\in(3,4)$ such that $l:=\frac{p(2-\gamma)}{4-p}\leq m$. Because $g_0$ has finite moments of order $l$, we can apply the propogation of moments (Lemma~\ref{prop:propofmoments}), to get an a.s. finite $c_{mom}$ such that $\sup_{[0,T]} \int \jap{v}^l f_t(v)dv < c_{mom}.$

By Lemma~\ref{lem:fisher4controlsLinf} (and using $a^{p/4} b^{1-p/4}\lesssim (a + b)$ to simplify the equations),
$$\Vert f_t\Vert_{L^\infty(\mathbb{R}^3)} \leq c_p\left[1+c_{mom} +\int_{\mathbb{R}^{3}}\jap{v}^{\gamma-2}f_t(v) \left\vert\nabla\log f_t(v) \right\vert^{4} dv\right].$$
Chaining with Lemma~\ref{lem:J_1controlsFisher4}, for any triplet $\bar{v}=(v_1,v_2,v_3)$ of $\delta$-non-aligned points,
$$\Vert f_t\Vert_{L^\infty(\mathbb{R}^3)} \leq c_p\left[1+c_{mom} + c_{\delta,R} (\iota_{\bar{v}}(t))^{-1}\mathcal{J}_1(f_t) \right],$$
with $\iota_{\bar{v}}$ defined in Lemma~\ref{lem:Höldercontof3aligned}. Finally, by Lemma~\ref{lem:Jcoerciveerrorcut}, 
\begin{equation}
\label{eq:linffinalcontrol}
\Vert f_t\Vert_{L^\infty(\mathbb{R}^3)} \leq c_p\left[1+c_{mom} + c_{\delta,R} (\iota_{\bar{v}}(t))^{-1}\left(\mathcal{J}(f_t) + 0.99^\gamma \cdot 432I(f_t)\right) \right].
\end{equation}

By Hölder continuity of $\iota_{\bar{v}}$ (Lemma~\ref{lem:Höldercontof3aligned}), there exists $\tau>0$ depending only on $\delta,\kappa$ and $\int_0^T D(f_s\otimes f_s)ds$ such that for any $t_0$, for any $\bar{v}$ triplet of $\delta$-non-aligned point, if $\iota_{\bar{v}}(t_0)\geq \kappa$, then for all $t\in[t_0,t_0+\tau]$, 
$\iota_{\bar{v}}(t)\geq \frac{\kappa}{2}$.

We apply Lemma~\ref{lem:entropyimpliesnonaligned} (with $H^*$ and $E^*$ as above) to the initial condition $\mu=g_0$ (which is a.s. equal to $f_0$). There exists three $\delta$-non-aligned-points $\bar{v}=(v_1,v_2,v_3)\in (B(0,R))^3$ such that
$\iota_{\bar{v}}(0)\geq \kappa$.

With $t_0=0$, we have that, a.s., for all $t\in[0,\tau]$, $(\iota_{\bar{v}}(t))^{-1} \leq \frac{2}{\kappa}$, so that plugging in~\eqref{eq:linffinalcontrol},
$$\Vert f_t\Vert_{L^\infty(\mathbb{R}^3)} \leq c_p\left[1+c_{mom} +  \frac{2c_{\delta,R}}{\kappa}\left(\mathcal{J}(f_t) + 0.99^\gamma \cdot 432I(f_t)\right) \right].$$

The right-hand side lies in $L^1([0,\tau])$ a.s. by Lemma~\ref{lem:Jasfinite} for $\mathcal{J}$ and Proposition~\ref{prop:level3estimatesforf} for $I$. Hence, almost surely,
$(f_t)_{t\in[0,\tau]}\in L^1([0,\tau],L^\infty(\mathbb{R}^3))$ and is a weak solution to the Landau equation. By the uniqueness result \cite[Theorem 1.3]{GuerinFournier2008} for $\gamma\in(-3,-2]$, or \cite[Theorem 2]{Fournier2010} for $\gamma=-3$, $(f_t)_{t\in[0,\tau]}=(g_t)_{t\in[0,\tau]}$.

We can repeat this process because $\tau$ is uniform in time and in the choice of the 3 points. Suppose that $f_t=g_t$ on $[0,n\tau]$ for some integer $n>0$. Since $H(g_{n\tau})\leq H(g_0)=H_0$ because it is a smooth solution (see Remark~\ref{rem:postthm}, and it also holds true for the energy), we can apply Lemma~\ref{lem:entropyimpliesnonaligned} again with \textit{the same} $H^*$ and $E^*$, with $\mu=f_{n\tau}=g_{n\tau}$. It yields three new $\delta$-non-aligned-points $\tilde{v}=(\tilde{v}_1,\tilde{v}_2,\tilde{v}_3)\in (B(0,R))^3$ such that
$\iota_{\tilde{v}}({n\tau})\geq \kappa$, so that $\iota_{\tilde{v}}(t)\geq \frac{\kappa}{2}$ on $[n\tau,(n+1)\tau]$. We conclude once again that
$(f_t)_{t\in[n\tau,(n+1)\tau]}\in L^1([n\tau,(n+1)\tau],L^\infty(\mathbb{R}^3))$, so that by \cite[Theorem 1.3]{GuerinFournier2008}/\cite[Theorem 2]{Fournier2010} again, $(f_t)_{t\in[0,(n+1)\tau]}=(g_t)_{t\in[0,(n+1)\tau]}$. We hence conclude that
$$(f_t)_{t\in[0,T]}=(g_t)_{t\in[0,T]}$$ as desired.
\end{proof}

\begin{remark}
    As explained in Remark~\ref{rem:defwsol}, the weak formulation to which the uniqueness results \cite{GuerinFournier2008,Fournier2010} apply is not symmetrized with respect to the test function as it is in Definition~\ref{def:weaksolution}. However, since $(f_t)_{t\in[0,\tau]}\in L^1([0,\tau],L^\infty(\mathbb{R}^3))$ holds when we apply the theorem, the unsymmetrized version makes sense and is equal to the symmetrized version.
\end{remark}

\section{Concluding the propagation of chaos}

\begin{proof}[Proof of Theorem~\ref{thm:main} and Corollary~\ref{cor:main}]
We summarize the proof of Theorem~\ref{thm:main} and Corollary~\ref{cor:main}: Recall that $$\mu^N_t =\frac{1}{N}\sum_{i=1}^N \delta_{V^N_i(t)} \in \mathcal{P}(\mathbb{R}^3)$$ is the empirical measure of the particle system~\eqref{eq:particlesystem}.

By Proposition~\ref{prop:tightness2}, the family $\left(\mathcal{L}\left( (\mu^N_t)_{t\in[0,T]}\right) \right)_{N\geq 2 }$ is tight in $\mathcal{P} \left(C\left([0,T],\mathcal{P}(\mathbb{R}^3)\right)\right)$.
By Proposition~\ref{prop:asfequalg}, the only cluster point of this sequence is $\delta_{(g_t)_{t\in[0,T]}}$. This means that the whole sequence converges to this cluster point. Since this limit is deterministic, the convergence in law of $(\mu^N_t)_t$ to $(g_t)_{t\in[0,T]}$ can be upgraded to a convergence in probability, hence Theorem~\ref{thm:main}.

By Proposition~\ref{prop:convFNj}, for any $t\in[0,T]$, for any $j\geq 1$, the weak convergence of probabilities
$$F^{N:j}_t\rightharpoonup (g_t)^{\otimes j} $$
holds, since if $\pi_t=\delta_{g_t}\in\mathcal{P}(\mathcal{P}(\mathbb{R}^3))$, the associated marginal is then $\pi_t^j=(g_t)^{\otimes j}$. This is Corollary~\ref{cor:main}.
\end{proof}

\section{Infinite-dimension affinity of generalized Fisher information}
\label{sec:infdimlimofDandK}
In this last section we prove point (3) in Theorem~\ref{thm:side}. The analogous property for the entropy has been known for a long time \cite{RobinsonRuelle1967}, and the case of the usual Fisher information can be found in \cite{HaurayMischler2012,Rougerie2020}. Although this theorem plays a key role in the propagation of chaos, its proof is quite orthogonal to the rest of this work. It also requires a new framework, explaining why we isolated it at the end of this paper. 

\subsection{Preliminaries and an easy inequality}

Most of our proof uses parts of the work by Rougerie \cite{Rougerie2020} as a black box, but we will summarize it for the reader's convenience. Below we first prove one easy half of the theorem. We then introduce the Hilbert framework which will end in us proving an abstract version of the other half in Theorem~\ref{thm:abstractlvl3affinity}. We finally apply the abstract theorem to our functionals to conclude.

We fix $k_0$ and a vector field $b$ satisfying the hypotheses of Theorem~\ref{thm:side}, $\partial=b\cdot\nabla_{1...k_0}$, and $\nu\in\mathcal{P}(\mathcal{P}(\mathbb{R}^3))$ as in point (3). We start by the easiest inequality. We note that we in fact did not use this inequality in this paper.

\begin{lemma}
\label{lem:lvl3affinityfirsthalf}
    It holds that
    $$\lim_{j\geq k_0} I^\partial(\nu^j)\leq\int_{\mathcal{P} (\mathbb{R}^3)} I^\partial(\rho^{\otimes k_0})\nu(d\rho),$$
    $$\lim_{j\geq k_0} K^\partial(\nu^j)\leq\int_{\mathcal{P} (\mathbb{R}^3)} K^\partial(\rho^{\otimes k_0})\nu(d\rho),$$
    and the sequences on the left hand-side are increasing in $j$.
\end{lemma}
\begin{proof}
    The proof is the same for both functionals, we treat the case of $I^\partial$. By super-additivity (Theorem~\ref{thm:side}, point (2)), $(I^\partial(\nu^j))_j$ is increasing in $j$.
    By convexity of $I^\partial$ (Theorem~\ref{thm:side}, point (1)) and Jensen's inequality,
    $$ I^\partial(\nu^j) = I^\partial\left( \int_{\mathcal{P} (\mathbb{R}^3)}\rho^{\otimes j}\nu(d\rho)\right)\leq\int_{\mathcal{P} (\mathbb{R}^3)}I^\partial(\rho^{\otimes j})\nu(d\rho).$$
    But $I^\partial(\rho^{\otimes j})=I^\partial(\rho^{\otimes k_0})$ by a direct computation:
    \begin{align*}
        I^\partial(\rho^{\otimes j})&=\int_{\mathbb{R}^{3j}}\frac{\vert \partial \rho^{\otimes k_0}(v_1,...,v_{k_0})\vert^2\vert \rho^{\otimes (j-k_0)}(v_{k_0+1},...,v_{j})\vert^2}{\rho^{\otimes j}} dv_1...dv_j\\
        &=\int_{\mathbb{R}^{3j}}\frac{\vert \partial \rho^{\otimes k_0}(v_1,...,v_{k_0})\vert^2}{\rho^{\otimes k_0}(v_1,...,v_{k_0})}\rho^{\otimes (j-k_0)}(v_{k_0+1},...,v_{j}) dv_1...dv_j\\
        &=\int_{\mathbb{R}^{3k_0}}\frac{\vert \partial \rho^{\otimes k_0}(v_1,...,v_{k_0})\vert^2}{\rho^{\otimes k_0}(v_1,...,v_{k_0})} dv_1...dv_{k_0}
    \end{align*}
    because $\partial$ only acts on $v_1,...v_{k_0}$.
    
\end{proof}

\subsection{The Hilbertian framework and proof for an abstract operator}

The reverse inequality will require the introduction of an Hilbertian framework, following \cite{Rougerie2020}. We forget the functionals $I^\partial$ and $K^\partial$ for a moment and focus on obtaining an abstract result, Theorem~\ref{thm:abstractlvl3affinity} below.

We still keep a fixed $\nu\in\mathcal{P}(\mathcal{P}(\mathbb{R}^3))$ as in point (3) of Theorem~\ref{thm:side}. Recall that $I(\nu^j)<+\infty$ by hypothesis so that $\nu^j$ admits a density, which we still denote by $\nu^j$. We define
$$\psi^j = \sqrt{\nu^j} \in L^2(\mathbb{R}^{3j})$$
which satisfies $\Vert \psi^j\Vert_{L^2(\mathbb{R}^{3j})}=1$. The reason for introducing the square roots is that our functionals admit nice expressions in terms of it (for instance, recall that the usual Fisher information is $I(\nu^j)=j^{-1}4\int \vert \nabla \psi^j\vert^2$). We let
$\Gamma^j$ be the orthogonal projection on the span of $\psi_j$ in $L^2(\mathbb{R}^{3j})$. It is of course a finite rank operator.

We recall that the trace of an operator $A$ on $L^2(\mathbb{R}^{3j})$ is defined, when the formula makes sense, by
\begin{equation}
\label{eq:deftrace}
    \Tr(A):= \sum_k \langle e_k, Ae_k\rangle
\end{equation}
where $(e_k)_k$ is any orthonormal basis of the separable Hilbert space $L^2(\mathbb{R}^{3j})$. The trace norm is
$$\Vert A \Vert_{\Tr} = \Tr((A^*A)^\frac{1}{2})\in [0,+\infty].$$
It is always well-defined because $(A^*A)^\frac{1}{2}$ is a non-negative operator so the series~\eqref{eq:deftrace} features only non-negative terms. We say that $A$ is trace-class if its trace norm is finite (its trace $\Tr(A)$ is then defined). The trace-class endowed with the trace-norm forms a Banach space and the duality bracket
$$(A,B)\mapsto \Tr(AB)$$
identifies the trace-class as the dual of compact operators, and as the predual of bounded operators.

Since $\Gamma^j$ is of finite rank, it is trace class, and its trace is
$$\Tr(\Gamma^j)=\Vert \Gamma^j \Vert_{\Tr} =\Vert\psi^j\Vert_{L^2}=1.$$
For trace-class operators, the partial trace is the analogue of taking marginals. For $1\leq k \leq j$, the partial trace $\Gamma^{j:k}$ of $\Gamma^j$ is the operator on $L^2(\mathbb{R}^{3k})$ defined by
    $$\Tr(A\Gamma^{j:k})=\Tr((A\otimes \text{Id}^{\otimes(j-k)})\Gamma^j)$$
for any bounded operator $A$ on $L^2(\mathbb{R}^{3k})$.

We have set up the needed definitions. Following \cite{Rougerie2020}, we first build a sequence of operators akin to the sequence of probability measures $(\nu^j)_j$:
\begin{lemma}
    \label{lem:gammak}
    Up to a subsequence in $j$ that we omit to write (and which is independent of $k$), the weak-star convergence
    $$\Gamma^{j:k}\xrightharpoonup[j\rightarrow \infty]{*} \gamma^k,$$
    holds for any $k\geq 1$, where $\gamma_k$ is a trace class operator on $L^2(\mathbb{R}^{3k})$.
\end{lemma}
\begin{proof}
    By the Banach-Alaoglu theorem, the unit ball of the trace-class on $L^2(\mathbb{R}^{3k})$, which is the dual of compact operators, is compact for the weak-$*$ topology, so that from the sequence $(\Gamma^{j:k})_j$ we can extract a converging subsequence. A diagonal argument allows us to make this sequence independent of $k$.
\end{proof}

One can upgrade this convergence and show that the operators $\gamma_k$ are compatible with each other. (This upgrade uses the boundedness of the Fisher information and the second moment condition for confinement, it can thus be seen as some operator version of the Rellich compactness theorem $H^1\subset\subset L^2$.)
\begin{lemma}
\label{lem:gammakstrong}
    The convergence $\Gamma^{j:k}\xrightarrow[j\rightarrow\infty]{} \gamma^k$ holds in trace norm.
    Moreover, for all $k\geq 1$, $\Tr(\gamma^k)=1$, and the sequence $(\gamma_k)_k$ is compatible in the sense that
    $\gamma^{k+1:k}=\gamma^k$. 
\end{lemma}
\begin{proof}
    The proof is done in [Paragraph 3.3, \cite{Rougerie2020}], we summarize it for the sake of self-containedness. For any $i$, we denote by $\Delta_i$ the Laplacian in the $v_i$ variable. Let any $j\geq k$, by symmetry of $\psi^j$,
    \begin{align*}
        I(\nu^j)&=4\int_{\mathbb{R}^{3j}} \vert \nabla_1 \psi^j\vert^2\\
        &=4\int_{\mathbb{R}^{3j}} (-\Delta_1\psi^j) \psi^j \\
        &=4\langle -\Delta_1\psi^j,\psi^j \rangle\\
    \end{align*}
    by integration by parts, denoting by $\langle \cdot,\cdot\rangle$ the $L^2$ scalar product.
    Then, using symmetry again
    \begin{align*}
        \langle -\Delta_1\psi^j,\psi^j \rangle&=\frac{1}{k}\langle\sum_{i=1}^k(-\Delta_i)\psi^j,\psi^j \rangle\\
        &=\frac{1}{k}\Tr \left(\sum_{i=1}^k(-\Delta_i) \Gamma^j\right)\\
        &=\frac{1}{k}\Tr \left(\sum_{i=1}^k(-\Delta_i) \Gamma^{j:k}\right)
    \end{align*}
    by definition of the partial trace. The same manipulation with the energy $\vert \cdot\vert^2$ instead of the Laplacian yields
    $$\Tr \left(\sum_{i=1}^k\left((-\Delta_i) +\vert v_i\vert^2+1\right)\Gamma^{j:k}\right) = \frac{k}{4}I(\nu^j)+k\int_{\mathbb{R}^3} (\vert v \vert^2+1)\nu^1(dv).$$
    By hypothesis the right hand side is bounded uniformly in $j$. Let $L_k=\sum_{i=1}^k\left((-\Delta_i) +\vert v_i\vert^2+1\right)$, it is well known that this operator has compact resolvent. By cyclicity of the trace,
    $$\sup_j \Tr\left(L_k^{1/2}\Gamma^{j:k} L_k^{1/2}\right) <+\infty,$$
    so that we can further extract a weak-$*$ convergent subsequence. By identifying the limit against smooth finite rank operators,
    $$L_k^{1/2}\Gamma^{j:k} L_k^{1/2}\xrightharpoonup[j\rightarrow \infty]{*} L_k^{1/2}\gamma^{k} L_k^{1/2}.$$
    The compactness of $L_k^{-1}$ means we can use it to test the weak-$*$ convergence:
    $$1=\Tr(\Gamma^{j:k})=\Tr\left(L_k^{-1} L_k^{1/2}\Gamma^{j:k} L_k^{1/2}\right)\xrightarrow[j\rightarrow \infty]{}\Tr\left(L_k^{-1} L_k^{1/2}\gamma^{k} L_k^{1/2}\right)=\Tr(\gamma^k)=1,$$
    where cyclicity of the trace has been used twice. Hence the trace norm converges, which implies the strong convergence $\Gamma^{j:k}\rightarrow\gamma^k$ in the trace class.

    The compatibility is an easy consequence: for any bounded operator $A$ on $L^2(\mathbb{R}^{3k})$,
    $$\Tr(A \gamma^{k+1:k})=\Tr((A\otimes \Id)  \gamma^{k+1})=\lim_j \Tr((A\otimes \Id)  \Gamma^{j:k+1})=\lim_j \Tr(A  \Gamma^{j:k})=\Tr(A \gamma^{k}),$$
    the limits being justified because the bounded operators are the dual of the trace-class. This means $\gamma^{k+1:k}=\gamma^{k}$.
\end{proof}
\begin{remark}
    In the above proof, we used the formula
    $$I(\nu^j)=4\Tr(-\Delta_1 \Gamma^j).$$
    Written in terms of $\Gamma^j$, the Fisher information becomes not only convex but linear. This linearity in $\Gamma^j$ is the cornerstone of the proof in \cite{Rougerie2020} of the infinite-dimensional affinity for the Fisher information. Basically, linearity allows us to exchange the Fisher information and the integral in the expression~\eqref{eq:quantum hierarchy} below, with an equality rather than the inequality we obtain when using Jensen's. We will generalize this, replacing $-\Delta_1$ by an abstract operator.
\end{remark}
Recall that the classical de Finetti-Hewitt-Savage Theorem relates the family of compatible marginals $(\nu^j)_j$ to the overarching probability $\nu \in \mathcal{P}(\mathcal{P}(\mathbb{R}^3))$ which makes it able to write $\nu^j$ as a superposition of tensorized probabilites $\rho^{\otimes j}$. Similarly, the \textit{quantum} de Finetti-Hewitt-Savage Theorem provides a probability over the unit sphere $\mathbb{S}(L^2(\mathbb{R}^3))$ which allows us to decompose $\gamma^k$ as tensorized states:
\begin{lemma}
\label{lem:quantum hierarchy}
    There exists $\xi\in \mathcal{P}(\mathbb{S}(L^2(\mathbb{R}^3)))$ such that, for all $k\geq 1$:
    \begin{equation}
        \label{eq:quantum hierarchy}
        \gamma^k =\int_{\mathbb{S}(L^2(\mathbb{R}^3))} p(u^{\otimes k}) \xi(du)
    \end{equation}
    where $p(\phi)$ is the orthogonal projection on the span of $\phi\in \mathbb{S}(L^2(\mathbb{R}^{3k}))$, \textit{i.e.} $p(\phi)=\phi\langle\phi,\cdot\rangle$.
    Moreover, $\xi$ and $\nu$ are connected as follows: if $\Theta:\mathbb{S}(L^2(\mathbb{R}^{3}))\rightarrow \mathcal{P}(\mathbb{R}^3) $ is the continuous map $\phi \mapsto \vert \phi\vert^2$, then
    the pushforward of $\xi$ is $\nu$: $\Theta \# \xi=\nu$.
\end{lemma}
\begin{proof}
    The probability $\xi$ is given by applying the quantum de Finetti-Hewitt-Savage theorem \cite{HudsonMoody1976}, the hypotheses holding by Lemma~\ref{lem:gammakstrong}. The link between $\xi$ and $\nu$ is provided by \cite[Lemma 3.2]{Rougerie2020} after recalling that a Borel probability measure over a metric space is entirely determined by the integral of continuous bounded functions. We reproduce the proof here:
    it suffices to check
    $$\int_{\mathcal{P}(\mathbb{R}^3)} \Phi(\rho)\nu(d\rho) = \int_{\mathbb{S}(L^2(\mathbb{R}^{3}))} \Phi(\Theta(u))\xi(du)$$
    for $\Phi$ of the form $\Phi(\rho)=\int_{\mathbb{R}^{3k}} \rho^{\otimes k} \varphi(v_1,...,v_k)$ for a continuous bounded $\varphi$, because these functions generate a dense sub-algebra of the continuous bounded functions over $\mathcal{P}(\mathbb{R}^3)$ (Those 'monomials' can be found for instance in \cite{MischlerMouhot2012}). We then have, for any $j\geq k$,
    \begin{align*}
        \int_{\mathcal{P}(\mathbb{R}^3)} \Phi(\rho)\nu(d\rho) &= \int_{\mathbb{R}^{3k}} \int_{\mathcal{P}(\mathbb{R}^3)} \rho^{\otimes k} \nu(d\rho)\varphi  \\
        &=\int_{\mathbb{R}^{3k}} \nu^k \varphi\\
        &=\int_{\mathbb{R}^{3j}} \nu^j \varphi\\
        &=\Tr((\varphi\otimes \text{Id}^{\otimes(j-k)} )\Gamma^{j})\\
        &=\Tr(\varphi \Gamma^{j:k})
    \end{align*}
    Letting $j\rightarrow\infty$, the multiplication by $\varphi$ being a bounded operator:
     \begin{align*}
        \int_{\mathcal{P}(\mathbb{R}^3)} \Phi(\rho)\nu(d\rho) &=
        \Tr(\varphi \gamma^{k})\\
        &=\int_{\mathbb{S}(L^2(\mathbb{R}^3))} \Tr(\varphi p(u^{\otimes k})) \xi(du).
    \end{align*}
    We then have
    $$\Tr(\varphi p(u^{\otimes k})) = \int_{\mathbb{R}^{3k}} (\vert u \vert^2)^{\otimes k} \varphi = \Phi(\vert u \vert^2) = \Phi(\Theta(u)). $$
    The proof is concluded.
\end{proof}

\begin{remark}
    We can interpret Lemma~\ref{lem:quantum hierarchy} as some sort of commutation result: Suppose that $\xi$ was exactly the pushforward of $\nu$ by $\mathcal{P}(\mathbb{R}^3) \ni \rho \mapsto \sqrt{\rho} \in \mathbb{S}(L^2(\mathbb{R}^{3}))$. Then the convergence $\Gamma^{j:k}\rightarrow \gamma^k$ would informally mean
    $$\Gamma^j=p\left(\sqrt{\nu^j}\right)=p\left(\sqrt{\int_{\mathcal{P}(\mathbb{R}^3)} \rho^{\otimes j}\nu(d\rho)}\right)\approx_{j\rightarrow\infty} \gamma^j = \int_{\mathcal{P}(\mathbb{R}^3)} p(\sqrt{\rho^{\otimes j}})\nu(d\rho),$$ 
    which is an approximate commutativity between the map
    $p(\sqrt{ \cdot\  })$ and the integration against $\nu(d\rho)$.
\end{remark}

Using this quantum hierarchy, we can prove an abstract version of the desired inequality from Theorem~\ref{thm:side}:

\begin{theorem}
\label{thm:abstractlvl3affinity}
    Let $Q$ be a (possibly unbounded) non-negative self-adjoint operator on $L^2(\mathbb{R}^{3k_0})$ for some $k_0\geq 1$. Assume that for all $u\in\mathbb{S}(L^2(\mathbb{R}^{3}))$ in the support of $\xi$ (given by Lemma~\ref{lem:quantum hierarchy}), $$\langle Q u^{\otimes k_0}, u^{\otimes k_0}\rangle \geq \langle Q \vert u\vert^{\otimes k_0}, \vert u\vert^{\otimes k_0}\rangle.$$
    Then
    $$\limsup_{j\geq k_0}  \langle Q\psi^j,\psi^j\rangle \geq \int_{\mathcal{P}(\mathcal{P}(\mathbb{R}^{3}))} \langle Q \sqrt\rho^{\otimes k_0},\sqrt\rho^{\otimes k_0}\rangle \nu(d\rho).$$
\end{theorem}
\begin{remark}
    Here and below, $\langle Q \psi, \psi\rangle$ should be understood as the quadratic form associated to $Q$, so that $\langle Q \psi, \psi\rangle=\langle Q^\frac{1}{2} \psi, Q^\frac{1}{2}\psi\rangle$ if $\psi$ is in the domain of $Q^\frac{1}{2}$ and $+\infty$ otherwise.
\end{remark}
\begin{remark}
    Taking $Q=-\Delta_1$ (hence $k_0=1$), we recover (half of) the infinite dimension affinity of the Fisher information, since $4\langle -\Delta \sqrt{\rho}, \sqrt{\rho}\rangle=I(\rho)$. The other half is essentially Lemma~\ref{lem:lvl3affinityfirsthalf}.
\end{remark}
\begin{proof}
    We approximate pointwise from below the operator $Q$ by bounded self-adjoint non-negative operators:
    By the spectral theorem for self-adjoint operators, there exists an isometry $U$ between $L^2(\mathbb{R}^{3k_0})$ and another separable Hilbert space $L^2(X,d\mu)$ such that $UQU^{-1}$ is the multiplication operator by some real function $q$. We have $q\geq 0$ $\mu-$a.e. by non-negativity of $Q$. We define approximations by cutting off $q$ to be bounded: for a positive integer $m$,
    $\tilde{Q}_{m}:=U^{-1}\max(q,m)U$, which is a non-negative self-adjoint bounded operator (of norm less than $m$).
    The monotone convergence theorem ensures that $\langle \tilde{Q}_m \psi,\psi\rangle$ is non-decreasing in $m$ and
    $$ \langle \tilde{Q}_m \psi,\psi\rangle \xrightarrow[m\rightarrow +\infty]{} \langle Q \psi, \psi\rangle\in[0,+\infty].$$
    By tensorization with the identity (which we omit to write), we can make $\tilde{Q}_m$ act on $L^2(\mathbb{R}^{3j})$ for any $j\geq k_0$.
    
    We have, for any $j$ and any $m$,
    \begin{align*}
        \langle Q\psi^j,\psi^j\rangle &\geq\langle \tilde{Q}_m\psi^j,\psi^j\rangle\\
        &=\Tr(\tilde{Q}_m\Gamma^j)\\
        &=\Tr(\tilde{Q}_m\Gamma^{j:k_0})\\
    \end{align*}
    Since the operator $\tilde{Q}_m$ is bounded, 
    $$\Tr(\tilde{Q}_m\Gamma^{j:k_0})\xrightarrow[j\rightarrow \infty]{}\Tr(\tilde{Q}_m\gamma^{k_0})$$
    because of the strong convergence from Lemma~\ref{lem:gammakstrong}. But 
    $$\Tr(\tilde{Q}_m\gamma^{k_0})=\int_{\mathbb{S}(L^2(\mathbb{R}^{3}))} \Tr(\tilde{Q}_m p(u^{\otimes k_0})) \xi(du)$$
    by Lemma~\ref{lem:quantum hierarchy}. Finally,
    $$\Tr(\tilde{Q}_m p(u^{\otimes k_0})) \xi(du) = \langle \tilde{Q}_m u^{\otimes k_0},u^{\otimes k_0}\rangle$$ and monotone convergence in $m$ under the integral in the inequality
    $$\limsup_j\langle Q\psi^j,\psi^j\rangle \geq \int_{\mathbb{S}(L^2(\mathbb{R}^{3}))} \langle \tilde{Q}_m u^{\otimes k_0},u^{\otimes k_0}\rangle \xi(du)$$
    yields
    $$\limsup_j \langle Q\psi^j,\psi^j\rangle \geq \int_{\mathbb{S}(L^2(\mathbb{R}^{3}))} \langle Q u^{\otimes k_0},u^{\otimes k_0}\rangle \xi(du).$$
    We then use the hypothesis to write
    \begin{equation}
    \label{eq:usingpospreserv}
        \langle Q u^{\otimes k_0},u^{\otimes k_0}\rangle \geq \langle Q \vert u\vert^{\otimes k_0}, \vert u\vert^{\otimes k_0}\rangle = \langle Q (\sqrt{\vert u\vert^2})^{\otimes k_0}, (\sqrt{\vert u\vert^2})^{\otimes k_0}\rangle.
    \end{equation}
    By Lemma~\ref{lem:quantum hierarchy}, $\nu$ is the pushforward of $\xi$ by $u\mapsto \vert u \vert^2$, so
    $$\int_{\mathbb{S}(L^2(\mathbb{R}^{3}))} \langle Q u^{\otimes k_0},u^{\otimes k_0}\rangle \xi(du)\geq \int_{\mathcal{P}(\mathcal{P}(\mathbb{R}^{3}))} \langle Q \sqrt\rho^{\otimes k_0},\sqrt\rho^{\otimes k_0}\rangle \nu(d\rho).$$
\end{proof}

\subsection{Proof for the first order Fisher information $I^\partial$}

Let us consider again a vector field $b$ as in Theorem~\ref{thm:side}, and $\partial=b\cdot\nabla_{1...k_0}$.
To conclude the proof of point (3) for $I^\partial$, we want to apply Theorem~\ref{thm:abstractlvl3affinity}, so we must write $I^\partial(\nu^j)$ in terms of $\psi_j$. The following integration by parts yields:
\begin{align*}
    I^\partial(\nu^j)&=4\int_{\mathbb{R}^{3j}} \vert \partial  \sqrt{\psi^j}\vert^2\\
    &=-4\int_{\mathbb{R}^{3j}} (\partial^2  \sqrt{\psi^j}) \psi^j\\
    &=4\langle Q\psi^j,\psi^j\rangle
\end{align*}
where $Q=-\partial^2$. To apply Theorem~\ref{thm:abstractlvl3affinity}, we need to show self-adjoint-ness of $Q$. Consider $Q$ acting on the domain of smooth rapidly decaying functions, it is a symmetric non-negative (by integration by parts) densely defined operator on $L^2(\mathbb{R}^6)$. Its Friedrichs extension \cite{AkhiezerGlazman2013, Friedrichs1934} is automatically a self-adjoint operator, and its domain is easily seen to be
$$\left\{ \varphi \in L^2(\mathbb{R}^6) \vert \partial^2\varphi \in L^2(\mathbb{R}^6)  \right\}$$
with $\partial^2\varphi$ understood in the distributional sense. The domain of $Q^{\frac{1}{2}}$ is without surprise
$$\left\{ \varphi \in L^2(\mathbb{R}^6) \vert \partial \varphi \in L^2(\mathbb{R}^6) \right \}.$$
This means that for any $\rho\in \mathcal{P}(\mathbb{R}^6)$, our understanding of $4\langle Q\sqrt{\rho}, \sqrt{\rho}\rangle$ as $4\langle Q^\frac{1}{2}\sqrt{\rho},Q^\frac{1}{2} \sqrt{\rho}\rangle$ precisely coincides with $I^\partial(\rho)$ when $Q$ is understood as the Friedrichs extension: both terms are finite at the same time (whenever $\sqrt{\rho}$ exists and is in the domain of $Q^{\frac{1}{2}}$) and then take the same value.

It remains to check the hypothesis of Theorem~\ref{thm:abstractlvl3affinity}: 

\begin{lemma}
    \label{lem:pospreservingQ}
        For any $g\in L^2(\mathbb{R}^{6})$,
        $$\langle Q g,g\rangle \geq \langle Q \vert g\vert,\vert g\vert\rangle.$$
\end{lemma}
\begin{proof}
    We have
    \begin{align*}
        \langle Q g,g\rangle
        =\int_{\mathbb{R}^{3j}} \left\vert b\cdot\nabla_{1...k_0}g \right\vert^2
        \geq \int_{\mathbb{R}^{3j}} \left\vert b\cdot\nabla_{1...k_0}\vert g \vert\right\vert^2
        =\langle Q \vert g \vert,\vert g \vert\rangle
    \end{align*}
    using that, writing $g=\vert g\vert e^{i\theta}$, we have 
    $$\vert b\cdot\nabla_{1...k_0} g \vert^2 = \vert b\cdot\nabla_{1...k_0}\vert g \vert\vert^2 + \vert g \vert^2 \vert b\cdot\nabla_{1...k_0}\theta\vert^2.$$
\end{proof}

We can now give the
\begin{proof}[Proof of Theorem~\ref{thm:side} (3) for $I^\partial$]
We write that $I^\partial(\nu^j)=4\langle Q\psi^j,\psi^j \rangle$ with $Q=-\partial^2$ and thanks to Lemma~\ref{lem:pospreservingQ}, we can apply Theorem~\ref{thm:abstractlvl3affinity}. We obtain
$$\lim_{j\geq k_0} I^\partial(\nu^j) \geq \int_{\mathcal{P}(\mathcal{P}(\mathbb{R}^{3}))} 4\langle Q \sqrt\rho^{\otimes k_0},\sqrt\rho^{\otimes k_0}\rangle \nu(d\rho).$$
but
$$4\langle Q \sqrt\rho^{\otimes k_0},\sqrt\rho^{\otimes k_0}\rangle=I^\partial(\rho^{\otimes k_0})$$
so we get one half of the desired result. The other inequality holds by Lemma~\ref{lem:lvl3affinityfirsthalf}.
\end{proof}

\subsection{Proof for the second order Fisher information $K^\partial$}

    The proof for the second order functional $K^\partial$ is a little more involved, because $\mathcal{K}^j(\nu^j)$ does not admit such an easy rewriting in terms of a trace involving $\Gamma^j$. We rely on the other functional $K^\partial_{1/2}$ involving a square root introduced in Section~\ref{sec:on the fisher dissipation and aux}. Recall that it is defined by
    $$ K_{1/2}^{\partial}(\nu^j)=4\int_{\mathbb{R}^{3N}}\left( \partial^2\psi^j \right)^2 dV=4\int_{\mathbb{R}^{3N}}( \partial^4\psi^j) \psi^j dV.$$
    We define this time the operator $L=\partial^4$.
    Once again, by integration by parts, $L$ is symmetric nonnegative, hence we can consider its self-adjoint Friedrichs extension (still denoted $L$), and the domain of $L^\frac{1}{2}$ is naturally
    $$\left\{ \varphi \in L^2(\mathbb{R}^6) \vert  \partial^2 \varphi \in L^2(\mathbb{R}^6)\right\},$$
    so that once again $4\langle L \sqrt{\rho},\sqrt{\rho}\rangle=4\langle L^\frac{1}{2}\sqrt{\rho},L^\frac{1}{2}\sqrt{\rho}\rangle$ coincides with $K^\partial_{1/2}(\rho)$ for any $\rho \in \mathcal{P}(\mathbb{R}^6)$.

    We need to show the hypothesis of Theorem~\ref{thm:abstractlvl3affinity} is satisfied, which will require a little more work. Note that we need to restrict ourselves to the $u$ in the support of $\xi$ this time.
    \begin{lemma}
    \label{lem:pospreservingL}
        For any $u\in L^2(\mathbb{R}^{3})$ in the support of $\xi$ (given by Lemma~\ref{lem:quantum hierarchy}),
        $$\langle L u^{\otimes k_0},u^{\otimes k_0}\rangle \geq \langle L \vert u^{\otimes k_0}\vert,\vert u^{\otimes k_0}\vert\rangle.$$
    \end{lemma}
    \begin{proof}
        \textit{Step 1.} In this step we claim that for any $g=u^{\otimes k_0}$ with $u$ in the support of $\xi$, writing $g=\vert g\vert e^{i\theta}$, it holds that
        $$\vert \nabla \theta \vert^2\vert g \vert^2=0.$$
        To do so, we exploit the proof of the infinite-dimensional affinity of the Fisher information $I$. With $\Delta_{1...k_0}$ the Laplacian in the first $k_0$ variables, we have for any $j\geq k_0$,
        $$I(\nu^j)=4\langle -\Delta_{1...k_0} \psi^j,\psi^j\rangle.$$
        It holds that,
        $$\vert \nabla g\vert^2 =\vert \nabla \vert g\vert\vert^2+\vert \nabla \theta \vert^2\vert g \vert^2\geq \vert \nabla \vert g\vert\vert^2,$$
        which allows us, since $\langle -\Delta_{1...k_0} g,g\rangle =\langle \nabla_{1...k_0} g,\nabla_{1...k_0} g\rangle $, to apply Theorem~\ref{thm:abstractlvl3affinity} with $Q=-\Delta_{1...k_0}$ to get
        $$\lim_j I(\nu^j) \geq 4\int_{\mathcal{P}(\mathcal{P}(\mathbb{R}^{3}))} \langle -\Delta_{12} \sqrt\rho^{\otimes 2},\sqrt\rho^{\otimes 2}\rangle \nu(d\rho)=\int_{\mathcal{P}(\mathcal{P}(\mathbb{R}^{3}))} I(\rho) \nu(d\rho).$$
        But the right hand side is greater than $\lim_j I(\nu^j)$ (because the equivalent of Lemma~\ref{lem:lvl3affinityfirsthalf} holds by convexity of $I$). This means that every inequality is an equality: in particular in the proof of Theorem~\ref{thm:abstractlvl3affinity}, we have equality in ~\eqref{eq:usingpospreserv}, so that for any $g=u^{\otimes k_0}$ with $u$ in the support of $\xi$,
        $$\int_{\mathbb{R}^6}\vert \nabla g\vert^2 =\int_{\mathbb{R}^6}\vert \nabla \vert g\vert\vert^2,$$
        which shows $\vert \nabla \theta \vert^2\vert g \vert^2=0.$
        
        \noindent\textit{Step 2.} We compute: since by Step 1
        $$\nabla g = (\nabla\vert g\vert )e^{i\theta}, $$
        we have
        $$\partial g =( \partial\vert g\vert )e^{i\theta}. $$
        Differentiating again,
        $$\partial^2 g = (\partial^2\vert g\vert) e^{i\theta}+i( \partial\vert g\vert )(\partial\theta )e^{i\theta}.$$
        Hence 
        $$\vert\partial^2 g\vert^2\geq\vert\partial^2 \vert g\vert \vert^2. $$
        Since $\langle L g,g\rangle = \int_{\mathbb{R}^6} \vert\partial^2 g\vert^2$, this concludes.
    \end{proof}
    We can now conclude the
    \begin{proof}[Proof of Theorem~\ref{thm:side} (3) for $K^\partial$]
    We apply Theorem~\ref{thm:abstractlvl3affinity} with $Q=L$ and get
    \begin{align*}
        \lim_{j\geq k_0}K^\partial_{1/2}(\nu^j)=\lim_{j\geq k_0}  4\langle L\psi^j,\psi^j\rangle &\geq \int_{\mathcal{P}(\mathcal{P}(\mathbb{R}^{3}))} 4\langle L \sqrt\rho^{\otimes k_0},\sqrt\rho^{\otimes k_0}\rangle \nu(d\rho)\\
        &=\int_{\mathcal{P}(\mathcal{P}(\mathbb{R}^{3}))} K^\partial_{1/2}(\rho^{\otimes k_0}) \nu(d\rho).
    \end{align*}
    But by Proposition~\ref{prop:func_comparison} with $\beta=1/2$, we have the equivalence
    $$K^\partial=K^\partial_1 \geq K^\partial_{1/2}  \geq  \frac{1}{16} K^\partial_{1},$$
    so
    $$\lim_{j\geq k_0}K^\partial(\nu^j)\geq \lim_{j\geq k_0}K_{1/2}^\partial(\nu^j)\geq \int_{\mathcal{P}(\mathcal{P}(\mathbb{R}^{3}))} K^\partial_{1/2}(\rho^{\otimes k_0}) \nu(d\rho)\geq \frac{1}{16}\int_{\mathcal{P}(\mathcal{P}(\mathbb{R}^{3}))} K^\partial(\rho^{\otimes k_0}) \nu(d\rho)$$
    which is the remaining inequality in Theorem~\ref{thm:side}, the other being given by Lemma~\ref{lem:lvl3affinityfirsthalf}.
\end{proof}

\appendix
\section{Derivation of the Kolmogorov equation}
\label{sec:derivkol}
We recall the derivation of the Kolmogorov equation~\eqref{eq:kol} from the particle system~\eqref{eq:particlesystem}. We hence work with $N\geq 2$ fixed.
We rewrite~\eqref{eq:particlesystem} in a more canonical form as
\begin{equation}
\label{eq:particlesystemappendix}
    d\bd{V}(t) = \bd{b}(\bd{V}(t)) dt + \bd{\Sigma}(\bd{V}(t)) d\bd{B}(t).
\end{equation}
Here $\bd{V}(t)=(V^N_1(t),...,V^N_N(t))\in\mathbb{R}^{3N}$, and the drift $\bd{b}(V)\in\mathbb{R}^{3N}$ is defined as
$$\left[\bd{b}(V)\right]_i=\frac{2}{N-1}\sum_{\substack{j=1\\ j\neq i}}^N b^N(v^i - v^j) \in \mathbb{R}^3,$$
recalling that $b^N=\nabla\cdot(\alpha^N a)$. The bracket $[\cdot]_i$ designates the components in the $i$-th $\mathbb{R}^3$ factor in the cartesian product $\mathbb{R}^{3N}=\mathbb{R}^3\times...\times\mathbb{R}^3$).
We also let $\bd{B}(t) \in \mathbb{R}^{3N(N-1)/2}$ be the column vector of the $B^N_{ij}(t)$ in lexicographic order:
$$\bd{B}(t)=(B^N_{12}(t),B^N_{13}(t),...,B^N_{1N}(t),B^N_{23}(t),...,B^N_{(N-1)N}(t))^T$$
and the matrix $\bd{\Sigma}(V)$ with $3N$ lines and $3N(N-1)/2$ columns, given in $3\times3$ blocks as:
\begin{align*}
    \bd{\Sigma} = \bd{\Sigma}(V) = \frac{\sqrt{2}}{\sqrt{N-1}}
    \left[\begin{array}{ccccccccc}
   \sigma^{12}  & ... & \sigma^{1N} & 0 & ... & 0 & 0 & ... & 0\\
   -\sigma^{21} & 0 &  0 & \sigma^{23} &...&\sigma^{2N} & 0 &...& 0 \\
   0 & -\sigma^{31} & 0 & -\sigma^{32} & 0 & 0 & \sigma^{34} & ... & 0 \\
   & & & & ... \\
   0 & 0 & -\sigma^{N1} & 0 & 0 & -\sigma^{N2} &  0 & ... & -\sigma^{N(N-1)}
\end{array}\right]
\end{align*}
with shorthand notation $\sigma^{ij}=\sigma(v^i-v^j)=\sigma(v^j-v^i)$ (recall that $\sigma= \sqrt{\alpha^N a}$ and note that $\sigma^{ij}=\sigma^{ji}$).

Consider a test function $\varphi(V)$, the Itô formula yields:
\begin{align*}
    \frac{d}{dt}\mathbb{E}(\varphi(\bd{V}(t)))&=\mathbb{E}\left[\nabla \varphi(\bd{V}(t)) \cdot\bd{b}(\bd{V}(t)) + \frac{1}{2}\nabla^2 \varphi(\bd{V}(t)): \bd{\Sigma}(\bd{V}(t)) \bd{\Sigma}^T(\bd{V}(t))\right],
\end{align*}
which, rewritten in terms of the law $F^N_t$,
\begin{align*}
    \frac{d}{dt}\int_{\mathbb{R}^{3N}}\varphi(V)F^N_t(dV)&=\int_{\mathbb{R}^{3N}}\nabla \varphi(V) \cdot\bd{b}(V) F^N_t(dV)\\
    &+ \frac{1}{2}\int_{\mathbb{R}^{3N}}\nabla^2 \varphi(V): \bd{\Sigma} (V)\bd{\Sigma}^T(V) F^N_t(dV)\\
    &=\int_{\mathbb{R}^{3N}}\varphi(V) \left[-\nabla \cdot(\bd{b}(V) F^N_t(dV)) + \frac{1}{2}\nabla^2:( \bd{\Sigma} (V)\bd{\Sigma}^T(V) F^N_t(dV))\right]
\end{align*}
after formally integrating by parts, so that the Kolmogorov equation is given by
\begin{equation}
\label{eq:kolgeneral}
    \partial_t F^N_t = - \nabla \cdot (\bd{b} F^N_t) + \frac{1}{2}\nabla^2: (\bd{\Sigma} \bd{\Sigma}^T F^N_t)
\end{equation}
We need to compute the $N\times N$ symmetric matrix $\bd{\Gamma}:=\bd{\Sigma} \bd{\Sigma}^T$:
\begin{align*}
     \bd{\Gamma} = \frac{2}{N-1}
    &\left[\begin{array}{c c c}
   \sum_{k\neq 1} (\alpha^Na)^{1j}  & 0 & 0\\
   0 & \sum_{k\neq 2} (\alpha^Na)^{2j} &  0 \\
   &...& \\
   0 & 0 & \sum_{k\neq N} (\alpha^Na)^{Nj}
\end{array}\right] \\&-\frac{2}{N-1}\
\left[\begin{array}{cccc}
   0  & (\alpha^Na)^{12} & ... & (\alpha^Na)^{1N} \\
   (\alpha^Na)^{21}  & 0 & (\alpha^Na)^{23} & (\alpha^Na)^{2N} \\
   &...& \\
   (\alpha^Na)^{N1}  & ... & (\alpha^Na)^{N(N-1)} & 0
\end{array}\right]
\end{align*}
(with $(\alpha^Na)^{ij}=\alpha^N(\vert v_i-v_j\vert)a(v_i-v_j)$) or in coordinates, for $i,j\in\{1,...,N\}$:
$$[\bd{\Gamma}]_{ij}=\frac{2}{N-1} \left[ \delta_{i=j} \sum_{\substack{k=1\\k\neq i}}^N (\alpha^Na)^{ik} - \delta_{i\neq j}(\alpha^Na)^{ij} \right].$$
We deduce the following relation
\begin{align*}
    [\nabla\cdot\bd{\Gamma}]_{i}&=\frac{2}{N-1}\sum_{j=1}^N \nabla_j\cdot\left[ \delta_{i=j} \sum_{\substack{k=1\\k\neq i}}^N (\alpha^Na)^{ik} - \delta_{i\neq j}(\alpha^Na)^{ij} \right]\\
    &=\frac{2}{N-1}\left(\sum_{\substack{k=1\\k\neq i}}^N b^{N}(v_i-v_k) +\sum_{\substack{j=1\\j\neq i}}^N b^{N}(v_i-v_j) \right)\\
    &=2[\bd{b}(V)]_i.
\end{align*}
Expanding the derivatives in~\eqref{eq:kolgeneral} yields:
\begin{align}
    \partial_t F^N_t =& - \nabla \cdot (\bd{b} F^N_t) + \frac{1}{2}\nabla^2: (\bd{\Gamma} F^N_t)\nonumber\\
    =& - (\nabla \cdot \bd{b})F^N_t -  \bd{b} \cdot \nabla F^N_t + \frac{1}{2} (\nabla^2: \bd{\Gamma}) F^N_t + \frac{1}{2} \bd{\Gamma}:\nabla^2 F^N_t + (\nabla \cdot \bd{\Gamma})\cdot \nabla F_t^N\nonumber\\
    =& - \frac{1}{2}(\nabla^2: \bd{\Gamma})F^N_t -  \frac{1}{2} (\nabla \cdot \bd{\Gamma}) \cdot \nabla F^N_t\nonumber \\
    &+  \frac{1}{2}(\nabla^2: \bd{\Gamma})F^N_t + \frac{1}{2} \bd{\Gamma}:\nabla^2 F^N_t + (\nabla \cdot \bd{\Gamma})\cdot \nabla F_t^N\nonumber \\
    =& \frac{1}{2} (\nabla \cdot \bd{\Gamma}) \cdot \nabla F^N_t + \frac{1}{2} \bd{\Gamma}:\nabla^2 F^N_t \nonumber\\
    =& \frac{1}{2} \nabla \cdot (\bd{\Gamma} \nabla F^N_t)
\end{align}
where we have used the relations $2 \bd{b}=\nabla \cdot \bd{\Gamma}$ and $\nabla\cdot \bd{b}=\frac{1}{2}(\nabla^2: \bd{\Gamma})$ to go from the second line to the third. Notice that the last line provides an explicit expression of the Kolmogorov equation in divergence form.

To obtain~\eqref{eq:kol}, we plug in the expression of $\bd{\Gamma}$, differentiating between diagonal and non diagonal terms:
\begin{align*}
    \partial_t F^N_t &= \frac{1}{(N-1)} \sum_{i,j=1}^{N} \nabla_i \cdot (\bd{\Gamma}_{ij} \nabla_j F^N_t) \\
    &= \frac{1}{(N-1)} \sum_{i=1}^{N} \nabla_i \cdot (\bd{\Gamma}_{ii} \nabla_i F^N_t) + \frac{1}{(N-1)} \sum_{i=1}^{N}\sum_{j \neq i} \nabla_i \cdot (\bd{\Gamma}_{ij} \nabla_j F^N_t) \\
    &= \frac{1}{(N-1)} \sum_{i=1}^{N} \sum_{j \neq i} \nabla_i\cdot  ((\alpha^Na)^{ij} \nabla_i F^N_t) - \frac{1}{(N-1)} \sum_{i=1}^{N}\sum_{j \neq i} \nabla_i\cdot  ((\alpha^Na)^{ij} \nabla_j F^N_t) \\
    &= \frac{1}{(N-1)} \sum_{\substack{i, j = 1 \\ i\neq j}}^N \nabla_i \cdot ((\alpha^Na)^{ij} (\nabla_i F^N_t - \nabla_j F^N_t))
\end{align*}
Symmetrizing the sum, and then indexing on ordered pairs, we get:
\begin{align*}
    \partial_t F^N_t &=\frac{1}{2(N-1)} \sum_{\substack{i, j = 1 \\ i\neq j}}^N (\nabla_i - \nabla_j)\cdot ((\alpha^Na)^{ij} (\nabla_i F^N_t - \nabla_j F^N_t)) \\
    &= \frac{1}{(N-1)} \sum_{\substack{i, j = 1 \\ i<j}}^N (\nabla_i - \nabla_j)\cdot ((\alpha^Na)^{ij} (\nabla_i F^N_t - \nabla_j F^N_t))\\
    &= \frac{1}{(N-1)} \sum_{\substack{i, j = 1 \\ i<j}}^N Q_{ij}^N (F^N_t),\\
\end{align*}
which is~\eqref{eq:kol}.

\sloppy
\printbibliography

\end{document}